\documentclass[10pt]{amsart}

\usepackage{amsmath,amssymb,verbatim,amsthm,hyperref,bbm, mathrsfs,amscd,pb-diagram}
\usepackage{aliascnt} 
\usepackage{appendix}
\DeclareMathAlphabet{\mathpzc}{OT1}{pzc}{m}{it}
\usepackage[all]{xy}
\usepackage{multirow,longtable}
\usepackage{float} 
\usepackage{cleveref}
\usepackage{marginnote}

\usepackage[bordercolor=gray!20,backgroundcolor=blue!10,linecolor=none,textsize=footnotesize,textwidth=0.8in]{todonotes}
\makeatletter
\providecommand\@dotsep{5}
\renewcommand{\listoftodos}[1][\@todonotes@todolistname]{%
  \@starttoc{tdo}{#1}}
\makeatother

\crefname{equation}{equation}{equations}

\theoremstyle{plain}

\newtheorem{Thm}{Theorem}[section]

\newtheorem{Cor}{Corollary}[section]

\newtheorem{Prop}{Proposition}[section]

\newtheorem{Lem}{Lemma}[section]

\theoremstyle{definition}
\newtheorem{Remark}{Remark}[section]

\makeatletter
\newtheorem*{rep@theorem}{\rep@title} \newcommand{\newreptheorem}[2]{%
\newenvironment{rep#1}[1]{%
\def\rep@title{\bf #2 \ref{##1} }%
\begin{rep@theorem} }%
{\end{rep@theorem} } }
\makeatother
\newreptheorem{theorem}{Theorem}
\newreptheorem{lemma}{Lemma}

\numberwithin{equation}{section}
\newcommand{\Card}[1]{\left\vert #1\right\vert} 
\newcommand{\Places}{\mathcal{P}} 
\DeclareMathOperator{\zfun}{\zeta} 
\DeclareMathOperator{\zint}{\mathcal{Z}} 
\DeclareMathOperator{\Lfun}{\mathcal{L}} 
\newcommand{\st}{\operatorname{\mathfrak{st}}} 
\newcommand{\Hecke}{\mathcal{H}} 
\newcommand{\coset}[1]{\left[ #1 \right]}  
\newcommand{\gen}[1]{\left\langle #1 \right\rangle}  
\newcommand{\FNorm}[1]{\left\vert #1 \right\vert} 
\newcommand{\Nm}{\operatorname{Nm}}
\newcommand{\Ker}{\operatorname{Ker}}

\newcommand{\Ind}{\operatorname{Ind}}
\newcommand{\ind}{\operatorname{ind}}
\newcommand{\Hom}{\operatorname{Hom}}
\newcommand{\Gal}{\operatorname{Gal}}
\newcommand{\Aut}{\operatorname{Aut}}

\newcommand{\Tr}{\operatorname{Tr}}
\newcommand{\Res}{\operatorname{Res}}

\newcommand{\C}{\mathbb C}
\newcommand{\A}{\mathbb{A}}

\newcommand{\R}{\mathbb{R}}
\newcommand{\N}{\mathbb{N}}

\newcommand{\M}{\mathcal{M}}
\newcommand{\mO}{\mathcal{O}}

\newcommand{\bk}[1]{\left(#1\right)} 
\newcommand{\bm}{\begin{multline*}}
\newcommand{\tu}{\end  {multline*}}

\DeclareMathOperator{\Id}{\mathbbm{1}} 
\newcommand{\Ga}{\mathbb{G}_a} 
\newcommand{\Gm}{\mathbb{G}_m} 
\newcommand{\Eisen}{\mathcal{E}}
\DeclareMathOperator{\unif}{\varpi} 
\newcommand{\modf}[1]{\mathcal{\delta}_{#1}} 
\newcommand{\Sch}[1]{\operatorname{\mathcal{S}}\left( #1 \right)} 
\newcommand{\meas}{\lambda} 
\newcommand{\less}{<}
\newcommand{\more}{>}
\newcommand{\conv}{*}
\renewcommand{\check}[1]{#1 ^{\vee}} 
\DeclareMathOperator{\Real}{\mathfrak{Re}} 
\newcommand{\piece}[1]{\left\{\begin{matrix} #1 \end{matrix}\right.} 
\newcommand{\set}[1]{\left\{ #1 \right\}} 
\newcommand{\mvert}{\mathrel{}:\mathrel{}} 
\newcommand{\res}[1]{\vert_{#1}}
\newcommand{\suml}{\sum\limits}
\newcommand{\prodl}{\, \prod\limits}
\newcommand{\intl}{\int\limits}
\newcommand{\ldual}[1]{{^L}#1}
\newcommand{\rmod}{/}
\newcommand{\lmod}{\backslash}
\newcommand{\Stab}{\operatorname{Stab}}
\newcommand{\shortexact}[3]{\set{1}\rightarrow #1\rightarrow #2 \rightarrow #3 \rightarrow \set{1}}
\newcommand{\cohom}[1]{\operatorname{H}^{#1}}

\newcommand{\placestimes}{\displaystyle\operatorname*{\otimes}_{\nu\in\Places}}

\makeatletter
\def\imod#1{\allowbreak\mkern10mu({\operator@font mod}\,\,#1)}
\makeatother

\title[New-way integrals for $\Lfun$-functions of cuspidal representations of $G_2$]{A Family of New-way Integrals for the Standard $\mathcal{L}$-function of Cuspidal Representations of the Exceptional Group of type $G_2$}
	\author{Avner Segal}
\address{School of Mathematics, Ben Gurion University of the Negev, POB 653, Be'er Sheva 84105, Israel}
\email{avners@math.bgu.ac.il}

\makeatletter
\renewcommand\section{\@startsection{section}{1}{\z@}%
                                  {-3.5ex \@plus -1ex \@minus-.2ex}%
                                  {2.3ex \@plus.2ex}%
                                  {\center\normalfont\large\bfseries}}
\makeatother

\newcounter{qcounter}

\begin{document}
\begin{abstract}
Let $\Lfun^{S}\bk{s,\pi,\chi,\st}$ be a standard twisted partial $\Lfun$-function of degree $7$ of the cuspidal automorphic representation $\pi$ of the exceptional group of type $G_2$. In this paper we construct a family of Rankin-Selberg integrals representing this $\Lfun$-function.
As an application, we prove that the representations attaining certain prescribed poles are exactly the representations attained by $\theta$-lift from a group of finite type.
\end{abstract}

\maketitle

\begin{center}
\end{center}

\tableofcontents



\section{Introduction}

$\Lfun$-functions are one of the central objects of modern number theory, representation theory, the theory of automorphic forms and other areas of mathematics.
They are used in order to achieve algebraic and arithmetic information by means of analytic inquiries.
The use of $\Lfun$-functions in number theory goes back to Dirichlet who used them in 1837 \cite{Dirichlet} to prove  that any arithmetic progression, with coprime coefficients, contains infinitely many prime numbers.


Let $G$ be a reductive group defined over a number field $F$.
Given a non-Archimedean place $\nu$ of a number field $F$ and an unramified representation $\pi_\nu$ of $G\bk{F_\nu}$, there is a semisimple conjugacy class $t_{\pi_\nu}\in\ldual{G}$, called the Satake parameter of $\pi_\nu$.
For a finite dimensional representation $\rho$ of $\ldual{G}$ the local $\Lfun$-factor is defined by
\[
\Lfun\bk{s,\pi_\nu,\rho}=\frac{1}{\det\bk{\Id-\rho\bk{t_{\pi_\nu}}q_\nu^{-s}}} ,
\]
where $q_\nu$ is the cardinality of the residue field of $F_\nu$.

Given an irreducible automorphic representation $\pi=\otimes_{\nu}\pi_\nu$ of $G\bk{\A_F}$ and a finite set of places $S$, such that $\pi_\nu$ is unramified for $\nu\notin{S}$, the global partial $\Lfun$-function is defined by
\[
\Lfun^S\bk{s,\pi,\rho}=\prod_{\nu\notin{S}} \Lfun\bk{s,\pi_\nu,\rho} \ .
\]
Langlands conjectured that this product converges for $Re(s) >> 0$ and admits a meromorphic continuation to the whole complex plane.
The most effective way known to prove this conjecture, for various cases, is by attaching to $\Lfun^S\bk{s,\pi,\rho}$ an integral representation with convenient analytic properties.
One method of doing this is the Rankin-Selberg method.
For a comprehensive survey of the Rankin-Selberg method consult \cite{MR2192819}.



Let $\pi$ be an irreducible cuspidal representation of the exceptional group $G_2\bk{\A_F}$ and let $\chi$ be a Hecke character of $GL_1\bk{\A_F}$.
The Langlands dual group of $G_2$ is isomorphic to $G_2\bk{\C}$ and the Langlands dual group of $GL_1$ is isomorphic to $\C^\times$.
We denote by $\st$ the irreducible seven-dimensional complex representation of $G_2\bk{\C}\times \C^\times$.
Our {\bf main objective} is to attach an integral representation for the partial $\Lfun$-function $\Lfun\bk{s,\pi,\chi,\st} = \Lfun\bk{s,\pi\boxtimes\chi,\st}$ and to prove its meromorphic continuation.

In this paper we consider a family of zeta integrals $\zint_E\bk{\chi,s,\varphi,f}$ parametrized by \'etale cubic algebras $E$ over $F$.
\'Etale cubic algebras over $F$ also parametrize the non-degenerate Fourier coefficients along the Heisenberg unipotent subgroup $U$ of $G_2$.
For any such algebra $E$ we denote by $\Psi_E$ the corresponding non-degenerate complex character of $U$.
Isomorphism classes of quasi-split forms of $D_4$ are also parametrized by \'etale cubic algebras over $F$.
For any such $E$ we denote by $H_E$ the corresponding quasi-split form of $D_4$.

In \cite[Theorem 3.1]{MR2181091} it is shown that any cuspidal representation of $G_2$ supports at least one non-degenerate Fourier coefficient along this unipotent subgroup.
The integral $\zint_E\bk{\chi,s,\varphi,f}$ involves a normalized degenerate Eisenstein series $\Eisen_E^\ast\bk{\chi,s,f,g}$ associated to the induced representation from the Heisenberg parabolic subgroup $P_E$ of $H_E$.

The main result of this paper is

\begin{reptheorem}{Thm:Main theorem}
\it
Let $\pi$ be an irreducible cuspidal representation of $G_2\bk{\A_F}$ supporting the Fourier coefficient corresponding to $E$ and let $\varphi=\displaystyle \bigotimes_{\nu\in\Places} \varphi_\nu\in\pi$ and $f_s= \displaystyle \bigotimes_{\nu\in\Places} f_{\nu}\in I_{P_E}\bk{\chi,s}$ be factorizable data.
Fix a finite set of places $S\subset\Places$ so that for $\nu\notin S$ all data is unramified.
Let
\begin{equation}
\zint_E \bk{\chi, s, \varphi, f} = \intl_{G_2\bk{F}\lmod G_2\bk{\A_F}} \varphi\bk{g} \Eisen_E^\ast\bk{\chi,s,f,g} dg
\end{equation}
then
\begin{equation}
\zint_E \bk{\chi, s, \varphi, f} = \Lfun^S\bk{s+\frac{1}{2},\pi,\chi,\st} d_S\bk{\chi, s,\Psi_E,\varphi_S,f_S} .
\end{equation}
Moreover, for any $s_0$ there exist vectors $\varphi_S$, $f_S$ such that $d_S\bk{\chi, s,\Psi_E,\varphi_S,f_S}$ is analytic in a neighborhood of $s_0$ and $d_S\bk{\chi, s_0,\Psi_E,\varphi_S,f_S}\neq 0$.

In particular, the family of twisted partial $\Lfun$-function $\Lfun^S\bk{s,\pi,\chi,\st}$ admits a meromorphic continuation to the whole complex plane.
\end{reptheorem}

This result generalizes \cite[theorem 3.1]{MR3284482} where we constructed an integral representation for the standard (untwisted) $\Lfun$-function supporting the Fourier coefficient associated with the split cubic algebra over $F$.


The foremost challenge in the proof of \cref{Thm:Main theorem} is that in the unfolded integral appears a non-unique model, namely a functional from $\Hom_{U\bk{\A_F}}\bk{\pi,\C_{\Psi_E}}$ which for many $\pi$-s is not a one dimensional space.
Usually such integrals, called \emph{new way integrals} after \cite{MR965059}, are not factorizable.
In \cite{MR965059} I. Piatetski-Shapiro and S. Rallis suggested a method to prove factorizability of such integrals.
The remarkable mechanism suggested there relies on the fact that the local integral at an unramified place $\nu$ will be equal to the local $\Lfun$-factor {\bf for any} functional from $\Hom_{U\bk{F_\nu}}\bk{\pi_\nu,\C_{\Psi_E}}$.

As a byproduct of the proof of \cref{Thm:Main theorem} we provide a parametrization of the $G_2\bk{F}$ orbits in $P_E\bk{F}\lmod H_E\bk{F}$.

The proof of \cref{Thm:Main theorem} is mostly unified for all \'etale cubic algebras $E$ over $F$ and the proof splits up for the different kinds of cubic algebras only at a very late stage, namely \cref{Subsec: Convolution} and \cref{Sec: Fourier Transform}, where the arithmetical difference between the different \'etale algebras become critical.

In the proof of the local unramified identity we use the same strategy as in \cite{MR3284482} and use an approximation to the generating function of the local $\Lfun$-factor instead of the generating function itself.
In \cref{Sec:Unramified Computation} we strengthen the results of \cite{MR3284482} concerning this approximation.

As an application of the main theorem we deduce the following result characterizing theta lifts from groups of finite type as the representations having a certain pole.
Let $E$ be an \'etale cubic algebra which is not a non-Galois field extension. Let $\chi_E$ be the character attached to $E\rmod F$ by class field theory and let $S_E=\Aut_F\bk{E}$.
We denote by $\theta_{S_E}$ the theta correspondence for the dual pair $S_E\times G_2$ in $H_E\rtimes S_E$.

\begin{reptheorem}{Conj: Theta lift}
\it
The following are equivalent
\begin{enumerate}
\item
$\Lfun^S\bk{s,\pi,\chi_E,\st}$ admits a pole  at $s=2$ of order $2$ if $E=F\times F\times F$ and of order $1$ otherwise.
\item
$\theta_{S_E}(\pi)\neq 0$.
\end{enumerate}

\end{reptheorem}

Other Rankin-Selberg integrals representing the standard $\Lfun$-function of cuspidal representations of $G_2$ were introduced in \cite{MR1203229} for generic representations and in \cite{UnPublishedGinzburg} for any cuspidal representation of $G_2$.
The later is done using a doubling construction showing that the set of poles of $\Lfun\bk{s,\pi,\chi,\st}$ are contained in the set of poles of a degenerate Eisenstein series of the exceptional group of type $E_8$.
D. Ginzburg and J. Hundley further conjectured \cite[Conjecture 1]{UnPublishedGinzburg} for $\Real\bk{s}\geq 0$ the orders of the poles are bounded by $2$.

{\bf Acknowledgments.} 
I would first want to thank Dihua Jiang for suggesting the family of zeta integral for the L-functions.
I would like to thank Wee Teck Gan for suggesting the approach for the proof of \cref{Thm: G2-orbits}.
It is of great pleasure to thank Nadya Gurevich for her guidance, for endless helpful discussions and mostly for introducing me into the wonderful world of automorphic forms.

This work constitutes of parts from the Ph.D. thesis of the author.

The author was partially supported by grant 1691/10 from the Israel Science Foundation.

\section{Preliminaries}
\label{Sec: Preliminaries}
Let $F$ be a number field and let $\Places$ be its set of places.
For any $\nu\in\Places$ we denote by $F_\nu$ the completion of $F$ at $\nu$.
If $\nu<\infty$ we denote by $\mO_\nu$ its ring of integers, by $\unif_\nu$ a uniformizer of $F_\nu$ and by $q_\nu$ the cardinality of the residue field of $F_\nu$.
We also denote by $\A_F=\A$ the ring of adeles of $F$.

\subsection{The Group $G_2$}
Let $G$ be the simple split reductive adjoint and simply connected group of type $G_2$ defined over $F$.
Let $B$ be a Borel subgroup of $G$ and $T$ a maximal torus in $B$.
Let $\alpha$ and $\beta$ be the short and long simple roots of $G$ with respect to $\bk{B,T}$.
Let $W$ be the Weyl group of $G$ with respect to $\bk{B,T}$.
The Dynkin diagram of $G$ is
\[
\xygraph{
!{<0cm,0cm>;<0cm,1cm>:<1cm,0cm>::}
!{(0.4,-1)}*{\alpha}="label1"
!{(0,-1)}*{\bigcirc}="1"
!{(0,-0.1)}="c"
!{(0.2,  0.1)}="c1"
!{(-0.2,0.1)}="c2"
!{(0.4,1)}*{\beta}="label2"
!{(0,1)}*{\bigcirc}="2"
"1"-@3"2" "c1"-"c" "c"-"c2" 
} .
\]
The set of positive roots of $G$ is
\[
\Phi^{+} = \set{\alpha, \beta, \alpha+\beta, 2\alpha+\beta, 3\alpha+\beta, 3\alpha+2\beta} .
\]
The fundamental weights of $G$ are denoted by
\[
\omega_1=2\alpha+\beta,\quad \omega_2=3\alpha+2\beta .
\]
For any simple root $\gamma$ let $w_\gamma\in W$ be the simple reflection with respect to it.
For any root $\gamma$ we fix a one-parametric subgroup $x_\gamma:\Ga\to G$.
Also, let $h_\gamma:\Gm\to T$ be the coroot subgroup such that for any root $\epsilon$
\[
\epsilon\bk{h_\gamma\bk{t}} = t^{\gen{\epsilon,\check{\gamma}}} .
\]
The group $G$ contains an Heisenberg maximal parabolic subgroup $P=M\cdot U$.
The Levi subgroup $M$ is isomorphic to $GL_2$ and is generated by the simple root $\alpha$, while $U$ is a five-dimensional Heisenberg group.
Finally, we let $\st:G\hookrightarrow GL_7$ be the standard 7-dimensional embedding.

\subsection{Twisted Partial $\Lfun$-functions}
The dual Langlands group $\ldual{G}$ of $G$ is isomorphic to $G_2\bk{\C}$.

Let $\pi=\placestimes \pi_\nu$ be an irreducible cuspidal representation of $G\bk{\A}$ and let $\chi=\placestimes\chi_\nu:F^\times\lmod \A^\times\to \C^\times$ be a Hecke character, both unramified outside of a finite subset $S\subset\Places$.
For $\nu\notin S$ we denote its Satake parameter by $t_{\pi_\nu}$.
We let
\[
\Lfun^S\bk{s,\pi,\chi,\st} = \prodl_{\nu\notin S} \frac{1}{\det\bk{I-\st\bk{t_{\pi_\nu}} \chi\bk{\unif_\nu} q_\nu^{-s} }} .
\]
This product converges for $\Real\bk{s}>>0$ to an analytic function.
In this paper we prove that $\Lfun^S\bk{s,\pi,\chi,\st}$ admits meromorphic continuation to the whole complex plane.


\begin{Remark}
The dual Langlands group $\ldual{GL_1}$ of $GL_1$ is isomorphic to $GL_1\bk{\C}=\C^\times$.
By abuse of notations by $\st$ we also denote the irreducible seven-dimensional complex representation of $G_2\bk{\C}\times \C^\times$.
Since for $\nu\notin S$ the Satake parameter of $\chi_\nu$ is $\chi_\nu\bk{\unif_\nu}\in \C^\times$ it holds that
\[
\Lfun^S\bk{s,\pi,\chi,\st} = \prodl_{\nu\notin S} \frac{1}{\det\bk{I-\st\bk{t_{\pi_\nu}, \chi\bk{\unif_\nu}} q_\nu^{-s} }} .
\]
\end{Remark}

\subsection{\'Etale Cubic Algebras Over $F$}
Let $E$ be an \'etale cubic algebra over $F$, then $E$ is one of the following:
\begin{enumerate}
\item\underline{$E=F\times F\times F$:}
This is called the split cubic algebra over $F$.
In this case, $\Aut_F\bk{E}\equiv S_3$.
For $\bk{a,b,c}\in F\times F\times F$ let $\Nm_{E\rmod F}\bk{a,b,c}=abc$.

\item\underline{$E=F\times K$:}
Here $K$ is a quadratic (and hence Galois) extension of $F$.
Furthermore, $\Aut_F\bk{E}=\Gal\bk{K\rmod F}=\set{1,\sigma}$.
For $\bk{a,b}\in F\times K$ let $\Nm_{E\rmod F}\bk{a,b}= a\Nm_{K\rmod F}\bk{b} = abb^\sigma$.

\item\underline{$E$ is a Cubic Galois Field Extension:}
Here we assume that $E$ is a cubic Galois extension of $F$.
In this case, $\Aut_F\bk{E}=\Gal\bk{E\rmod F}=\set{1,\sigma,\sigma^2}$
For $a\in E$ let $\Nm_{E\rmod F}\bk{a}= aa^\sigma a^{\sigma^2}$.

\item \underline{$E$ is a Cubic non-Galois Field Extension:}
In this case, let $L$ be the Galois closure of $E$ over $F$, this is a Sextic Galois extension with $\Gal\bk{L\rmod F}=\gen{\sigma,\ \tau\mvert \sigma^3=1,\ \tau^2=1}$.
Note that $L$ is also a Galois extension of $E$.
We achieve the following tower of extensions
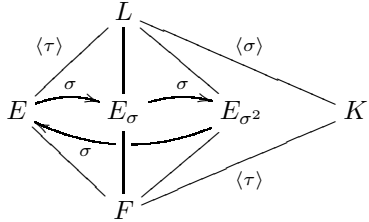
\[
\xymatrix{
&L \ar@{-}[ld]_{\gen{\tau}} \ar@{-}[d] \ar@{-}[rd] \ar@{-}[rrd]^{\gen{\sigma}} \\
E \ar@{-}[rd] \ar@/^/[r]^{\sigma}& E_{\sigma} \ar@{-}[d] \ar@/^/[r]^{\sigma} & E_{\sigma^2} \ar@{-}[ld] \ar@/^1pc/[ll]^-(0.7){\sigma}|!{[l];[dl]}\hole & K \ar@{-}[lld]^{\gen{\tau}} \\
& F }
\]
Where $K=L^{\gen{\sigma}}$ and $E$, $E_\sigma=L^{\gen{\sigma\tau\sigma^2}}$ and $E_{\sigma^2}=L^{\gen{\sigma^2\tau\sigma}}$ are the $\sigma$-conjugates of $E$ in $L$.
For $a\in E$ let $\Nm_{E\rmod F}\bk{a}= aa^\sigma a^{\sigma^2}$.
\end{enumerate}

\begin{Remark}
We call the first three types \textbf{Galois \'etale cubic algebras over $F$}.
\end{Remark}

For reasons that will become clear later in this section, we seek a certain realization of $E$ as a quotient of the polynomial ring in two variables $F\coset{x,y}$.
Let $E$ be an \'etale cubic algebra, then one can choose a homogeneous cubic polynomial $p_E\in F\coset{x,y}$ so that $F\coset{x}\rmod \bk{p_E\bk{x,1}}\cong E$.
Note that $p_E$ splits in a finite extension of $F$ into
\[
p_E\bk{x,y} = \bk{x-ay} \bk{x-by} \bk{x-cy} ,
\]
So one may choose $p_E$ so that
\[
p_E\bk{x,y}=x^3-T_{\bk{a,b,c}}x^2y+D_{\bk{a,b,c}}xy^2-N_{\bk{a,b,c}}y^3 ,
\]
where  $\bk{a,b,c}\in \overline{F}\times \overline{F}\times \overline{F}$ and
\begin{itemize}
\item $T_{\bk{a,b,c}}=a+b+c\in F$.
\item $D_{\bk{a,b,c}}=ab+bc+ca\in F$.
\item $N_{\bk{a,b,c}}=abc\in F$.
\end{itemize}

According to \cite[Proposition 2.2]{MR776169}, for any $E$ one can choose $\bk{a,b,c}$ satisfying the following condition:
\begin{list}{\underline{\textbf{(CT)}:}}{\usecounter{qcounter}} 
\item
$\bk{a,b,c}$ is one of the following:
\begin{itemize}

\item \underline{$E=F\times F\times F$:} $\bk{a,b,c}=\bk{1,-1,0}$.

\item \underline{$E=F\times K$, $K$ is a field:} Choose $\theta\in K$ such that $K=F\coset{\theta}$ and $\theta+\theta^\sigma=0$. We choose $\bk{a,b,c}=\bk{0,\theta,\theta^\sigma}\in K\times K\times K$.

\item \underline{$E$ is a field:} Choose $\theta\in E$ such that $E=F\coset{\theta}$ and $\theta+\theta^\sigma+\theta^{\sigma^2}=0$. We choose $\bk{a,b,c}=\bk{\theta,\theta^\sigma,\theta^{\sigma^2}}\in E\times E\times E$.
\end{itemize}

\end{list}
Here \textbf{CT} stands for \emph{cubic algebra triplets}.
We henceforth assume that $\bk{a,b,c}$ satisfy {\bf (CT)}.
In particular, we assume that $T_{\bk{a,b,c}}=0$.

\subsection{Characters of $U$}
In this subsection, we will construct for any $E$ a character $\Psi_E\in \Hom\bk{U\bk{\A},\C^\times}$.
We parametrize the elements of $U$ by
\[
u\bk{r_1,r_2,r_3,r_4,r_5} := x_{\beta}\bk{r_1} x_{\alpha+\beta}\bk{r_2} x_{2\alpha+\beta}\bk{r_3} x_{3\alpha+\beta}\bk{r_4} x_{3\alpha+2\beta}\bk{r_5} .
\]
The natural action of $M$ on $U$ induces an action on $\Hom\bk{U,\Ga}$.
It is shown in \cite{MR1637485} that for any field $L$ of characteristic 0, the $M\bk{L}$-orbits in $\Hom\bk{U\bk{L},L}$ are naturally parametrized by isomorphism classes of cubic $L$-algebras.

For the rest of this paper, fix a non-trivial additive complex unitary character $\psi=\placestimes \psi_\nu$ of $F\lmod \A$.
This gives rise to the correspondence
\[
\set{\begin{matrix} \text{Isomorphism classes of} \\ \text{cubic algebras over $F$} \end{matrix}} 
\longleftrightarrow \Hom\bk{U\bk{\A},\C^\times}\rmod M\bk{F} .
\]
In particular, we call $\Psi\in \Hom\bk{U\bk{\A},\C^\times}\rmod M\bk{F}$ \emph{non-degenerate} if it corresponds to an \'etale cubic algebra over $F$.

For any $E$ we will now attach a representative $\Psi_E$ of the corresponding $M\bk{F}$-orbit in $\Hom\bk{U\bk{\A},\C^\times}$.
We let
\begin{equation}
\Psi_E\bk{u\bk{r_1,r_2,r_3,r_4,r_5}} := \psi\bk{r_4-T_{\bk{a,b,c}}r_3+D_{\bk{a,b,c}}r_2-N_{\bk{a,b,c}}r_1} .
\end{equation}






In \cite{MR3284482} we use a distinguished representative of the class of complex characters associated with the split cubic algebra over $F$, we denote it by $\Psi_s$.
In \cref{Section: G2-orbits} we show the relation between $\Psi_s$ and $\Psi_{F\times F\times F}$.

\subsection{Fourier Coefficients and Wave Front of a Representation}
We denote by $\mathcal{A}\bk{G}$ the space of automorphic forms on $G\bk{\A}$.
For any $\varphi\in\mathcal{A}\bk{G}$ and any $\Psi\in \Hom\bk{U\bk{\A},\C^\times}$ we let
\[
L_\Psi\bk{\varphi}\bk{g} = \intl_{U\bk{F}\lmod U\bk{\A}} \varphi\bk{ug} \overline{\Psi\bk{u}} du .
\]
For any $g\in G\bk{\A}$ this defines a functional $L_\Psi\bk{\cdot}\bk{g}\in \Hom_{U\bk{\A}}\bk{\mathcal{A}\bk{G},\C_\Psi}$.
For an automorphic representation $\pi$ in $\mathcal{A}\bk{G}$, we say that $\pi$ supports the $\bk{U,\Psi}$-Fourier coefficient if there exist $\varphi\in\pi$ so that $L_\Psi\bk{\varphi}\not\equiv 0$.
It was shown in \cite{MR2181091} that for any cuspidal representation $\pi$ in $\mathcal{A}\bk{G}$, there exists an \'etale cubic algebra $E$ so that $\pi$ supports the Fourier coefficient corresponding to $E$.
Conversely, it is shown in \cite{MR1918673} that for any \'etale cubic algebra $E$ there exists a cuspidal representation $\pi$ that supports the $\bk{U,\Psi_E}$-Fourier coefficient.

\subsection{Local Fourier Transform}
For a finite $\nu\in\Places$, let $K_\nu$ denote the maximal compact subgroup $G\bk{\mO_\nu}$ of $G\bk{F_\nu}$.
Given a complex character $\Psi$ of $U\bk{F_\nu}$ let
\[
\M_\Psi = \set{f:G\bk{F_\nu}\to \C\mvert f\bk{ugk} = \overline{\Psi\bk{u}}f\bk{g}\ \forall u\in U\bk{F_\nu},\ k\in K_\nu} .
\]

We let $\Hecke_\nu=\Hecke\bk{G\bk{F_\nu},K_\nu}$ denote the spherical Hecke algebra of $G\bk{F_\nu}$ with respect to $K_\nu$.
For $f\in \Hecke_\nu$ define its $\Psi$-Fourier transform $f^\Psi$ by
\[
f^\Psi\bk{g} = \intl_{U\bk{F_\nu}} f\bk{ug} \Psi\bk{u} du .
\]
Obviously $f^\Psi\in \M_\Psi$.

\begin{Remark}
For any $f\in\M_{\Psi}$, $f$ is determined by its values on $M\bk{F}\cap B\bk{F}$.
For any $g=h_\alpha(t_1)h_\beta(t_2) x_{\alpha}\bk{d}\in M\bk{F}\cap B\bk{F}$ we may also write $g=x_{\alpha}\bk{p} h_\alpha(t_1)h_\beta(t_2)$ with $p=\frac{d t_1^2}{t_2}$. 
The two different presentations will prove useful later.
\end{Remark}

The following lemma will be useful while evaluating functions in $\M_{\Psi_E}$, the proof is analogues to \cite[Lemma A.1]{MR3284482}.

\begin{Lem}
\label{Lem: Conditions on m}
Let $f\in \M_{\Psi_E}$, then for $f\bk{h_\alpha(t_1)h_\beta(t_2) x_{\alpha}\bk{d}}=0$ unless the following holds:
\[
N_E\frac{t_2^2}{t_1^3}+D_E\frac{dt_2}{t_1}+\frac{d^3t_1^3}{t_2}, D_E\frac{t_2}{t_1}+\frac{3d^2t_1^3}{2t_2}, \frac{3dt_1^3}{t_2}, \frac{t_1^3}{t_2}\in\mO .
\]

\end{Lem}

%
%
%
%



\subsection{Quasi-Split Forms of $D_4$}
We recall the following parametrization of quasi-split forms of $D_4$ over $F$:
\[
\set{\text{Quasi-split forms of $D_4$ over $F$}} \longleftrightarrow \set{\varphi:\Gal\bk{\bar{F}\rmod F}\to\Aut\bk{Dyn\bk{D_4}}},
\]
where $Dyn\bk{D_4}$ is the Dynkin diagram of type $D_4$.
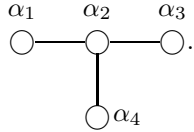
\[
\xygraph{
!{<0cm,0cm>;<0cm,1cm>:<1cm,0cm>::}
!{(0,-1)}*{\bigcirc}="1"
!{(0.4,-1)}*{\alpha_1}="label1"
!{(0,0)}*{\bigcirc}="2"
!{(0.4,0)}*{\alpha_2}="label2"
!{(0,1)}*{\bigcirc}="3"
!{(0.4,1)}*{\alpha_3}="label3"
!{(-1,0)}*{\bigcirc}="4"
!{(-1,0.4)}*{\alpha_4}="label4"
"1"-"2" "2"-"3" "2"-"4"
} .
\]
Since $\Aut\bk{Dyn\bk{D_4}}\cong S_3$ we have
\[
\set{\text{Quasi-split forms of $D_4$ over $F$}} \longleftrightarrow \set{\varphi:\Gal\bk{\bar{F}\rmod F}\to S_3}.
\]
On the other hand, there is a bijection
\[
\set{\varphi:\Gal\bk{\bar{F}\rmod F}\to S_3} \longleftrightarrow \set{\begin{matrix} \text{Isomorphism classes of} \\ \text{\'etale cubic algebras over $F$} \end{matrix}}.
\]
For any cubic algebra $E$ let $S_E=Aut_F(E)$ which is a twisted form of $S_3$. 
An action of $S_E$ on the algebraic group $Spin_8$ determines  
a simply-connected quasi-split form $H_E=D_4^E$ of $D_4$ over $F$. 
We fix a Chevalley-Steinberg system of \'epinglage \cite[4.1.3]{MR756316}
\[
\set{T_E, B_E, x_\gamma:\Ga\rightarrow\bk{H_E}_\gamma,\gamma\in \Phi_{D_4}} \ ,
\]
where $T_E\subset B_E$ is a maximal torus contained in a Borel subgroup (both defined over $F$) and $\Phi_{D_4}$ are the roots of $H_E\otimes \overline{F}\equiv D_4\bk{\overline{F}}$.
For any $\gamma$ in the reduced root system of $H_E$ we denote by $F_\gamma$ the field of definition of $\gamma$.



We now give a more detailed description of $H_E$ for the different kinds of \'etale cubic algebras over $F$ in terms of the action of $\Gal\bk{\overline{F}\rmod F}$ on $H_E\bk{\overline{F}}$.

\begin{enumerate}
\item\underline{$E=F\times F\times F$:}
In this case $H_E$ is the split reductive simply-connected group of type $D_4$ over $F$.
It corresponds to the trivial action of $\Gal\bk{\bar{F}\rmod F}$.

\item\underline{$E=F\times K$:}
In this case we make a choice of one of the roots $\alpha_1,\alpha_3,\alpha_4$ and glue the other two together.
The fact that for each choice of distinct root we get an isomorphic algebraic group is called \emph{triality}.
In what follows we choose $\alpha_1$ to be the distinct root.
This is the case where $E=F\times K$ with $K$ a quadratic (and hence Galois) extension of $F$.
It is enough to define an action of $\Gal\bk{K\rmod F}=\gen{\sigma}$ on $Spin_8\bk{K}$.
This action is determined by
\begin{align*}
\sigma\bk{x_{\alpha_1}\bk{k}}&=x_{\alpha_1}\bk{\sigma\bk{k}} \\
\sigma\bk{x_{\alpha_2}\bk{k}}&=x_{\alpha_2}\bk{\sigma\bk{k}} \\
\sigma\bk{x_{\alpha_3}\bk{k}}&=x_{\alpha_4}\bk{\sigma\bk{k}} \\
\sigma\bk{x_{\alpha_4}\bk{k}}&=x_{\alpha_3}\bk{\sigma\bk{k}} .
\end{align*}

\item\underline{$E$ is a Cubic Galois Field Extension:}
Here we assume that $E$ is a cubic Galois extension of $F$.
It is enough to define an action of $\Gal\bk{E\rmod F}=\gen{\sigma\mvert \sigma^3=1}$ on $Spin_8\bk{E}$.
This action is determined by
\begin{align*}
\sigma\bk{x_{\alpha_2}\bk{e}}&=x_{\alpha_2}\bk{\sigma\bk{e}} \\
\sigma\bk{x_{\alpha_1}\bk{e}}&=x_{\alpha_3}\bk{\sigma\bk{e}} \\
\sigma\bk{x_{\alpha_3}\bk{e}}&=x_{\alpha_4}\bk{\sigma\bk{e}} \\
\sigma\bk{x_{\alpha_4}\bk{e}}&=x_{\alpha_1}\bk{\sigma\bk{e}} .
\end{align*}

\item \underline{$E$ is a Cubic non-Galois Field Extension:}
Here we assume that $E$ is a cubic non-Galois extension of $F$. In order to define $H_E\bk{F}$ we first consider the Galois closure $L$ of $E$ over $F$ as above with $\Gal\bk{L\rmod F}=\gen{\sigma,\ \tau\mvert \sigma^3=1,\ \tau^2=1}$.


The action of $\Gamma_E$ on $Spin_8\bk{L}$ is determined by
\begin{align*}
\tau\bk{x_{\alpha_i}\bk{l}}&=x_{\alpha_i}\bk{\tau\bk{l}} \ \forall i \\
\sigma\bk{x_{\alpha_2}\bk{l}}&=x_{\alpha_2}\bk{\sigma\bk{l}} \\
\sigma\bk{x_{\alpha_1}\bk{l}}&=x_{\alpha_3}\bk{\sigma\bk{l}} \\
\sigma\bk{x_{\alpha_3}\bk{l}}&=x_{\alpha_4}\bk{\sigma\bk{l}} \\
\sigma\bk{x_{\alpha_4}\bk{l}}&=x_{\alpha_1}\bk{\sigma\bk{l}} .
\end{align*}

\end{enumerate}

Also, we have $B_E\cap G=B$.
Furthermore, there exists an Heisenberg parabolic subgroup $P_E=M_E\cdot U_E$ of $H_E$ so that $P_E\cap G = P$, $M_E\cap G=M$ and $U_E\cap G=U$.
In particular
\[
M_E\equiv \set{g\in \Res_{E\rmod F} \operatorname{GL}_2 \mvert \det\bk{g}\in \Gm} .
\]

\subsection{The Degenerate Eisenstein Series}
%
%

Fix a Hecke character $\chi:\A^\times\to \C^\times$.
We consider the normalized induction
\[
I_{P_E}\bk{\chi,s} = \Ind_{P_E}^{H_E} \bk{\chi\circ{\det}_{M_E}} \FNorm{{\det}_{M_E}}^{s+\frac{5}{2}},
\]
where $\det_{M_E}$ is the determinant character associated to the Levi subgroup $M_E$.
Note that for the modulus character of $P_E$ it holds that $\modf{P_E}\res{M_E}=\FNorm{\det_{M_E}}^5$.

For any $K$-finite standard section $f_s\in I_{P_E}\bk{\chi,s}$ we define the following degenerate Eisenstein series
\begin{equation}
\Eisen_E\bk{\chi, f_s, s, g} = \suml_{\gamma\in P_E\bk{F}\lmod H_E\bk{F}} f_s\bk{\gamma g} .
\end{equation}
This series converges for $\Real\bk{s}>>0$ and admits a meromorphic continuation to the whole complex plane. We normalize the Eisenstein series as follows
\[
\Eisen_E^\ast\bk{\chi, f_s, s, g}=j_E\bk{\chi,s} \Eisen_E\bk{\chi, f_s, s, g},
\]
where
\[
j_E\bk{\chi,s}= \piece{\Lfun_F\bk{s+\frac{5}{2},\chi}\Lfun_F\bk{s+\frac{3}{2},\chi}^2\Lfun_F\bk{2s+1,\chi^2},& E=F\times F\times F \\ 
\Lfun_F\bk{s+\frac{5}{2},\chi}\Lfun_K\bk{s+\frac{3}{2},\chi\circ\Nm_{E\rmod F}}\Lfun_F\bk{2s+1,\chi\circ\Nm},& E=F\times K \\
\frac{\Lfun_F\bk{s+\frac{5}{2},\chi}\Lfun_E\bk{s+\frac{3}{2},\chi\circ\Nm_{E\rmod F}}\Lfun_F\bk{2s+1,\chi^2}}{\Lfun_F\bk{s+\frac{3}{2},\chi}},& E \text{ field}} .
\]



\section{The Zeta Integral}
Let $\pi=\placestimes \pi_\nu$ be an irreducible cuspidal representation supporting an $\bk{U,\Psi_E}$-Fourier coefficient for some \'etale cubic algebra $E$ over $F$.
Fix a corresponding character $\Psi_E=\placestimes\Psi_{{\bk{a,b,c}},\nu}$ of $U\bk{\A}$ where $\bk{a,b,c}$ satisfy {\bf (CT)}.
Let $\chi=\placestimes \chi_\nu$ be a Hecke character.
For $\varphi\in \pi$ and a standard section $f_s\in I_{P_E}\bk{\chi,s}$ we consider the following integral
\begin{equation}
\label{Eq:Zeta Integral}
\zint_E \bk{\chi, s, \varphi, f} = \intl_{G\bk{F}\lmod G\bk{\A}} \varphi\bk{g} \Eisen_E^\ast\bk{\chi,s,f,g} dg .
\end{equation}
Since $\varphi$ is cuspidal, and hence rapidly decreasing, this integral defines a meromorphic function in the complex plane.
Hence the meromorphic continuation of $\zint_E \bk{\chi, s, \varphi, f}$ follows from that of $\Eisen_E^\ast\bk{\chi,s,f,g}$.

Let $\varphi=\placestimes \varphi_\nu\in\pi$ and $f_s=\placestimes f_{\nu}\in I_{P_E}\bk{\chi,s}$ be pure tensor products.
Let $S\subset\Places$ be a finite set of places such that for $\nu\notin S$ it holds that
\begin{itemize}
\item $2,3\not\vert\nu$ and $\nu\neq\infty$.
\item $E_\nu$ is unramified over $F_\nu$.
\item $\pi_\nu$ and $\chi_\nu$ are unramified.
\item $\varphi_\nu$ and $f_{\nu}$ are spherical.
\item $\psi_\nu$ is of conductor $\mO_\nu$.
\item Either $a= 0$ or $\FNorm{a}\in\mO_\nu^\times$ and similarly for $b$ and $c$.
\item Either $D_{a,b,c}= 0$ or $D_{a,b,c}\in\mO_\nu^\times$ and similarly for $N_{a,b,c}$.
\end{itemize}

\begin{Remark}
Note that if $N_{\bk{a,b,c}}=0$ then necessarily $D_{\bk{a,b,c}}\neq 0$.
\end{Remark}

Our main result is
\begin{Thm}
\label{Thm:Main theorem}
Let $\pi$, $\varphi$, $f$ and $S\subset\Places$ be as above.
Then
\begin{equation}
\zint_E \bk{\chi, s, \varphi, f} = \Lfun^S\bk{s+\frac{1}{2},\pi,\chi,\st} d_S\bk{\chi, s,\Psi_E,\varphi_S,f_S}
\end{equation}
Moreover, for any $s_0$ there exist vectors $\varphi_S$, $f_S$ such that $d_S\bk{\chi, s,\Psi_E,\varphi_S,f_S}$ is analytic in a neighborhood of $s_0$ and $d_S\bk{\chi,s_0,\Psi_E,\varphi_S,f_S}\neq 0$.

In particular, the family of twisted partial $\Lfun$-function $\Lfun^S\bk{s,\pi,\chi,\st}$ admits a meromorphic continuation to the whole complex plane.
\end{Thm}

Most of this paper is devoted to the proof of this theorem. In this section we will outline the main ideas and defer the technical part to later sections and appendices.
In this paper we generalize the approach presented in \cite{MR3284482} while dealing with some delicate issues arising in the non-split case while, as noted in the introduction, do this uniformly for most of the way.
The only place where the proofs differ between the different \'etale cubic algebras is at \cref{Subsec: Convolution} and \cref{Sec: Fourier Transform} where the difference is caused by an arithmetic difference.

An application of this theorem to the $\theta$-lift for the dual pairs $\bk{G,S_E}$ is discussed in \cref{Sec: Application 1}.

In order to perform the unfolding of this integral, we describe the $G\bk{F}$-orbits in $P_E\bk{F}\lmod H_E\bk{F}$, this is done in \cref{Section: G2-orbits}.
In \cref{Sec:Unfolding} we prove the following theorem.
\begin{Thm}[Unfolding]
\label{Thm:Unfolding}
For $\Real\bk{s}\gg 0$ it holds that
\begin{equation}
\label{eq:unfolded integral}
\zint_E \bk{\chi, s, \varphi, f} = \intl_{U\bk{\A}\lmod G\bk{\A}} L_{\Psi_E}\bk{\varphi}\bk{g} F^\ast \bk{\Psi_E, \chi ,g,s} ,
\end{equation}
where
\begin{equation}
F^\ast \bk{\Psi_E,\chi,g,s} = \intl_{\A} f_s\bk{\mu_E x_{3\alpha+\beta}\bk{r}g} \psi\bk{r} dr
\end{equation}
and
\[
\mu_E = w_2w_1w_3w_4 x_{\alpha_1}\bk{a} x_{\alpha_3}\bk{b} x_{\alpha_4}\bk{c} .
\]
\end{Thm}

\begin{Remark}
Note that if the section $f$ is factorizable then so is $F^\ast$. In particular
\[
F^\ast\bk{\Psi_E,\chi,g,s} = \prodl_{\nu\in\Places} F^\ast_\nu \bk{\Psi_{E,\nu},\chi_\nu,g_\nu,s},
\]
where 
\[
F^\ast_\nu \bk{\Psi_{E,\nu},\chi_\nu,g_\nu,s} = \intl_{F_\nu} f_s\bk{\mu_E x_{3\alpha+\beta}\bk{r}g_\nu} \psi_\nu\bk{r} dr .
\]

On the other hand, the integral in \cref{eq:unfolded integral} is a priori not factorizable since the space $\Hom_{U\bk{\A}}\bk{\pi,\C_\Psi}$ is not necessarily one-dimensional and may even be infinite dimensional.
Nevertheless, we will show that the integral is factorizable after all, this follows from an inductive process suggested in \cite{MR965059} and the following two results.

\end{Remark}

\begin{Thm}[Unramified Computation]
\label{Thm:Unramified Computation}
Let $\pi_\nu$ be an irreducible unramified representation of $G\bk{F_\nu}$, and let $v_0$ be a fixed spherical vector in $\pi_\nu$.
There exists $s_0\in\R$ such that for any $\Real\bk{s}>s_0$ and {\bf any} $\Lambda\in \Hom_{U\bk{F_\nu}}\bk{\pi_\nu,\C_{\Psi_{E,\nu}}}$ it holds that
\begin{equation}
\intl_{U\bk{F_\nu}\lmod G\bk{F_\nu}} F^\ast_\nu\bk{\Psi_{E,\nu},\chi_\nu,g,s} \Lambda\bk{\pi_\nu\bk{g}v_0} dg = \Lfun\bk{s+\frac{1}{2},\pi_\nu,\chi_\nu,\st} \Lambda\bk{v_0} .
\end{equation}
\end{Thm}

\begin{Remark}
The case of a split Fourier coefficient and $\chi\equiv 1$ is dealt in \cite{MR3284482}.
The function $F^\ast$ has a different form there due to the different choice of representative for the open orbit that gives rise to a different character $\Psi_s$ in the same orbit of $\Psi_{F\times F\times F}$.
In \cref{Sec:Unramified Computation} we relate $F^\ast_\nu\bk{\Psi_{E,\nu},\chi_\nu,g,s}$ and $F^\ast_\nu\bk{\Psi_{s,\nu},\chi_\nu,g,s}$.
\end{Remark}

\begin{Thm}[Ramified Computation]
\label{Thm:Ramified Computation}
For any $s_0\in\C$ there exist datum $\varphi_S$ and $f_S$ such that $d_S\bk{\chi, s,\Psi_E,\varphi_S,f_S}$
is holomorphic and non-vanishing in a neighborhood of $s_0$.
\end{Thm}

\Cref{Thm:Unramified Computation} is proved in \cref{Sec:Unramified Computation} together with the appendices, \cref{Thm:Ramified Computation} is proved in \cref{Sec:Ramified Computation}.
We now show how to derive the main theorem from the other results presented in this section.
The proof is essentially the same as in \cite{MR3284482}, we include it for the convenience of the reader.

\begin{proof}[Proof of \cref{Thm:Main theorem}]
By \cref{Thm:Unfolding}
\begin{equation}
\label{eq:Integral as a limit of integrals}
\zint_E\bk{\chi, s,\varphi,f}=
\mathop{\mathop{\lim_{\longrightarrow}}_{S\subset\Omega\subset{\Places}}}_{\Card{\Omega}<\infty}
\int_{U\bk{\A}_{\Omega}\setminus G\bk{\A}_{\Omega}} L_{\Psi_s}\bk{\varphi}\bk{g}F_\Omega^\ast\bk{\Psi_{\bk{a,b,c},\Omega},\chi_\Omega,g,s}\, dg \ ,
\end{equation}
where $\displaystyle G\bk{\A}_{\Omega}=\prod_{\nu\in\Omega}G\bk{F_{\nu}}$ and
\[
F_\Omega^\ast\bk{\Psi_{E,\Omega},\chi_\Omega,g,s}=j_{E,\Omega}(s)\int\limits_{F_\Omega} f_s\bk{\mu_E x_{3\alpha+\beta}\bk{r}g} \psi_\Omega \bk{r} dr .
\]
Fix $s_0\in\R$ such that the right hand side of \cref{eq:unfolded integral} converges 
for $\Real{s}\more{s_0}$. The integrals on the right hand side of \cref{eq:Integral as a limit of integrals} must also converge there.
Fix a finite subset $S\subseteq\Omega\subset\Places$ and $\nu\notin \Omega$.
Also fix $s_1\in\R$ such that \cref{Thm:Unramified Computation} holds for $\Real{s}\more{s_1}$ and $\pi_\nu$.
It the holds that
\begin{align*}
&\int_{U\bk{\A}_{\Omega\cup\{\nu\}}\setminus G\bk{\A}_{\Omega\cup\set{\nu}}} L_{\Psi_E}\bk{\varphi}\bk{g}F_{\Omega\cup \set{\nu}}^\ast\bk{\Psi_{E,\Omega\cup\set{\nu}},\chi_{\Omega\cup\set{\nu}},g,s}\, dg = \\
& \int_{U\bk{\A}_{\Omega}\setminus G\bk{\A}_{\Omega}} 
\int_{U\bk{F_{\nu}}\setminus G\bk{F_{\nu}}} L_{\Psi_E}\bk{\varphi}\bk{g g_{\nu}} F_{\Omega\cup \set{\nu}}^\ast\bk{\Psi_{E,\Omega\cup\set{\nu}},\chi_{\Omega\cup\set{\nu}}, g g_{\nu},s}\, dg_{\nu}\, dg= \\
& \int_{U\bk{\A}_{\Omega}\setminus G\bk{\A}_{\Omega}} F_\Omega^\ast\bk{\Psi_{E,\Omega},\chi_\Omega,g,s} 
\int_{U\bk{F_{\nu}}\setminus G\bk{F_{\nu}}} L_{\Psi_E}\bk{\varphi}\bk{g g_{\nu}} F_\nu^\ast\bk{\Psi_{E,\nu},\chi_\nu,g_{\nu},s}\, dg_{\nu}\, dg= \\
&\Lfun\bk{s+\frac{1}{2},\pi_{\nu},\st} 
\int_{U\bk{\A}_{\Omega}\setminus G\bk{\A}_{\Omega}} L_{\Psi_E}\bk{\varphi}\bk{g} F_\Omega^\ast\bk{\Psi_{E,\Omega},\chi_\Omega,g,s}\, dg \ ,
\end{align*}
where the last equality is due to \cref{Thm:Unramified Computation}. A priori the last equality holds only
 for $\Real{s}\more\max\set{s_0,s_1}$, but since $\Lfun\bk{s+\frac{1}{2},\pi_{\nu},\st}$ is a meromorphic 
function the equality actually holds for $\Real{s}\more{s_0}$. Plugging this into 
\cref{eq:Integral as a limit of integrals} we get
\begin{align*}
\zint_E\bk{\chi, s,\varphi,f}&=
\mathop{\mathop{\lim_{\longrightarrow}}_{S\subset\Omega\subset{\Places}}}_{\Card{\Omega}<\infty} 
\int_{U\bk{\A}_{\Omega}\setminus G\bk{\A}_{\Omega}} L_{\Psi_E}\bk{\varphi}\bk{g}F_\Omega^\ast\bk{\Psi_{E,\Omega},\chi_\Omega,g,s}\, dg= \\
&= \mathop{\mathop{\lim_{\longrightarrow}}_{S\subset\Omega\subset{\Places}}}_{\Card{\Omega}<\infty} 
\prod_{\nu\in{\Omega\setminus{S}}} \Lfun\bk{s+\frac{1}{2},\pi_{\nu},\chi,\st} 
\int_{U\bk{\A}_{S}\setminus G\bk{\A}_{S}} L_{\Psi_E}\bk{\varphi}\bk{g}F_S^\ast\bk{\Psi_{E,S},\chi_S,g,s}\, dg=\\
&=\Lfun^{S}\bk{s+\frac{1}{2},\pi,\chi,\st} \int_{U\bk{\A}_{S}\setminus G\bk{\A}_{S}} 
L_{\Psi_E}\bk{\varphi}\bk{g}F_S^\ast\bk{\Psi_{E,S},\chi_S,g,s}\, dg \ .
\end{align*}

We finish the proof by fixing our datum according to \cref{Thm:Ramified Computation}and taking
\[
d_{S}\bk{s,{\varphi}_S,f_{S}}=\int_{U\bk{\A}_{S}\setminus G\bk{\A}_{S}} 
L_{\Psi_E}\bk{\varphi}\bk{g}F_S^\ast\bk{\Psi_{E,S},\chi_S,g,s}\, dg \ .
\]
\end{proof}

\section{$G_2$ Orbits in $P_E\lmod H_E$}
\label{Section: G2-orbits}
In order to perform the unfolding of \cref{Eq:Zeta Integral} we need to describe $P_E\bk{F}\lmod H_E\bk{F}\rmod G\bk{F}$ for any \'etale cubic algebra $E$ over $F$.
In order to do this, we give a description of $\bk{P_E\lmod H_E}\bk{\overline{F}}\rmod G\bk{\overline{F}}$.
Although $\bk{P_E\lmod H_E}\bk{\overline{F}}\rmod G\bk{\overline{F}}$ is independent of $E$, the representatives we choose depend on $E$ so that we can use the method of Galois descent in order to compute $\bk{P_E\lmod H_E}\bk{F}\rmod G\bk{F}$.

Denote $P_E\lmod H_E$ by $X_E$, this is a projective variety with a right $G$-action.
Note that $X_E\bk{F}=\bk{P_E\lmod H_E}\bk{F} = P_E\bk{F}\lmod H_E\bk{F}$ due to \cite[Theorem 4.13a]{MR0207712}.
Let $Q=L\cdot V$ be the non-Heisenberg maximal parabolic subgroup of $G$.
Its Levi part $L\cong GL_2$ is generated by the root $\beta$.
The unipotent radical $V$ of $Q$ is a three-step unipotent group, we denote its commutator $\coset{V,V}$ by $R$.
We recall from \cite[Lemma 2.1]{MR1617425} that $X_E\bk{\overline{F}}$ has five $G\bk{\overline{F}}$-orbits, given as follows.
\begin{Lem}
\label{Lem: Dihua split orbits}
The following is a list of representatives of the $G\bk{\bar{F}}$-orbits in $X_E\bk{\bar{F}}$ and their stabilizers:
\begin{enumerate}
\item $\mu=1$ and the stabilizer of $P_E\bk{\bar{F}}\mu G\bk{\bar{F}}$ is $G^\mu=P$.
\item $\mu=w_2w_1, w_2w_3, w_2w_4$ and the stabilizer of $P_E\bk{\bar{F}}\mu G\bk{\bar{F}}$ is $G^\mu=LR$.
\item $\mu= w_2w_3x_{-\alpha_1}\bk{1}$ is a representative of the open orbit and the stabilizer of $P_E\bk{\bar{F}}\mu G\bk{\bar{F}}$ is $G^\mu=T_{3\alpha+2\beta}\cdot U^\mu$, where
\[
T_{3\alpha+2\beta}=\set{h_{3\alpha+2\beta}\bk{t}\mvert t\in \bar{F}^\times},\quad
U^\mu=\set{u\bk{r_1,r_2,r_2,r_4,r_5}\mvert r_i \in \bar{F}}
\]
\end{enumerate}
\end{Lem}

We now give a different set of representatives for these orbits.
For a triple $\bk{a,b,c}\in \overline{F}\times \overline{F}\times \overline{F}=\bk{\Res_{E\rmod F}\Ga}\bk{\overline{F}}$ we denote
\[
\mu\bk{a,b,c} =  w_2w_1w_3w_4 x_{\alpha_1}\bk{a} x_{\alpha_3}\bk{b} x_{\alpha_4}\bk{c} .
\]

For $x,x'\in X_E\bk{\overline{F}}$ we say that $x\sim x'$ if they lie in the same $G\bk{\overline{F}}$-orbit.
\begin{Lem}
\label{Lem:New orbit representatives}
Let $\bk{a,b,c}\in \overline{F}\times \overline{F}\times \overline{F}$ then
\begin{enumerate}
\item $a=b=c$ if and only if $\mu\bk{a,b,c}\sim 1$
\item 
\begin{itemize}
\item$a=b\neq c$ if and only if $\mu\bk{a,b,c}\sim w_2w_4$ 
\item $a\neq b=c$ if and only if $\mu\bk{a,b,c}\sim w_2w_1$ 
\item $a=c\neq b$ if and only if $\mu\bk{a,b,c}\sim w_2w_3$ 
\end{itemize}

\item $a$, $b$ and $c$ are distinct if and only if $\mu\bk{a,b,c}\sim w_2w_3x_{-\alpha_1}\bk{1}$
\end{enumerate}
\end{Lem}


\begin{Remark}
The proof of this lemma relies on the use of M\"obius transformations.
As $M_E\bk{\overline{F}}\cong \bk{GL_2\times GL_2\times GL_2}^0\bk{\overline{F}}$ we have a natural map
\begin{align*}
\bk{GL_2\times GL_2\times GL_2}^0\bk{\overline{F}} & \overset{\Upsilon}{\to} P_E\bk{\overline{F}}\lmod H_E\bk{\overline{F}} \\
m &\mapsto P_E\bk{\overline{F}} w_2\ m .
\end{align*}
We also note that $w_2 \bk{M_E\bk{\overline{F}} \cap B_E\bk{\overline{F}}} w_2 \subset P_E\bk{\overline{F}}$ and hence this map factors through
\[
\bk{B_0\times B_0\times B_0}^0\bk{\overline{F}} \lmod \bk{GL_2\times GL_2\times GL_2}^0\bk{\overline{F}} ,
\]
where $B_0$ denotes the Borel subgroup of $GL_2$.
This quotient admits an action of $GL_2$ from the right diagonally.
On the other hand, $M\cong GL_2$ acts from the right on $P_E\bk{\overline{F}}\lmod H_E\bk{\overline{F}}$.
The map $\Upsilon$ is evidently $GL_2$-equivariant.
Namely, for any $\bk{m_1,m_2,m_3}\in \bk{GL_2\times GL_2\times GL_2}^0\bk{\overline{F}}$ and any $m\in GL_2\bk{\overline{F}}$ it holds that
\[
\Upsilon\bk{\bk{m_1,m_2,m_3}\cdot m} = \Upsilon\bk{\bk{m_1,m_2,m_3}} \cdot m
\]

Recall that due to the Bruhat decomposition, $\bk{B_0\bk{\overline{F}}\lmod GL_2\bk{\overline{F}}}\cong \mathbb{P}^1\bk{\overline{F}}$.
This identification can be realized as follows
\[
x\longleftrightarrow B_0\bk{\overline{F}} \begin{pmatrix}0&1\\-1&0\end{pmatrix} \begin{pmatrix}1&x\\0&1\end{pmatrix},\quad \infty\longleftrightarrow B_0\bk{\overline{F}}.
\]
Hence, the right action of $M\bk{\overline{F}}$ on $P_E\bk{\overline{F}}\lmod H_E\bk{\overline{F}}$ is induced from the action of $GL_2\bk{\overline{F}}$ on triples of points in $\mathbb{P}^1\bk{\overline{F}}$.
In particular, using this projective coordinates, we note that $w_2w_1 = \Upsilon\bk{0,\infty,\infty}$, $w_2w_3 = \Upsilon\bk{\infty,0,\infty}$, $w_2w_4 = \Upsilon\bk{\infty,\infty,0}$ and $w_2w_3x_{-\alpha_1}\bk{1} = \Upsilon\bk{1,0,\infty}$.
As for the trivial orbit, $1$ is not in the image of $\Upsilon$ but $1\sim \Upsilon\bk{\infty,\infty,\infty}$.

Furthermore, the orbits of $GL_2\bk{\overline{F}}$ in $\mathbb{P}^1\bk{\overline{F}}\times \mathbb{P}^1\bk{\overline{F}}\times \mathbb{P}^1\bk{\overline{F}}$ correspond to the items in \cref{Lem:New orbit representatives}.
Namely, $\bk{\mathbb{P}^1\bk{\overline{F}}\times \mathbb{P}^1\bk{\overline{F}}\times \mathbb{P}^1\bk{\overline{F}}} \rmod GL_2\bk{\overline{F}}$ is given by:
\begin{enumerate}
\item $\set{\bk{a,a,a}\mvert a\in \mathbb{P}^1\bk{\overline{F}}}$.
\item
\begin{itemize}
\item $\set{\bk{a,a,c}\mvert a,c\in \mathbb{P}^1\bk{\overline{F}} \text{ distinct}}$.
\item $\set{\bk{a,b,b}\mvert a,b\in \mathbb{P}^1\bk{\overline{F}} \text{ distinct}}$.
\item $\set{\bk{a,b,a}\mvert a,b\in \mathbb{P}^1\bk{\overline{F}} \text{ distinct}}$.
\end{itemize}
\item $\set{\bk{a,b,c}\mvert a,b,c\in \mathbb{P}^1\bk{\overline{F}} \text{ distinct}}$.
\end{enumerate}
In fact, the lemma proves that there is a natural bijection
\begin{equation}
\bk{B_0\times B_0\times B_0}^0\bk{\overline{F}} \lmod \bk{GL_2\times GL_2\times GL_2}^0\bk{\overline{F}} \rmod GL_2\bk{\overline{F}} \longleftrightarrow P_E\bk{\overline{F}} \lmod H_E\bk{\overline{F}} \rmod G\bk{\overline{F}} .
\end{equation}

\end{Remark}

\begin{proof}
\begin{enumerate}
\item In this case $\mu\bk{a,a,a} = w_2w_1w_3w_4 x_\alpha\bk{a}\in G_2\bk{\overline{F}}$ and hence $\mu\bk{a,a,a}\sim 1$.

\item We prove for example that $\mu\bk{a,a,c}\sim w_2w_4$ for $a\neq c$.
In this case, let
\[
m\bk{a,a,c} = x_\alpha\bk{-a} w_\alpha \check{\beta}\bk{a-c} x_\alpha\bk{1}\in M\bk{F}\subset G\bk{F} .
\]
One checks that
\[
\mu\bk{a,a,c} m\bk{a,a,c} \bk{w_2w_4}^{-1} \in B_E\bk{\overline{F}}\subset P_E\bk{\overline{F}}
\]
and hence
\[
\mu\bk{a,a,c} \sim w_2w_4 .
\]



\item
Let 
\[
m_{a,b,c}=\check{\alpha}\bk{b-a} \check{\beta}\bk{\frac{\bk{a-b}^3}{\bk{a-c}\bk{b-c}}} x_\alpha\bk{c \frac{\bk{a-b}}{\bk{c-a}\bk{b-c}}} w_\alpha x_{\alpha}\bk{\frac{\bk{a-c}}{\bk{a-b}}} \in M\bk{F}\subset G\bk{F}.
\]
One checks that
\[
\mu\bk{a,b,c} m_{a,b,c} \bk{w_2w_3x_{-\alpha_1}\bk{1}}^{-1}\in B_E\bk{\overline{F}}\subset P_E\bk{\overline{F}},
\]
hence
\[
\mu\bk{a,b,c} \sim w_2w_3x_{-\alpha_1}\bk{1} \ .
\]
Since \cref{Lem: Dihua split orbits} gives a list of representatives of orbits, the other direction follows immediately.
\end{enumerate}
\end{proof}

For the rest of the paper we fix a triple $\bk{a,b,c}$ as in {\bf (CT)} and write $\mu_E=\mu\bk{a,b,c}$.

\begin{Cor}
\label{Cor: Stabilizer of the open orbit}
$\Stab_{G\bk{\overline{F}}}\bk{P_E\bk{\overline{F}}\mu_E} = T_{3\alpha+2\beta}\cdot \ker \Psi_E$.
\end{Cor}

\begin{proof}
Recall from \cite{MR3284482} that $\Psi_s$ is given by
\[
\Psi_s\bk{u\bk{r_1,r_2,r_3,r_4,r_5}} = \psi\bk{r_2+r_3} .
\]
and note that 
\[
\Stab_{G\bk{\overline{F}}}\bk{P_E\bk{\overline{F}}w_2w_3x_{-\alpha_1}\bk{1}} = T_{3\alpha+2\beta}\cdot \ker \Psi_s .
\]
We also note that $\ker\Psi_s$ and $\ker\Psi_E$ define algebraic groups over $F$.


From \cref{Lem:New orbit representatives}
\[
P_E\bk{\overline{F}}w_2w_3x_{-\alpha_1}\bk{1} = P_E\bk{\overline{F}}\mu_E m_{a,b,c}
\]
and hence
\[
\Stab_{G\bk{\overline{F}}}\bk{P_E\bk{\overline{F}}\mu_E} = m_{a,b,c} \Stab_{G\bk{\overline{F}}}\bk{P_E\bk{\overline{F}}w_2w_3x_{-\alpha_1}\bk{1}} m_{a,b,c}^{-1} .
\]
On the other hand,
\[
\Psi_E\bk{u} = \Psi_s\bk{m_{a,b,c}^{-1} u m_{a,b,c}}
\]
and hence
\[
\ker \Psi_E = m_{a,b,c}^{-1} \bk{\ker \Psi_s} m_{a,b,c} .
\]
Also, $T_{3\alpha+2\beta}=\mathcal{Z}\bk{M}$ and hence $m_{a,b,c}^{-1} T_{3\alpha+2\beta} m_{a,b,c} = T_{3\alpha+2\beta}$.
\end{proof}

We now move to describing the $G\bk{F}$ orbits in $X_E\bk{F}$.
Given $\mu\in H_E\bk{F}\subseteq H_E\bk{\overline{F}}$ we have a short exact sequence
\[
\shortexact{\Stab_{G\bk{\bar{F}}}\bk{\mu}}{G_2\bk{\bar{F}}}{P_E\bk{\overline{F}}\mu G\bk{\overline{F}}}
\]
Denote by $X_\mu$ the $G$-orbit of $P_E\bk{F} \mu\in X_E\bk{F}$.
We then have (\cite[II.4.7]{MR2723693}) a long exact sequence in the cohomology
\[
\xymatrix{
\set{1} \ar[r] & \Stab_{G\bk{\bar{F}}}\bk{\mu}^{\Gamma} \ar[r] & G_2\bk{F} \ar[r]& {X_\mu\bk{F}} \ar `r[d] `_l[lll] `^d[dlll] `^r[dll] [dll] \\
& \cohom{1}\bk{\Gamma_E, \Stab_{G\bk{\bar{F}}}\bk{\mu}} \ar[r] & \cohom{1}\bk{\Gamma_E, G_2\bk{\bar{F}}} \ar[r]& \cohom{1}\bk{\Gamma_E, X_\mu\bk{\bar{F}}}} .
\]
In particular, there is a bijection (of sets)
\[
X_\mu\bk{F}\rmod G_2\bk{F} \longleftrightarrow \Ker\coset{\cohom{1}\bk{\Gamma_E, \Stab_{G\bk{F}}\bk{\mu}} \rightarrow \cohom{1}\bk{\Gamma_E, G_2\bk{F}}} .
\]
Since $G_2$ is a split and reductive algebraic group, then according to \cite[\S 8]{MR1321649} it holds that $\cohom{1}\bk{\Gamma_E, G_2\bk{F}}=\set{1}$ and hence $X_\mu\bk{F}\rmod G_2\bk{F}$ is parametrized by $\cohom{1}\bk{\Gamma_E, \Stab_{G\bk{F}}\bk{\mu}}$.

\begin{Thm}
\label{Thm: G2-orbits}
The $G\bk{F}$-orbits in $P_E\bk{F}\lmod H_E\bk{F}$ are
\begin{enumerate}
%
%
\item \underline{$E=F\times F\times F$:} $P_E\bk{F}$, $P_E\bk{F}w_2w_1$, $P_E\bk{F}w_2w_3$, $P_E\bk{F}w_2w_4$ and $P_E\bk{F}w_2w_1w_3w_4 x_{\alpha_1}\bk{1}x_{\alpha_3}\bk{-1}$.

\item \underline{$E=F\times K$, $K=F\coset{\theta}$ a field:} $P_E\bk{F}$, $P_E\bk{F}w_2w_1$ and $P_E\bk{F}w_2w_1w_3w_4 x_{\alpha_3}\bk{\theta}x_{\alpha_4}\bk{\theta^\sigma}$.

\item \underline{$E=F\coset{\theta}$ a field:} $P_E\bk{F}$ and $P_E\bk{F}w_2w_1w_3w_4 x_{\alpha_1}\bk{\theta}x_{\alpha_3}\bk{\theta^\sigma}x_{\alpha_4}\bk{\theta^{\sigma^2}}$.
\end{enumerate}
\end{Thm}

\begin{proof}
\begin{enumerate}
\item This is exactly \cref{Lem: Dihua split orbits} combined with \cref{Lem:New orbit representatives}.

\item
We start by noting that $P_E\bk{F}$, $P_E\bk{F}w_2w_4$ and $P_E\bk{F}w_2w_1w_3w_4 x_{\alpha_1}\bk{\theta}x_{\alpha_3}\bk{\theta^\sigma}$ are the only $G\bk{\overline{F}}$-orbits in $P_E\bk{\overline{F}}\lmod H_E\bk{\overline{F}}$ that intersects $H_E\bk{F}$, we need only show that for these $\mu$-s, $P_E\bk{F}\mu$ is one $G\bk{F}$-orbit.


\begin{itemize}
\item \underline{$\mu=1$:}
In this case, $\Stab_{G}\bk{\mu}=P$ and we have a short exact sequence
\[
\shortexact{U}{P}{M} .
\]
$\cohom{1}\bk{\Gamma_E,U\bk{F}}=\set{1}$ since $U$ is unipotent and $\cohom{1}\bk{\Gamma_E,M\bk{F}}=\set{1}$ due to Hilbert 90' and thus $\cohom{1}\bk{\Gamma_E,P\bk{F}}=\set{1}$.
Hence $X_1\bk{F}\rmod G_2\bk{F}=\set{1}$.

\item \underline{$\mu=w_2w_1$:}
In this case, $\Stab_{G}\bk{P_E\mu}=L\cdot R$ and we have a short exact sequence
\[
\shortexact{R}{\Stab_{G}\bk{P_E\mu}}{L} .
\]
$\cohom{1}\bk{\Gamma_E,R\bk{F}}=\set{1}$ since $R$ is unipotent and $\cohom{1}\bk{\Gamma_E,L\bk{F}}=\set{1}$ due to Hilbert 90' and thus $\cohom{1}\bk{\Gamma_E,\Stab_{G\bk{F}}\bk{P_E\bk{F}\mu}}=\set{1}$.
Hence $X_\mu\bk{F}\rmod G_2\bk{F}=\set{P_E\bk{F}\mu}$.

\item \underline{$\mu=P_E\bk{F}w_2w_1w_3w_4 x_{\alpha_3}\bk{\theta}x_{\alpha_4}\bk{\theta^\sigma}$:}
In this case, $\Stab_{G}\bk{P_E\mu}=T_{3\alpha+2\beta}\cdot \ker\Psi_E$ and we have a short exact sequence
\[
\shortexact{\ker\Psi_E}{\Stab_{G}\bk{P_E\mu}}{T_{3\alpha+2\beta}} .
\]
$\cohom{1}\bk{\Gamma_E,\bk{\ker\Psi_E}\bk{F}}=\set{1}$ since $\ker\Psi_E$ is unipotent and $\cohom{1}\bk{\Gamma_E,T_{3\alpha+2\beta}\bk{F}}=\set{1}$ due to Hilbert 90' and thus $\cohom{1}\bk{\Gamma_E,\Stab_{G\bk{F}}\bk{P_E\bk{F}\mu}}=\set{1}$.
Hence $X_\mu\bk{F}\rmod G_2\bk{F}=\set{P_E\bk{F}\mu}$.
\end{itemize}

\item This case is proven similarly.
\end{enumerate}
\end{proof}


\begin{Remark}
We recall that for $\mu\in H_E\bk{F}$
\[
\Stab_{G\bk{F}}\bk{P_E\bk{F}\mu} = \Stab_{G\bk{\overline{F}}}\bk{P_E\bk{\overline{F}}\mu} \cap G\bk{F} .
\]
\end{Remark}

\begin{Cor}
\begin{enumerate}
\item It holds that $\Stab_{G\bk{F}}\bk{P_E\bk{F}} = P\bk{F}$.
\item It holds that $\Stab_{G\bk{F}}\bk{P_E\bk{F}w_2w_1} = \Stab_{G\bk{\overline{F}}}\bk{P_E\bk{F}w_2w_3} = \Stab_{G\bk{F}}\bk{P_E\bk{F}w_2w_4} = L\cdot R$ if applicable.
\item It holds that $\Stab_{G\bk{F}}\bk{P_E\bk{F}\mu_E} = T_{3\alpha+2\beta}\cdot \ker \Psi_E$ .
\end{enumerate}
\end{Cor}


%
%

\section{Unfolding of the Zeta Integral}
\label{Sec:Unfolding}
In this section we prove \cref{Thm:Unfolding}.
We also denote by $N_\beta$ the unipotent radical of the Borel subgroup $B\cap L$ of $L$.



\begin{proof}[Proof of \cref{Thm:Unfolding}]
For $\Real\bk{s}>>0$ it holds that
\begin{align*}
\frac{1}{j_E\bk{\chi,s}}\zint_E \bk{\chi, s, \varphi, f} &= \intl_{G\bk{F}\lmod G\bk{\A}} \varphi\bk{g} \Eisen_E\bk{\chi,s,f,g} dg \\
&= \intl_{G\bk{F}\lmod G\bk{\A}} \varphi\bk{g} \suml_{\gamma\in P_E\bk{F}\lmod H_E\bk{F}} f_s\bk{\gamma g} dg \\
&= \suml_{\mu\in P_E\bk{F}\lmod H_E\bk{F}\rmod G\bk{F}} I_\mu \bk{\varphi,f_s} ,
\end{align*}
where
\[
I_\mu\bk{\varphi, f_s} = \intl_{G^\mu\bk{F}\lmod G\bk{\A}} \varphi\bk{g} f_s\bk{\mu g} dg .
\]
We now show that $I_\mu\bk{\varphi,f_s}=0$ unless $\mu$ is a representative of the open orbit.

\begin{enumerate}
\item \underline{If $\mu=1$ then:}
In this case
\begin{align*}
I_\mu\bk{\varphi,f_s} &= \intl_{P\bk{F}\lmod G\bk{\A}} \varphi\bk{g} dg \\
&= \intl_{M\bk{F}U\bk{\A}\lmod G\bk{F}} f_s\bk{g} \bk{\intl_{U\bk{F}\lmod U\bk{\A}} \varphi\bk{ug} du} dg = 0,
\end{align*}
since $\varphi$ is cuspidal.

\item \underline{If $\mu\in \set{w_2w_1, w_2w_3, w_2w_4}$ then:}
In this case,
\begin{align*}
I_\mu\bk{\varphi,f_s} &= \intl_{L\bk{F}\cdot R\bk{F}\lmod G\bk{\A}} \varphi\bk{g} f_s\bk{\mu g} dg \\
&= \intl_{L\bk{F}\cdot R\bk{\A}\lmod G\bk{\A}} f_s\bk{\mu g} \bk{\intl_{R\bk{F}\lmod R\bk{\A}} \varphi\bk{rg} dr} dg .
\end{align*}

We recall from \cite[Theorem 5]{MR1020830} that
\[
\intl_{R\bk{F}\lmod R\bk{\A}} \varphi\bk{rg} dr = \suml_{\nu\in N_\beta\bk{F} \lmod L\bk{F}} W_{\psi}\bk{\varphi} \bk{\nu g} ,
\]
where $W_{\psi}\bk{\varphi}$ is the standard Whittaker coefficient of $\varphi$.
We then have
\begin{align*}
I_\mu\bk{\varphi,f_s} &=  \intl_{L\bk{F}\cdot R\bk{\A}\lmod G\bk{\A}} f_s\bk{\mu g} \bk{\suml_{\nu\in N_\beta\bk{F} \lmod L\bk{F}} W_{\psi}\bk{\varphi} \bk{\nu g}} dg \\
&= \intl_{N_\beta\bk{\A}\cdot R\bk{\A}\lmod G\bk{\A}} f_s\bk{\mu g} W_{\psi}\bk{\varphi} \bk{\nu g} \bk{\intl_{N_\beta\bk{F}\lmod N_\beta\bk{\A}} \psi\bk{n} dn} dg = 0 .
\end{align*}

\item \underline{If $\mu=\mu_E$ then:}
Fix $\mu=\mu_E$ to be the representative of the open orbit as in \cref{Thm: G2-orbits}.
Also let $T^\mu=T_{3\alpha+2\beta}$ and $U^\mu=\ker \Psi_E$ so that $\Stab_{G}\bk{P_E\mu}=T^\mu U^\mu$.
It holds that
\[
I_\mu\bk{\varphi, f_s}=\intl_{T^\mu\bk{F}U^\mu\bk{\A}\lmod G\bk{\A}} f_s\bk{\mu g} \bk{\intl_{U^\mu\bk{F}\lmod U^\mu\bk{\A}} \varphi\bk{ug} du} dg.
\]

We now expand the function given by the inner integral along the 1-dimensional subgroup generated by the root $3\alpha+\beta$.
Using cuspidality, we collapse the sum with the outer integration, and conclude that the above equals
\[
I_\mu\bk{\varphi, f_s}=\intl_{U^\mu\bk{\A}\lmod G\bk{\A}} f_s\bk{\mu g} \bk{\intl_{U\bk{F}\lmod U\bk{\A}} \varphi\bk{ug} \overline{\Psi_E\bk{u}} du} dg.
\]
Since $U=U^\mu\cdot U_{3\alpha+\beta}$ we have
\[
I_\mu\bk{\varphi, f_s}=\intl_{U\bk{\A}\lmod G\bk{\A}} \bk{\intl_{U\bk{F}\lmod U\bk{\A}} \varphi\bk{ug} \overline{\Psi_E\bk{u}} du} \bk{\intl_{\A} f_s\bk{\mu x_{3\alpha+\beta}\bk{r}g} \psi\bk{r} dr} dg .
\]
\end{enumerate}
\end{proof}

\section{Unramified Computation}
\label{Sec:Unramified Computation}
In this section we describe the unramified computation and recall some results from \cite{MR3284482}.
We postpone some of the more involved calculations to the appendices.
For this section and the appendices we fix a place $\nu\notin S$ and drop $\nu$ from all of the notations, i.e. we assume $F$ to be a local non-Archimedean field, $E$ an unramified Galois \'etale cubic algebra over $F$, $\pi$ an irreducible unramified representation of $G\bk{F}$ etc.
We fix on $G\bk{F}$ the unique Haar measure $\meas$ such that $\meas\bk{K}=1$, where $K=G\bk{\mO}$.

Recall the Satake isomorphism between the spherical Hecke algebra $\Hecke=\Hecke\bk{G,K}$ and the Grothendick ring $Rep\bk{\ldual{G}}$ described in \cite{MR1696481}.
Denote by $A_j\in \Hecke$ the elements corresponding to $Sym^j\bk{\st}$ under the Satake isomorphism.
In particular, for any unramified representation $\pi$ and a spherical vector $v_0\in\pi$ it holds that
\begin{equation}
\label{Eq: Satake Transform}
\intl_{G\bk{F}} A_j\bk{g}\pi\bk{g}v_0 \ dg = \Tr\bk{Sym^j\bk{\st}\bk{t_\pi}} v_0,
\end{equation}
where $t_\pi$ is the Satake parameter of $\pi$.

For any such $\pi$, the Satake isomorphism induces an algebra homomorphism
\begin{align*}
\Hecke &\to \C \\
f &\to \widehat{f}\bk{\pi} ,
\end{align*}
where $\widehat{f}\bk{\pi}$ is given by
\[
\intl_{G\bk{F}} f\bk{g}\pi\bk{g}v_0 \ dg = \widehat{f}\bk{\pi} v_0 .
\]
In particular, for any $f_1,f_2\in\Hecke$ it holds that $\widehat{f_1\conv f_2}=\widehat{f_1}\cdot\widehat{f_2}$.
The homomorphism $f\to\widehat{f}\bk{\pi}$ can be extended linearly to a map of formal power series $\Hecke\coset{\coset{T}}\to \C\coset{\coset{T}}$.
Let $T=q^{-s}$.

\begin{Prop}
For any finite order character $\chi$ of F, there exists a generating function $\Delta_{\chi,s}\in \Hecke\coset{\coset{q^{-s}}}$, uniformly converging on a right half plane, such that for any unramified representation $\pi$ with a spherical vector $v_0$ and \textbf{any} functional $l$ on $\pi$, it holds that
\begin{equation}
\intl_{G\bk{F}} \Delta_{\chi,s}\bk{g}\Lambda\bk{\pi\bk{g}v_0} dg = \Lfun\bk{s,\pi,\chi,\st} \Lambda\bk{v_0}
\end{equation}
for $\Real\bk{s}\gg 0$.
\end{Prop}

\begin{proof}
We have Poincar\'e's identity
\begin{align*}
\Lfun\bk{s,\pi,\chi,\st}&=\frac{1}{\det\bk{\Id-q^{-s}\st\bk{t_{\pi\boxtimes\chi}}}} = \prodl_{i=1}^7 \frac{1}{1-q^{-s} \chi\bk{\unif}\st\bk{t_\pi}_{ii}} = \\
&=\prodl_{i=1}^7 \bk{\suml_{k=0}^\infty \bk{q^{-s} \chi\bk{\unif}\st\bk{t_\pi}_{ii}}^k} = \suml_{k=0}^\infty \Tr\bk{Sym^k\bk{t_\pi}} \chi^k\bk{\unif} q^{-ks} 
\end{align*}
Plugging \cref{Eq: Satake Transform} in the above equality yields
\[
\Lfun\bk{s,\pi,\chi,\st} \Lambda\bk{v_0} = \Lambda \bk{\suml_{k=0}^\infty \bk{\intl_{G\bk{F}} A_k\bk{g}\pi\bk{g}v_0 dg} \chi^k\bk{\unif} q^{-ks}}.
\]
In \cite{MR3284482} we prove that this double integral is absolutely convergent (note that $\chi$ is unitary) for $\Real\bk{s}\gg 0$ and hence one can change the order of summation and integration.
The assertion then follows for
\[
\Delta_{\chi,s} = \suml_{k=0}^\infty A_k \chi^k\bk{\unif}q^{-ks} .
\]
\end{proof}

Furthermore, for any complex character $\Psi$ of $U\bk{F}$ and any $l\in \Hom_{U\bk{F}}\bk{\pi,\C_\Psi}$ it holds that
\begin{equation}
\Lfun\bk{s,\pi,\chi,\st} \Lambda\bk{v_0} = \intl_{U\bk{F}\lmod G\bk{F}} \Lambda\bk{\pi\bk{g}v_0} \Delta_{\chi,s}^{\Psi} \bk{g} dg .
\end{equation}

Thus, in order to prove \cref{Thm:Unramified Computation} it is enough to prove
\begin{equation}
\label{Eq: Local Unramified Identity}
\boxed{\Delta_{\chi,s}^{\Psi_E} \bk{g} = F^\ast\bk{\Psi_E,\chi,g,s}\quad \forall g\in G\bk{F} .} 
\end{equation}

We now reduce the proof of \cref{Eq: Local Unramified Identity} to the case of $\chi=1$.
For a given unitary character $\chi$ of $F^\times$ and $s\in\C$ we denote
\[
\chi_s\bk{t} = \chi\bk{t} \FNorm{t}^{s+\frac{5}{2}} .
\]

In \cref{Subsec: Computing F_s} we prove the following result.
\begin{Prop}
\label{Prop: Computing F_s}
For $g=\check{\alpha}\bk{t_1} \check{\beta}\bk{t_2} x_{\alpha}\bk{d}\in M\bk{F}$, let
\[
\alpha_1=\piece{1,&\FNorm{\frac{t_2}{t_1^2}a+d}_{F_{\alpha_1}}\leq 1\\ \frac{t_2}{t_1^2}a+d,& \FNorm{\frac{t_2}{t_1^2}a+d}_{F_{\alpha_1}}\more 1},\  
\alpha_2=\piece{1,&\FNorm{\frac{t_2}{t_1^2}b+d}_{F_{\alpha_3}}\leq 1\\ \frac{t_2}{t_1^2}b+d,& \FNorm{{\frac{t_2}{t_1^2}b+d}}_{F_{\alpha_3}}\more 1},\ 
\alpha_3=\piece{1,&\FNorm{\frac{t_2}{t_1^2}c+d}_{F_{\alpha_4}}\leq 1\\ \frac{t_2}{t_1^2}c+d,& \FNorm{\frac{t_2}{t_1^2}c+d}_{F_{\alpha_4}}\more 1}.
\]
Also denote $\alpha=\frac{t_1^3}{t_2}\alpha_1\alpha_2\alpha_3\in F$.
It holds that
\[
F\bk{\Psi_E,\chi,g,s}=\piece{0,&\FNorm{\alpha}_F>1\\
\chi_s\bk{\frac{t_2}{\alpha}} \frac{\Lfun\bk{s+\frac{3}{2},\chi}}{\Lfun\bk{s+\frac{5}{2},\chi}}
\bk{\FNorm{\alpha}-\chi_s\bk{\unif\alpha}q},& \FNorm{\frac{t_1^3}{t_2}}_F<1}
\]
\end{Prop}

It is clear that the power series of $F^\ast\bk{\Psi_E,\chi,g,s}$ with respect to $q^{-s}$ is of the form $\suml_{k=0}^\infty B_k \chi^k\bk{\unif}q^{-ks}$, where the coefficients $B_k\in\Hecke$ are independent of $\chi$.
Hence, in order to prove \cref{Eq: Local Unramified Identity}, it is enough to prove that
\begin{equation}
\label{Eq: Simplified Local Unramified Identity}
\boxed{\Delta_{1,s}^{\Psi_E} \bk{g} = F^\ast\bk{\Psi_E,1,g,s}\quad \forall g\in G\bk{F} .}
\end{equation}
We therefore devote the rest of this section and the appendices to prove this equality and drop $\chi$ from all notations.


\begin{Remark}
Since $\Delta_{s}^{\Psi_E} \bk{g}, F^\ast\bk{\Psi_E,g,s}\in \M_\Psi$, it suffices to prove the equality for $g\in M\bk{F}\cap B\bk{F}$.
\end{Remark}

While the right hand side of \cref{Eq: Simplified Local Unramified Identity} is given explicitly, we do not have an explicit formula for the generating function $\Delta\bk{g,s}$.
In \cite{MR3284482} we introduced a way to overcome this difficulty.

Recall the Cartan decomposition $G=KT^+K$, where
\[
T^{+}=\set{t\in{T}\mvert \FNorm{\gamma\bk{t}}\leq{1}\,\, \forall \gamma\in\Phi^{+}}.
\]
Let $D_s\in \Hecke\coset{\coset{q^{-s}}}$ be the bi-$K$-invariant defined on $T^+$ by
\[
D_s\bk{t}=\FNorm{\omega_1(t)}^{s+\frac{7}{2}} \qquad \forall t\in T^{+}. \ 
\]
The functions $D$ and $\Delta$ are closely related as can be seen from the following proposition \cite[Proposition 7.1 and 7.2]{MR3284482}.

\begin{Prop}
There exists $P_s\in \Hecke\coset{q^{-s}}$ and $s_0\in\R$ such that for $\Real{s}\more{s_0}$ it holds that
\[
D_s=\Delta_{s+\frac{1}{2}} \conv P_s \ .
\]
More precisely
\begin{equation}
P_s= \frac{P_0\bk{q^{-s-\frac{1}{2}}}A_0-P_1\bk{q^{-s-\frac{1}{2}}}A_1}
{\zfun\bk{s+\frac{3}{2}}\zfun\bk{s+\frac{7}{2}}\zfun\bk{s+\frac{1}{2}}},
\end{equation}
where
\[
P_0\bk{z}=
\frac{z^4}{q^2}+\bk{\frac{1}{q^2}+\frac{1}{q}}z^3+ \frac{z^2}{q}+\bk{\frac{1}{q}+1}z+1,
\quad 
P_1\bk{z}=\frac{z^2}{q} \ .
\]
Furthermore, there exists $s_0$ such that for any $\Real\bk{s}>s_0$ and $f\in\M_{\Psi_E}$, if $f\conv P_s\equiv 0$ then $f\equiv 0$.
\end{Prop}

\begin{Remark}
Recall from \cite{MR1696481} that $A_0=\Id_{K}$ and $A_1=q^{-3}\bk{\Id_{K}+\Id_{K\omega_1\bk{\unif}K}}$ .
\end{Remark}

\begin{Remark}
Since the Fourier transform is a map of $\Hecke$-modules, it holds that
\[
D_s^{\Psi_E} \equiv \Delta_{s+\frac{1}{2}}^{\Psi_E} \conv P_s
\]
\end{Remark}

\begin{Cor}

\Cref{Thm:Unramified Computation} follows from
\begin{equation}
\label{Eq: Refined Local Equality}
D_s^{\Psi_E} \equiv F^\ast\bk{\Psi_E,\cdot,s} \conv P_s .
\end{equation}
\end{Cor}

In order to prove \cref{Eq: Refined Local Equality} we compute explicitly both sides in \cref{Sec: Convolution} and \cref{Sec: Fourier Transform}.
The method is similar to that of the split case but while more delicate issues arise in the computation.

\begin{Remark}
\cref{Eq: Refined Local Equality} is proved in the split case in \cite{MR3284482} for $\Psi_s$ as described there.
It can then be deduced for $\Psi_{F\times F\times F}$ by conjugation with $m_{a,b,c}$ given in \cref{Section: G2-orbits}.
Hence, we devote \cref{Sec: Convolution} and \cref{Sec: Fourier Transform} to proving \cref{Eq: Refined Local Equality} assuming that $E$ is non-split.
\end{Remark}

\section{Ramified Computation}
\label{Sec:Ramified Computation}
In this section we prove \cref{Thm:Ramified Computation}.
The proof of is similar to that of \cite{MR3284482} and we include it for the convenience of the reader.

\begin{proof}[Proof of \cref{Thm:Ramified Computation}]
Let $\mu=\mu_E$.
Recall from \cref{Thm:Unfolding} that 
\[
d_S\bk{\chi, s,\Psi_E,\varphi_S, f_S}= \int\limits_{U_S\backslash G_S} L_{\Psi_E} \bk{\varphi}\bk{g} F^\ast\bk{\Psi_E,\chi, g,s}\, dg=
j_E\bk{\chi, s} \int\limits_{U^\mu_S\backslash G_S} L_{\Psi_E} \bk{\varphi} \bk{g} f_s\bk{\mu_E g} dg \ .
 \]
We denote by $\chi_s$ the extension of $\bk{\chi\circ{\det}_{M_E}} \FNorm{{\det}_{M_E}}^{s+\frac{5}{2}}$ to $P$ trivial on $U$.
For $g\in G_\mu\bk{F_\nu}$ we have
\[
\chi_s^{\mu_E}\bk{g}=
\chi_s\bk{\mu_E p \mu_E^{-1}}.
\]
Since $\mu$ generates the open double coset in $H_E$, there is an inclusion  
\[
i:\ind^{G\bk{F_S}}_{G^{\mu_E}\bk{F_S}} \chi_s^{\mu_E} \hookrightarrow \Ind^{H_E\bk{F_S}}_{P_E\bk{F_S}}\chi_s
\]
defined  by $i(f_s)\bk{p \mu_E g}=\chi_s\bk{p}f_s\bk{g}$ and 
$i\bk{f_s}$ vanishes on all other double cosets $P\bk{F_S}\mu' G\bk{F_S}$.  

For any $\phi_S\in \Sch{F_S}$ define an action of $\phi_S$ on $\pi$ by
\[
\phi_S\ast \varphi = 
\int_{F_S} \phi_S\bk{r}\pi_S\bk{x_{2\alpha+\beta}\bk{r}} \varphi\, dr \ 
\]
It is easy to see that 
\[
L_{\Psi_E}\bk{\pi_S\bk{h_{3\alpha+2\beta}\bk{t}} \bk{\phi_S\ast \varphi}}=\hat{\phi}_S \bk{t} 
L_{\Psi_E}\bk{\pi_S\bk{h_{3\alpha+2\beta}\bk{t}} \varphi}, 
\]
where $\hat \phi_S$ is a Fourier transform of $\phi_S$.

Let us write 
\[
J_S\bk{s,\varphi}=\int\limits_{F^\times_S} L_{\Psi_E}\bk{\varphi}\bk{h_{3\alpha+2\beta}\bk{t}}\mu_s\bk{h_{3\alpha+2\beta}\bk{t}} d^\times t .
\]

\begin{Lem}
\label{inner:ramified:integral}
For any $s_0\in \C$ and any $\varphi\in \pi$ such that $L_{\Psi_E}\bk{\varphi}\neq 0$ there exists $\phi_S\in \Sch{F_S}$ such that $J_S\bk{s,\phi_S\ast\varphi}\neq 0$ around $s_0$.
\end{Lem}

\begin{proof} One has 
\[
\int\limits_{F^\times_S} L_{\Psi_E}\bk{\phi_S\ast\varphi}\bk{h_{3\alpha+2\beta}\bk{t}} \mu_s\bk{h_{3\alpha+2\beta}\bk{t}} d^\times t=
\int\limits_{F^\times_S} \hat \phi_S\bk{t}\,
L_{\Psi_E}\bk{\varphi}\bk{h_{3\alpha+2\beta}\bk{t}} \mu_s\bk{h_{3\alpha+2\beta}\bk{t}} d^\times t .
\]
Since the image of $F^\times_S$ inside $F_S$ is locally closed we may choose $\phi_S$ such that $\hat \phi_S$ is supported on a relatively compact neighborhood of $1\in F^\times_S$.
Choose $\phi_S$ such that the support of $\hat\phi_S$ is sufficiently small to ensure the non-vanishing of $J_S\bk{s,\phi_S\ast\varphi}$ around $s_0$.
\end{proof}

Consider the decomposition
\[
G_S=G^\mu_S \cdot T_S^c\cdot U_S^c\cdot K_S, \quad 
T_S^c=\set{x_{\beta}\bk{t}\mvert t\in k_S^\times}, \quad
U_S^c=\set{x_\alpha\bk{r_1}x_{\alpha+\beta}\bk{r_2}\mvert  r_i\in k_S}
\]
For any Schwarz function $\Phi_S$ on $\bk{T^c\cdot U^c\cdot K}_S$ define $f_s\bk{\Phi_S}\in  \ind^{G_S}_{G_S^{\mu_E}}\chi_s^{\mu_E}$ by
\[
f_s\bk{\Phi_S}\bk{g_\mu g}=\chi_s\bk{g_\mu} \Phi_S\bk{g}\quad \forall g_\mu\in G^\mu_S,\ g\in \bk{T^c\cdot U^c\cdot K}_S.
\]
Then for any $\varphi\in \pi$ it holds that
\begin{equation}
d_S\bk{\chi, s,\Psi_E,\varphi_S, f\bk{\Phi}_s}=J_S\bk{s,\Phi_S\ast\varphi}
\end{equation}

By the Dixmier-Malliavin theorem \cite{DM} there exists a Schwarz function $\Phi_S$ on  $\bk{T^c\cdot U^c\cdot K}_S$ 
and $\varphi\in \pi$ such that $L_{\Psi_s}\bk{\Phi_S\ast \varphi}\neq 0$.
Then for any $\phi_S\in \Sch{k_S}$

\[
d_S\bk{\chi, s,\Psi_E,\phi_S\ast\varphi_S, f\bk{\Phi_S}}=J_S\bk{s,\phi_S\ast \bk{\Phi\ast \varphi}}.
\]

By lemma \ref{inner:ramified:integral} there exists a Schwarz function 
$\phi_S\in \Sch{k_S}$ such that $d_S\bk{\chi, s,\Psi_E,\phi_S\ast\varphi_S, f_s\bk{\Phi_S}}$ 
is an entire function and does not vanish in a neighborhood of $s_0$.
\end{proof}

\section{Application: $\theta$-lift for $\bk{G_2,S_E}$}
\label{Sec: Application 1}
For this section let $E$ be a Galois cubic \'etale algebra over $F$.
We let $\chi_E$ be the finite order Hecke character associated to $E\rmod F$ by class field theory.
Fix
\[
n_E = \piece{2,& E=F\times F\times F \\ 1,& \text{otherwise}}.
\]

For any \'etale cubic algebra $E$ over $F$ and the choice of Chevalley-Steinberg system as in \cref{Sec: Preliminaries} defines a splitting of the exact sequence
\[
\shortexact{H_E^{ad}}{\Aut\bk{H_E}}{S_E}.
\]
We can then form the semidirect product $H_E\rtimes S_E$.
The centralizer of $S_E$ in $H_E$ is isomorphic to $G$, this gives rise to a dual reductive pair
\[
G\times S_E\hookrightarrow H_E\rtimes S_E .
\]
We denote the associated $\theta$-lift by $\theta_{S_E}$.
It was shown in \cite{MR1918673} that if a cuspidal representation $\pi$ of $G$ lies in the cuspidal image of the $\theta$-lift from $S_E$ then it supports the $\Psi_E$-Fourier coefficient and $\Lfun^S\bk{s,\pi,\chi_E,\st}$ admits a pole of order $n_E$ at $s=2$.
In order to prove the converse, we recall \cite[Proposition 5.1]{MR1918673}.
\begin{Prop}
\label{Prop: Minimal Representation}
For any standard section $f_s$, the Eisenstein series $\Eisen\bk{\chi_E, f_s, s, g}$ has at most a pole of order $n_E$ at $s=\frac{3}{2}$.
A pole of order $n_E$ is attained.
Moreover, the space of automorphic forms
\[
Span_{\C}\set{\bk{s-\frac{3}{2}}^{n_E} \Eisen_E^\ast\bk{\chi_E, s, f_s, g}\bigg{\vert}_{s=\frac{3}{2}}} \ ,
\]
is an irreducible square-integrable automorphic representation isomorphic to the minimal representation
$\Pi_E$ of $H_E$.
\end{Prop}


\begin{Thm}
\label{Conj: Theta lift}
The following are equivalent
\begin{enumerate}
\item
$\pi$ supports the $\Psi_E$-Fourier coefficient and $\Lfun^S\bk{s,\pi,\chi_E,\st}$ admits a pole  at $s=2$ of order $n_E$.
\item
$\theta_{S_E}(\pi)\neq 0$.
\end{enumerate}
\end{Thm}

\begin{proof}
As noted above, it remains to prove that (1) implies (2).
Assume that $\Lfun^S\bk{s,\pi,\chi_E,\st}$ admits a pole of order $n_E$.
According to \cref{Thm:Main theorem}, this implies that there exist $f$ and $\varphi$ such that $\zint_E \bk{\chi, s, \varphi, f}$ admits a pole of order $n_E$.
Taking the residue at \cref{Eq:Zeta Integral} implies the assertion.

\end{proof}

\begin{Remark}
In fact, in an ongoing project I proved that if $\Lfun\bk{s,\pi,\chi_E,\st}$ admits a pole at $s=\frac{3}{2}$ of order $n_E$ then $\pi$ supports the $\Psi_E$-Fourier coefficient and so one may drop this from the assumption in (1).
\end{Remark}

\appendix
\appendixpage
\section{Computing ${F^\ast\bk{\Psi_E,\chi,\cdot,s}\conv P_s\bk{\cdot}}$}
\label{Sec: Convolution}
\subsection{Computing $F^\ast\bk{\Psi_E,\chi,\cdot,s}$}
\label{Subsec: Computing F_s}
In this subsection we prove \cref{Prop: Computing F_s}.

\begin{Remark}
For a triple $r=\bk{r_1,r_3,r_4}$ we write $x_{\alpha}\bk{r}= x_{\alpha_1}\bk{r_1} x_{\alpha_3}\bk{r_3} x_{\alpha_4}\bk{r_4}$.
\end{Remark}

\begin{proof}[Proof of \cref{Prop: Computing F_s}]
Using left $\bk{P_E,\chi_s}$-invariance and right $K$-invariance we have
\begin{align*}
F^\ast\bk{\Psi_E,\chi,g,s}&=\intl_{F} f_s\bk{w_\beta w_\alpha x_\alpha\bk{e} x_{3\alpha+\beta}\bk{r} \check{\alpha}\bk{t_1} \check{\beta}\bk{t_2} x_{\alpha}\bk{d}} \psi\bk{r} dr=\\
&=\chi_s\bk{\frac{t_2^2}{t_1^3}} \intl_{F} f_s\bk{w_\beta x_{\beta}\bk{\frac{t_2}{t_1^3}r} x_{-\alpha}\bk{\frac{t_2}{t_1^2}e+d}} \psi\bk{r} dr .
\end{align*}

We continue to consider this integral case-by-case.
\begin{list}{\underline{\textbf{Case \Roman{qcounter}:}}}{\usecounter{qcounter}}
\item Assume that $\FNorm{\frac{t_2}{t_1^2}a+d}_{F_{\alpha_1}}, \FNorm{\frac{t_2}{t_1^2}b+d}_{F_{\alpha_3}}, \FNorm{\frac{t_2}{t_1^2}c+d}_{F_{\alpha_4}}\leq{1}$.

In this case we have
\begin{align*}
F^\ast\bk{\Psi_E,\chi,g,s}
&=\chi_s\bk{\frac{t_2^2}{t_1^3}} \intl_{F} f_s\bk{w_\beta x_{\beta}\bk{\frac{t_2}{t_1^3}r}} \psi\bk{r} dr.
\end{align*}
Denote $\alpha=\frac{t_1^3}{t_2}$.
We split the integral in two
\begin{align*}
&\intl_{F} f_s\bk{w_\beta x_{\beta}\bk{\frac{t_2}{t_1^3}r}} \psi\bk{r} dr =  \intl_{\alpha\mO} f_s\bk{w_\beta x_{\beta}\bk{\frac{r}{\alpha}}} \psi\bk{r} dr + \intl_{F\setminus\alpha\mO} f_s\bk{w_\beta x_{\beta}\bk{\frac{r}{\alpha}}} \psi\bk{r} dr .
\end{align*}

Considering the first integral yields
\begin{align*}
\intl_{\alpha\mO} f_s\bk{w_\beta x_{\beta}\bk{\frac{r}{\alpha}}} \psi\bk{r} dr=
\intl_{\alpha\mO} \psi\bk{r} dr =
\piece{\FNorm{\alpha}_F,& \FNorm{\alpha}_F\leq 1 \\ 0,& \FNorm{\alpha}_F\more 1} ,
\end{align*}
while the second integral equals
\begin{align*}
\intl_{F\setminus\alpha\mO} f_s\bk{w_\beta x_{\beta}\bk{\frac{r}{\alpha}}} \psi\bk{r} dr
&=\intl_{F\setminus\alpha\mO} f_s\bk{ x_{-\beta}\bk{\frac{r}{\alpha}}} \psi\bk{r} dr=\\
&=\intl_{F\setminus\alpha\mO} f_s\bk{ x_{\beta}\bk{\frac{\alpha}{r}} \check{\beta}\bk{\frac{\alpha}{r}}} \psi\bk{r} dr=\\
&=\chi_s\bk{\alpha} \intl_{F\setminus\alpha\mO} \chi_s\bk{\frac{1}{r}} \psi\bk{r} dr .
\end{align*}
Since $\chi_s$ is $\mO^{\times}$-invariant this equals
\[
\intl_{F\setminus\alpha\mO} \chi_s\bk{\frac{1}{r}} \psi\bk{r} dr=\sum_{j=1-val\bk{\alpha}}^{\infty} \chi_s\bk{\unif^{j}} \intl_{\FNorm{r}=q^j} \psi\bk{r} dr .
\]
Recall that
\[
\intl_{\FNorm{r}=q^j} \psi\bk{r} dr=
\piece{0,& j>1\\ -1,& j=1 \\ q^j\bk{1-q^{-1}},& j\leq 0} .
\]
Inserting this to the previous expression yields
\begin{align*}
\intl_{F\setminus\alpha\mO} \chi_s\bk{\frac{1}{r}} \psi\bk{r} dr=\piece{0,& \FNorm{\alpha}_F>1\\ -\chi_s\bk{\unif},&\FNorm{\alpha}_F=1\\ \bk{1-q^{-1}}\bk{\frac{1-\chi_s\bk{\frac{1}{\alpha}}\FNorm{\alpha}_F}{1-\chi_s\bk{\unif^{-1}}\FNorm{\unif}_F}}-\chi_s\bk{\unif},&  \FNorm{\alpha}_F<1} .
\end{align*}


Combining all of the above yields
\[
F^\ast\bk{\Psi_E,\chi,g,s}=\piece{0,&\FNorm{\alpha}_F>1\\ \chi_s\bk{t_2}\bk{1-\chi_s\bk{\unif}},&\FNorm{\alpha}_F=1\\ 
\chi_s\bk{\frac{t_2}{\alpha}}\bk{\FNorm{\alpha}_F+\chi_s\bk{\alpha}\bk{\bk{1-q^{-1}}\frac{1-\chi_s\bk{\frac{1}{\alpha}}\FNorm{\alpha}_F}{1-\chi_s\bk{\unif^{-1}}\FNorm{\unif}_F}-\chi_s\bk{\unif}}},& \FNorm{\frac{t_1^3}{t_2}}_F<1} .
\]
Note that
\[
\chi_s\bk{\frac{t_2}{\alpha}}\bk{\FNorm{\alpha}_F+\chi_s\bk{\alpha}\bk{\bk{1-q^{-1}}\frac{1-\chi_s\bk{\frac{1}{\alpha}}\FNorm{\alpha}_F}{1-\frac{1}{\chi_s\bk{\unif}q}}-\chi_s\bk{\unif}}} = \chi_s\bk{\frac{t_2}{\alpha}} \frac{\Lfun\bk{s+\frac{3}{2},\chi}}{\Lfun\bk{s+\frac{5}{2},\chi}}
\bk{\FNorm{\alpha}_F-\chi_s\bk{\unif\alpha}q}
\]
and hence we may conclude that
\[
F^\ast\bk{\Psi_E,\chi,g,s}=\piece{0,&\FNorm{\alpha}_F>1\\
\frac{\chi_s\bk{t_2}}{\Lfun\bk{s+\frac{5}{2},\chi}},&\FNorm{\alpha}_F=1\\
\chi_s\bk{\frac{t_2}{\alpha}} \frac{\Lfun\bk{s+\frac{3}{2},\chi}}{\Lfun\bk{s+\frac{5}{2},\chi}}
\bk{\FNorm{\alpha}_F-\chi_s\bk{\unif\alpha}q},& \FNorm{\frac{t_1^3}{t_2}}_F<1}
\]


\item
The case where $\FNorm{\frac{t_2}{t_1^2}a+d}_{F_{\alpha_1}}\more 1$, $\FNorm{\frac{t_2}{t_1^2}b+d}_{F_{\alpha_3}}\leq 1$ and $\FNorm{\frac{t_2}{t_1^2}c+d}_{F_{\alpha_4}}\leq 1$ (or the cases equivalent to it) can not happen under \textbf{(CT)} for $E$ non-split.

\item 
Assume $\FNorm{\frac{t_2}{t_1^2}a+d}_{F_{\alpha_1}}\leq 1$, $\FNorm{\frac{t_2}{t_1^2}b+d}_{F_{\alpha_3}}\more 1$ and $\FNorm{\frac{t_2}{t_1^2}c+d}_{F_{\alpha_4}}\more 1$ (the cases equivalent to it follow similarly).
We then have
\begin{align*}
F^\ast\bk{\Psi_E,\chi,g,s}
&=\chi_s\bk{\frac{t_2^2}{t_1^3}} \intl_{F} f_s\bk{w_\beta x_{\beta}\bk{\frac{t_2}{t_1^3}r} x_{-\alpha_1}\bk{\frac{t_2}{t_1^2}b+d} x_{-\alpha_3}\bk{\frac{t_2}{t_1^2}c+d}} \psi\bk{r} dr=\\
&= \chi_s\bk{\frac{t_2^2}{t_1^3}} \intl_{F} f_s\bk{w_\beta x_{\beta}\bk{\frac{t_2}{t_1^3}r\bk{\frac{1}{\frac{t_2}{t_1^2}b+d}\cdot\frac{1}{\frac{t_2}{t_1^2}c+d}}}} \psi\bk{r} dr
\end{align*}
Denote $\alpha=\frac{t_1^3}{t_2}\bk{\frac{t_2}{t_1^2}b+d}\bk{\frac{t_2}{t_1^2}c+d}$, we then have as before
\[
\intl_{F} f_s\bk{w_\beta x_{\beta}\bk{\frac{t_2}{t_1^3}r\bk{\frac{1}{\frac{t_2}{t_1^2}b+d}\cdot\frac{1}{\frac{t_2}{t_1^2}c+d}}}} \psi\bk{r} dr = \intl_{\alpha\mO}\psi\bk{r} dr + \chi_s\bk{\alpha} \intl_{F\setminus\alpha\mO} \chi_s\bk{\frac{1}{r}} \psi\bk{r} dr .
\]


Repeating the evaluation of the first case yields the assertion.

\item 
Assume $\FNorm{\frac{t_2}{t_1^2}a+d}_{F_{\alpha_1}}\more 1$, $\FNorm{\frac{t_2}{t_1^2}b+d}_{F_{\alpha_3}}\more 1$ and $\FNorm{\frac{t_2}{t_1^2}c+d}_{F_{\alpha_4}}\more 1$.
We then have
\begin{align*}
F^\ast\bk{\Psi_E,\chi,g,s}
&=\chi_s\bk{\frac{t_2^2}{t_1^3}} \intl_{F} f_s\bk{w_\beta x_{\beta}\bk{\frac{t_2}{t_1^3}r} x_{-\alpha}\bk{\frac{t_2}{t_1^2}e+d}} \psi\bk{r} dr=\\
&=\chi_s\bk{\frac{\frac{t_2^2}{t_1^3}}{\Nm\bk{\frac{1}{\frac{t_2}{t_1^2}e+d}}}} \intl_{F} f_s\bk{w_\beta x_{\beta}\bk{\frac{\frac{t_2}{t_1^3}r}{\Nm\bk{\frac{t_2}{t_1^2}e+d}}} x_{\alpha}\bk{\frac{t_2}{t_1^2}e+d}} \psi\bk{r} dr,
\end{align*}
Denote $\alpha=\frac{t_1^3}{t_2}\bk{\frac{t_2}{t_1^2}a+d}\bk{\frac{t_2}{t_1^2}b+d}\bk{\frac{t_2}{t_1^2}c+d}$.
It holds that
\[
\intl_{F} f_s\bk{w_\beta x_{\beta}\bk{\frac{r}{\alpha}} x_{\alpha}\bk{\frac{t_2}{t_1^2}e+d}} \psi\bk{r} dr = \intl_{\alpha\mO}\psi\bk{r} dr + \chi_s\bk{\alpha} \intl_{F\setminus\alpha\mO} \chi_s\bk{\frac{1}{r}} \psi\bk{r} dr .
\]

Repeating the evaluation of the first case yield the assertion once again.
\end{list}

\end{proof}

As explained in \cref{Sec:Unramified Computation}, in what follows we may assume that $\chi=\Id$.
We now obtain a more explicit formula for $F^\ast\bk{\Psi_E,\Id,g,s}$.

\begin{Cor}
\label{Cor: Local factor}
Let $g=h_\alpha\bk{t_1}h_{\beta}\bk{t_2} x_\alpha\bk{d}$.
\begin{itemize}
\item \underline{$E=F\times K$:}
For $d=0$ it holds that
\[
F^\ast\bk{\Psi_E,\Id,g,s} = \piece{
\frac{\zfun\bk{5s-1}}{\zfun\bk{5s}} \FNorm{t_2}_F^{10s-1} \FNorm{t_1}_F^{3-15s} \bk{1-\FNorm{\frac{t_1^3}{t_2}}_F^{5s-1}q^{1-5s}},& \FNorm{\frac{t_2}{t_1^2}}_F\leq 1 \\
\frac{\zfun\bk{5s-1}}{\zfun\bk{5s}} \FNorm{t_2}_F \FNorm{t_1}_F^{5s-1} \bk{1-\FNorm{\frac{t_2}{t_1}}_F^{5s-1}q^{1-5s}},& \FNorm{\frac{t_2}{t_1^2}}_F\more 1} 
\]
and for $d\notin \mO$, it holds that
\[
F^\ast\bk{\Psi_E,\Id,g,s} = \piece{
\frac{\zfun\bk{5s-1}}{\zfun\bk{5s}} \FNorm{t_2}_F \FNorm{\frac{d}{t_1}}_F^{1-5s} \bk{1-\FNorm{\frac{d t_2}{t_1}}_F^{5s-1}q^{1-5s}},& \FNorm{\frac{dt_1^2}{t_2}}_F\leq 1 \& \FNorm{\frac{d t_2}{t_1}}_F\leq 1 \\
\frac{\zfun\bk{5s-1}}{\zfun\bk{5s}} \FNorm{t_2}_F^{10s-1} \FNorm{dt_1}_F^{3-15s} \bk{1-\FNorm{\frac{d^3t_1^3}{t_2}}_F^{5s-1}q^{1-5s}},& \FNorm{\frac{dt_1^2}{t_2}}_F\more 1 \& \FNorm{\frac{d^3t_1^3}{t_2}}_F\leq 1 \\
0,& \FNorm{\frac{dt_1^2}{t_2}}_F\leq 1 \& \FNorm{\frac{d t_2}{t_1}}_F\more 1 \\
& \FNorm{\frac{dt_1^2}{t_2}}_F\more 1 \& \FNorm{\frac{d^3t_1^3}{t_2}}_F\more 1} .
\]
\item \underline{$E$ a field:}
For $d=0$ it holds that
\[
F^\ast\bk{\Psi_E,\Id,g,s} = \piece{
\frac{\zfun\bk{5s-1}}{\zfun\bk{5s}} \FNorm{t_2}_F^{10s-1} \FNorm{t_1}_F^{3-15s} \bk{1-\FNorm{\frac{t_1^3}{t_2}}_F^{5s-1}q^{1-5s}},& \FNorm{\frac{t_2}{t_1^2}}_F\leq 1 \\
\frac{\zfun\bk{5s-1}}{\zfun\bk{5s}} \FNorm{t_2}_F^{2-5s} \FNorm{t_1}_F^{15s-3} \bk{1-\FNorm{\frac{t_2^2}{t_1^3}}_F^{5s-1}q^{1-5s}},& \FNorm{\frac{t_2}{t_1^2}}_F\more 1 }
\]
and for $d\notin \mO$, it holds that
\[
F^\ast\bk{\Psi_E,\Id,g,s} = \piece{
\frac{\zfun\bk{5s-1}}{\zfun\bk{5s}} \FNorm{t_2}_F^{2-5s} \FNorm{t_1}_F^{15s-3} \bk{1-\FNorm{\frac{t_2^2}{t_1^3}}_F^{5s-1}q^{1-5s}},& \FNorm{\frac{dt_1^2}{t_2}}_F\leq 1 \& \FNorm{\frac{t_2^2}{t_1^3}}_F\leq 1 \\
\frac{\zfun\bk{5s-1}}{\zfun\bk{5s}} \FNorm{t_2}_F^{10s-1} \FNorm{d t_1}_F^{3-15s} \bk{1-\FNorm{\frac{d^3t_1^3}{t_2}}_F^{5s-1}q^{1-5s}},& \FNorm{\frac{dt_1^2}{t_2}}_F\more 1 \& \FNorm{\frac{d^3t_1^3}{t_2}}_F\leq 1 \\
0,& \FNorm{\frac{dt_1^2}{t_2}}_F\leq 1 \& \FNorm{\frac{t_2^2}{t_1^3}}_F\more 1 \\
& \FNorm{\frac{dt_1^2}{t_2}}_F\more 1 \& \FNorm{\frac{d^3t_1^3}{t_2}}_F\more 1} .
\]
\end{itemize}
\end{Cor}

\subsection{Computing ${F^\ast\bk{\Psi_E,\Id,\cdot,s}\conv P_s\bk{\cdot}}$}
\label{Subsec: Convolution}
We recall from \cref{Sec:Unramified Computation} that
\[
P_s= \frac{P_0\bk{q^{-s-\frac{1}{2}}}A_0-P_1\bk{q^{-s-\frac{1}{2}}}A_1}
{\zfun\bk{s+\frac{3}{2}}\zfun\bk{s+\frac{7}{2}}\zfun\bk{s+\frac{1}{2}}} .
\]
As $A_0=\Id_{K}$ and $A_1=q^{-3}\bk{\Id_{K}+\Id_{K\omega_1\bk{\unif}K}}$, in order to compute ${F^\ast\bk{\Psi_E,\Id,\cdot,s}\conv P_s\bk{\cdot}}$ we need only to compute $F^\ast\bk{\Psi_E,\Id,\cdot,s}\conv \Id_{K\omega_1\bk{\unif}K}$.
For any $f\in \M_{\Psi_E}$ it holds that
\[
f\conv \Id_{K\omega_1\bk{\unif}K}\bk{g} = \intl_{K\omega_1\bk{\unif}K} f\bk{gh} dh = \suml_{\gamma\in \bk{K\omega_1\bk{\unif}K}\rmod K} f\bk{g\gamma} .
\]
Let $g=h_\alpha\bk{t_1}h_{\beta}\bk{t_2} x_\alpha\bk{d}$ and assume that if $d\in\mO$ then $d=0$ (we may do this do to right $K$-invariance).
Fixing the list of representatives of $\bk{K\omega_1\bk{\unif}K}\rmod K$ obtained in \cite[Appendix A.3]{MR3284482} we obtain
\begin{align}
\label{eq: convolution with characteristic function}
\bk{F^\ast\bk{\Psi_E,\Id,\cdot,s}\conv \Id_{K\omega_1\bk{\unif}K}}\bk{g}=& F_s\bk{\bk{\frac{t_1}{\unif},\frac{t_2}{\unif^2}}x_\alpha\bk{d}} + q^6 F_s\bk{\bk{\unif t_1,\unif^2 t_2}x_\alpha\bk{d}}+\\
&q F_s\bk{\bk{\frac{t_1}{\unif},\frac{t_2}{\unif}}x_\alpha\bk{\unif d}} + q \suml_{s\in \mO\rmod \bk{\unif}} F_s\bk{\bk{t_1,\frac{t_2}{\unif}}x_\alpha\bk{\frac{d-s}{\unif}}} + \nonumber \\
& q^4 F_s\bk{\bk{t_1,\unif t_2}x_\alpha\bk{\unif d}} + q^4 \suml_{s\in\mO\rmod\bk{\unif}} F_s\bk{\bk{\unif t_1,\unif t_2}x_\alpha\bk{\frac{d-s}{\unif}}}+ \nonumber\\
&\bk{GS\bk{\Psi_E,g}} F_s\bk{\bk{t_1,t_2}x_\alpha\bk{d}} \nonumber,
\end{align}
where $GS$ denotes the following \emph{Gaussian sum}:
\begin{equation}
\label{eq: Gaussian sums}
GS\bk{\Psi_E,g}=\bk{q-1}+ q\suml_{\stackrel{r\in\mO\rmod \bk{\unif}}{r\not\equiv 0}} \Psi_E\bk{g x_\beta\bk{\frac{r}{\unif}} g^{-1}} + q \suml_{\stackrel{r,y\in\mO\rmod \bk{\unif}}{r\not\equiv 0}} 
\Psi_E\bk{g u\bk{\frac{y^3 r}{\unif},\frac{y^2 r}{\unif},\frac{y r}{\unif},\frac{r}{\unif}}^{-1}  g^{-1}} .
\end{equation}

\begin{Lem}
\begin{itemize}
\item \underline{$E=F\times K$:}
For $d=0$ it holds that
\[
GS\bk{\Psi_E,g} = \piece{q^3-1,& \FNorm{\frac{t_2}{t_1^2}}\leq 1 \ \& \ \FNorm{\frac{t_1^3}{t_2}}<1 \text{ or }\FNorm{\frac{t_2}{t_1^2}}>1 \ \& \ \FNorm{\frac{t_1^3}{t_2}}<1 \\
-1,& \FNorm{\frac{t_2}{t_1^2}}\leq 1 \ \& \ \FNorm{\frac{t_1^3}{t_2}}=1\\
q^2-1.& \FNorm{\frac{t_2}{t_1^2}}>1 \ \& \ \FNorm{\frac{t_1^3}{t_2}}=1}
\]
and for $d\notin \mO$, it holds that
\[
GS\bk{\Psi_E,g} = \piece{q^3-1,& \FNorm{\frac{dt_1^2}{t_2}}\leq 1 \ \& \ \FNorm{\frac{dt_2}{t_1}}<1 \text{ or } \FNorm{\frac{dt_1^2}{t_2}}> 1 \ \& \ \FNorm{\frac{d^3 t_1^3}{t_2}}<1 \\
-1,& \FNorm{\frac{dt_1^2}{t_2}}\leq 1 \ \& \  \FNorm{\frac{dt_2}{t_1}}=1 \text{ or } \FNorm{\frac{dt_1^2}{t_2}}> \ 1 \& \ \FNorm{\frac{d^3 t_1^3}{t_2}}=1 \\
q-1,& \FNorm{\frac{dt_1^2}{t_2}}\leq 1 \ \& \ \FNorm{\frac{dt_2}{t_1}}>1 \text{ or } \FNorm{\frac{dt_1^2}{t_2}}> 1 \ \& \ \FNorm{\frac{d^3 t_1^3}{t_2}}>1 \\} .
\]
\item \underline{$E$ a field:}
For $d=0$ it holds that
\[
GS\bk{\Psi_E,g} = \piece{q^3-1,& \FNorm{\frac{t_2}{t_1^2}}\leq 1 \ \& \ \FNorm{\frac{t_1^3}{t_2}}<1 \text{ or } \FNorm{\frac{t_2}{t_1^2}}> 1 \ \& \ \FNorm{\frac{t_2^2}{t_1^3}}<1 \\
-1,& \FNorm{\frac{t_2}{t_1^2}}< 1 \ \& \ \FNorm{\frac{t_1^3}{t_2}}=1 \text{ or } \FNorm{\frac{t_2}{t_1^2}}> 1 \ \& \ \FNorm{\frac{t_2^2}{t_1^3}}=1 \\
-q^2-1,& \FNorm{\frac{t_2}{t_1^2}}=1 \ \& \ \FNorm{\frac{t_2^2}{t_1^3}}=1 \\}
\]
and for $d\notin \mO$, it holds that
\[
GS\bk{\Psi_E,g} = \piece{q^3-1,& \FNorm{\frac{dt_1^2}{t_2}}\leq 1 \ \& \ \FNorm{\frac{t_2^2}{t_1^3}}<1 \text{ or } \FNorm{\frac{dt_1^2}{t_2}}> 1 \ \& \ \FNorm{\frac{d^3 t_1^3}{t_2}}<1 \\
-1,& \FNorm{\frac{dt_1^2}{t_2}}\leq 1 \ \& \ \FNorm{\frac{t_2^2}{t_1^3}}=1 \text{ or } \FNorm{\frac{dt_1^2}{t_2}}> 1 \ \& \ \FNorm{\frac{d^3 t_1^3}{t_2}}=1 \\
q-1,& \FNorm{\frac{dt_1^2}{t_2}}\leq 1 \ \& \ \FNorm{\frac{t_2^2}{t_1^3}}>1 \text{ or } \FNorm{\frac{dt_1^2}{t_2}}> 1 \ \& \ \FNorm{\frac{d^3 t_1^3}{t_2}}>1 \\} .
\]
\end{itemize}
\end{Lem}

\begin{proof}
We prove this lemma only for the case where $d=0$, the case $d\notin \mO$ is proven in a similar way.
The proof relies on the fact that
\[
\suml_{\stackrel{r\in\mO\rmod \bk{\unif}}{r\not\equiv 0}} \psi\bk{\unif^n x} = \piece{0,& n<-1 \\ -1,& n=-1 \\ q-1,& n\geq 0}
\]

We first deal with the first summand in \cref{eq: Gaussian sums}
\[
\suml_{\stackrel{r\in\mO\rmod \bk{\unif}}{r\not\equiv 0}} \Psi_E\bk{g x_\beta\bk{\frac{r}{\unif}} g^{-1}} = \suml_{\stackrel{r\in\mO\rmod \bk{\unif}}{r\not\equiv 0}} \psi\bk{N_{a,b,c}\frac{t_2^2}{t_1^3}\frac{r_1}{\unif}} .
\]

\begin{itemize}
\item \underline{$E=F\times K$:}
In this case $N_{a,b,c}=0$ due to \textbf{(CT)} and hence
\[
\suml_{\stackrel{r\in\mO\rmod \bk{\unif}}{r\not\equiv 0}} \Psi_E\bk{g x_\beta\bk{\frac{r}{\unif}} g^{-1}} = q-1 .
\]

\item \underline{$E$ a field:}
In this case $N_{a,b,c}\in\mO^\times$ and hence
\[
\suml_{\stackrel{r\in\mO\rmod \bk{\unif}}{r\not\equiv 0}} \Psi_E\bk{g x_\beta\bk{\frac{r}{\unif}} g^{-1}} = \piece{q-1,& \FNorm{\frac{t_2^2}{t_1^3}}\less 1 \\
-1,& \FNorm{\frac{t_2^2}{t_1^3}} = 1\\ } .
\]
Here we assumed that $\FNorm{\frac{t_2^2}{t_1^3}}\leq 1$ in accordance with \cref{Rem: Conditions on m - toral}.
\end{itemize}

We now compute the second summand in \cref{eq: Gaussian sums}
\begin{align*}
&\suml_{\stackrel{r,y\in\mO\rmod \bk{\unif}}{r\not\equiv 0}} 
\Psi_E\bk{g u\bk{\frac{y^3 r}{\unif},\frac{y^2 r}{\unif},\frac{y r}{\unif},\frac{r}{\unif}}^{-1}  g^{-1}} = \\
&\suml_{\stackrel{r,y\in\mO\rmod \bk{\unif}}{r\not\equiv 0}} \psi\bk{\frac{r_1}{\unif}\bk{N_E\frac{t_2^2}{t_1^3}y^3+D_E\frac{t_2}{t_1}y^2 + \frac{t_1^3}{t_2}}} = \\
& \suml_{\stackrel{r\in\mO\rmod \bk{\unif}}{r\not\equiv 0}} \psi\bk{\frac{r_1}{\unif}\frac{t_1^3}{t_2}}
+\suml_{\stackrel{y\in\mO\rmod \bk{\unif}}{y\not\equiv 0}} \suml_{\stackrel{r\in\mO\rmod \bk{\unif}}{r\not\equiv 0}} \psi\bk{\frac{r_1}{\unif}\bk{N_E\frac{t_2^2}{t_1^3}y^3+D_E\frac{t_2}{t_1}y^2 + \frac{t_1^3}{t_2}}} .
\end{align*}
We first have
\[
\suml_{\stackrel{r\in\mO\rmod \bk{\unif}}{r\not\equiv 0}} \psi\bk{\frac{r_1}{\unif}\frac{t_1^3}{t_2}} = \piece{
q-1,& \FNorm{\frac{t_1^3}{t_2}}\less 1 \\
-1,& \FNorm{\frac{t_1^3}{t_2}} = 1\\
} ,
\]
Where we assumed that $\FNorm{\frac{t_2^2}{t_1^3}}\leq 1$ in accordance with \cref{Rem: Conditions on m - toral}.

\begin{itemize}
\item \underline{$E=F\times K$:}
In this case $N_{a,b,c}=0$.
For $y\not\equiv 0$ it then holds that
\[
\FNorm{\frac{r_1}{\unif}\bk{N_E\frac{t_2^2}{t_1^3}y^3+D_E\frac{t_2}{t_1}y^2 + \frac{t_1^3}{t_2}}} = \FNorm{\frac{t_1^3}{t_2}\bk{D_E\bk{\frac{t_2}{t_1}y}^2 + 1}} = \piece{
\FNorm{\frac{t_1^3}{t_2}},& \FNorm{\frac{t_2}{t_1^2}}\leq 1\\
\FNorm{\frac{t_2}{t_1}},& \FNorm{\frac{t_2}{t_1^2}}\more 1 }
\]
and hence
\[
\suml_{\stackrel{r\in\mO\rmod \bk{\unif}}{r\not\equiv 0}} \psi\bk{\frac{r_1}{\unif}\bk{N_E\frac{t_2^2}{t_1^3}y^3+D_E\frac{t_2}{t_1}y^2 + \frac{t_1^3}{t_2}}} = \piece{
q-1,& \FNorm{\frac{t_2}{t_1^2}}\leq 1 \& \FNorm{\frac{t_1^3}{t_2}}\less 1 \\
& \FNorm{\frac{t_2}{t_1^2}}\more 1 \& \FNorm{\frac{t_2}{t_1}}\less 1 \\
-1,& \FNorm{\frac{t_2}{t_1^2}}\leq 1 \& \FNorm{\frac{t_1^3}{t_2}}= 1 \\
& \FNorm{\frac{t_2}{t_1^2}}\more 1 \& \FNorm{\frac{t_2}{t_1}}= 1 \\
} ,
\]

\item \underline{$E$ a field:}
We further assume that $\FNorm{\frac{t_1^3}{t_2}} \leq \FNorm{\frac{t_2^2}{t_1^3}}\leq 1$ in accordance with \cref{Rem: Conditions on m - toral}.
In this case $N_{a,b,c}\in\mO$.
For $y\not\equiv 0$ it then holds that
\[
\suml_{\stackrel{r\in\mO\rmod \bk{\unif}}{r\not\equiv 0}} \psi\bk{\frac{r_1}{\unif}\bk{N_E\frac{t_2^2}{t_1^3}y^3+D_E\frac{t_2}{t_1}y^2 + \frac{t_1^3}{t_2}}} = 
\FNorm{\frac{t_1^3}{t_2}\bk{N_E\bk{\frac{t_2}{t_1^2}y}^3+ D_E\bk{\frac{t_2}{t_1^2}y}^2 + 1}} = \piece{
\FNorm{\frac{t_1^3}{t_2}},& \FNorm{\frac{t_2}{t_1^2}}\leq 1 \\
\FNorm{\frac{t_2^2}{t_1^3}},& \FNorm{\frac{t_2}{t_1^2}}\more 1}
\]
and hence
\[
\suml_{\stackrel{r_1 \in A}{r_1\not\equiv 0}} \psi\bk{\frac{t_1^3}{t_2}\frac{r_1}{\unif} \bk{N_E\bk{\frac{t_2}{t_1^2}y}^3+ D_E\bk{\frac{t_2}{t_1^2}y}^2 + 1}} = \piece{
q-1,& \FNorm{\frac{t_2}{t_1^2}}\leq 1 \& \FNorm{\frac{t_1^3}{t_2}}\less 1\\
& \FNorm{\frac{t_2}{t_1^2}}\more 1 \& \FNorm{\frac{t_2^2}{t_1^3}}\less 1 \\
-1,& \FNorm{\frac{t_2}{t_1^2}}\leq 1 \& \FNorm{\frac{t_1^3}{t_2}}= 1\\
& \FNorm{\frac{t_2}{t_1^2}}\more 1 \& \FNorm{\frac{t_2^2}{t_1^3}}= 1 \\
}, 
\]
\end{itemize}
Combining all of the above yields the assertion.

\end{proof}

We now turn to compute ${F^\ast\bk{\Psi_E,\Id,\cdot,s}\conv P_s\bk{\cdot}}\bk{g}$.
As this is computation boils down to inserting \cref{Cor: Local factor} and \cref{eq: Gaussian sums} into \cref{eq: convolution with characteristic function} we compute this only for $g=1$.

Denote $P=P_0\bk{q^{-s-\frac{1}{2}}}$ and $Q=P_1\bk{q^{-s-\frac{1}{2}}}$.
\begin{itemize}
\item
\underline{$E=F\times K$:}
In this case
\begin{align*}
&\bk{F^\ast\bk{\Psi_E,\cdot,s}\conv P_s\bk{\cdot}}\bk{1}= \\
& \frac{\zfun_F\bk{s+\frac{5}{2}} \zfun_K\bk{s+\frac{3}{2}} \zfun_F\bk{2s+1} }{\zfun_F\bk{s+\frac{3}{2}} \zfun_F\bk{s+\frac{7}{2}} \zfun_F\bk{s+\frac{1}{2}} } \coset{\frac{P-q^{-3}Q}{\zfun_F\bk{s+\frac{5}{2}}}-\frac{q^{-3}Q}{\zfun_F\bk{s+\frac{5}{2}}}\bk{\frac{q^{-s-\frac{5}{2}}\zfun_F\bk{s+\frac{3}{2}}}{\zfun_F\bk{2s-1}} + q^{-\bk{s-\frac{3}{2}}} -1 }} = \frac{1}{\zfun_F\bk{s+\frac{7}{2}}}.
\end{align*}

\item
\underline{$E$ cubic Galois field extension of $F$:}
In this case
\begin{align*}
&\bk{F^\ast\bk{\Psi_E,\cdot,s}\conv P_s\bk{\cdot}}\bk{1}= \\
& \frac{\zfun_F\bk{s+\frac{5}{2}} \zfun_E\bk{s+\frac{3}{2}} \zfun_F\bk{2s+1} }{\zfun_F\bk{s+\frac{3}{2}}^2\zfun_F\bk{s+\frac{7}{2}}\zfun_F\bk{s+\frac{1}{2}} }
\coset{\frac{P-q^{-3}Q}{\zfun_F\bk{s+\frac{5}{2}}}-\frac{q^{-3}Q}{\zfun_F\bk{s+\frac{5}{2}}}\bk{\frac{q^{-s-\frac{5}{2}}\zfun_F\bk{s+\frac{3}{2}}}{\zfun_F\bk{2s-1}} -q2-1}} = \frac{1}{\zfun_F\bk{s+\frac{3}{2}}\zfun_F\bk{s+\frac{7}{2}}}.
\end{align*}

\end{itemize}

\section{Computation of $D^{\Psi_E}_s$}
\label{Sec: Fourier Transform}
In this section we compute the $\Psi_E$-Fourier coefficient of $D_s$.
Since \cref{Thm:Unramified Computation} was already proved when $E=F\times F\times F$, we restrict ourselves to the assumption that $E$ is a non-split Galois \'etale cubic algebra over $F$.
\begin{Thm}
\label{Thm:Fourier Transform}
$D^{\Psi_E}_s\bk{g} = 0$ unless $g\in UTK$.
If $g=h_\alpha\bk{t_1}h_\beta\bk{t_2}$ it holds that
\begin{itemize}
\item
If $N_{a,b,c}=0$ then
\[
D_s^{\Psi_E}\bk{g}=\piece{\frac{q^{-\bk{s+\frac{7}{2}}n}}{\zfun_F\bk{s+\frac{7}{2}}} , & \FNorm{\frac{t_1^2}{t_2}}=1 \\ 
0,& \FNorm{\frac{t_1^2}{t_2}}\more 1 \\
\frac{q^{2n-m-\bk{s+\frac{7}{2}}n}}{\zfun_F\bk{s+\frac{3}{2}} \zfun_F\bk{s+\frac{7}{2}}},& \FNorm{\frac{t_1^2}{t_2}}\less 1}
\]

\item
If $N_{a,b,c}\in\mO^\times$ then
\[
D_s^{\Psi_E}\bk{g}=\piece{\frac{q^{-\bk{s+\frac{7}{2}}n}}{\zfun_F\bk{s+\frac{3}{2}}\zfun_F\bk{s+\frac{7}{2}}} , & \FNorm{\frac{t_1^2}{t_2}}=1 \\ 
0,& \FNorm{\frac{t_1^2}{t_2}}\more 1 \\
0,& \FNorm{\frac{t_1^2}{t_2}}\less 1}
\]

\end{itemize}
\end{Thm}

Before performing a direct computation of $D^{\Psi_E}_s\bk{g}$ we make some preparations.
We let $SO_7$ be the special orthogonal group that preserves the split symmetric form $\bk{\delta_{i,7-i}}$ viewed as a subgroup of $GL_7$.
We fix an embedding $\iota: SO_7\bk{F}\hookrightarrow G\bk{F}$ as in \cite{MR1617425}.
In particular,

For $h=u\bk{r_1,r_2,r_3,r_4,r_5}h_\alpha(t_1)h_\beta(t_2) x_{\alpha}\bk{d}$ it holds that
\begin{equation}
\iota\bk{h}=
\begin{pmatrix}
1 & 0 & r_2 & r_3 & \frac{-r_4}{2} & \frac{r_2r_3+r_5}{2} & \frac{r_2r_4-r_3^2}{2} \\
0 & 1 & r_1 & r_2& \frac{-r_3}{2} & \frac{r_1r_3-r_2^2}{2} & \frac{r_1r_4-2r_2r_3-r_5}{2} \\
0 & 0 & 1 & 0 & 0 & \frac{r_3}{2} & \frac{r_4}{2} \\
0 & 0 & 0 & 1 & 0 & -r_2 & -r_3 \\
0 & 0 & 0 & 0 & 1 & -r_1 & -r_2 \\
0 & 0 & 0 & 0 & 0 & 1 & 0 \\
0 & 0 & 0 & 0 & 0 & 0 & 1
\end{pmatrix}
\cdot
\begin{pmatrix}
t_1\\
&\frac{t_2}{t_1}\\
&&\frac{t_1^2}{t_2}\\
&&&1\\
&&&&\frac{t_2}{t_1^2}\\
&&&&&\frac{t_1}{t_2}\\
&&&&&&\frac{1}{t_1}
\end{pmatrix}
\cdot
\begin{pmatrix}
1 & d & 0 & 0 & 0 & 0 & 0 \\
0 & 1 & 0 & 0 & 0 & 0 & 0 \\
0 & 0 & 1 & -d & -\frac{d^2}{2} & 0 & 0 \\
0 & 0 & 0 & 1 & d & 0 & 0 \\
0 & 0 & 0 & 0 & 1 & 0 & 0 \\
0 & 0 & 0 & 0 & 0 & 1 & -d \\
0 & 0 & 0 & 0 & 0 & 0 & 1
\end{pmatrix} \ .
\end{equation}

The following results are simple to check and will be useful in what follows.
\begin{Lem}
The function $\Gamma: G\bk{F}\rightarrow \R$ given by
\[
\Gamma\bk{g}=\max_{1\le i,j\le 7} \FNorm{\iota(g)_{i,j}}
\]
is a bi-$K$-invariant function and for $t\in T^{+}$ it satisfies
\[
\Gamma\bk{t}=\FNorm{\omega_1(t)}^{-1} .
\]
\end{Lem}

Thus, we may write $D\bk{g,s}=\sum_{k=0}^\infty D_k\bk{g}q^{-\bk{s+\frac{7}{2}}k}$, where 
\[
D_k\bk{g}= \piece{1, & \Gamma(g)=q^k \\ 0, & {\rm otherwise}} \ .
\]
For any $g\in{G}$ define $U_k\bk{g}=\set{u\in U: \Gamma(ug)\le q^k}$ and let
\[
E_k\bk{g} = \piece{1,& \Gamma\bk{g}\leq q^k \\ 0,& \text{otherwise}} .
\]
Obviously
\[
D_k\bk{g}=E_k\bk{g}-E_{k-1}\bk{g}
\]
and in particular
\[
D^{\Psi_E}_s\bk{g,s}=\sum_{k=0}^{\infty} \bk{E_k^{\Psi_E}\bk{g}-E_{k-1}^{\Psi_E}\bk{g}} q^{-\bk{s+\frac{7}{2}}k} .
\]
Hence, in order to compute $D_s^{\Psi_E}\bk{g}$ we compute
\[
E_k^{\Psi_E}\bk{g}=
\int_{U_k\bk{g}} \Psi_E\bk{u} du .
\]

\begin{Lem}
\label{lemma:measuring lemma}
For $n_1,n_2,n_3\in\N$ with $n_1+n_2\ge n_3$ let
\[
\kappa\bk{n_1,n_2,n_3}:=\meas\set{\bk{x,y}\in F^2 \mvert \FNorm{x}\leq{q^{n_1}},\quad \FNorm{y}\leq{q^{n_2}},\quad \FNorm{xy}\leq{q^{n_3}}}
\]
it holds that
\[
\kappa\bk{n_1,n_2,n_3}:= q^{n_3}\bk{1+\bk{n_1+n_2-n_3}\bk{1-q^{-1}}} \ ,
\]
where $\meas$ is the Haar measure on $G$ such that $\meas\bk{K}=1$. 
\end{Lem}

For $n_1,n_2\in\N$ let
\begin{align*}
&\kappa_{\bk{a,b,c}}^{\bk{n_1}} = \meas\set{\bk{x,y}\in F^2\mvert \FNorm{x},\FNorm{y}, \FNorm{P_{\bk{a,b,c}}\bk{\frac{x}{y}}}\leq q^{n_1}} \\
&\kappa_{\bk{a,b,c}}^{\bk{n_1}}\bk{q^{n_2}} = \meas\set{\bk{x,y}\in F^2\mvert \FNorm{x}=\FNorm{y}=q^{n_2},  \FNorm{P_{\bk{a,b,c}}\bk{\frac{x}{y}}}\leq q^{n_1-n_2\deg\bk{P_{\bk{a,b,c}}}}},
\end{align*}
where
\[
P_{\bk{a,b,c}}\bk{z} = \piece{D_{\bk{a,b,c}}+z^2,& N_{\bk{a,b,c}=0} \\ 
N_{\bk{a,b,c}}+D_{\bk{a,b,c}}z+z^3,& N_{\bk{a,b,c}\in\mO^\times}} .
\]

\begin{Lem}
If $E$ is non-split then $\kappa_{a,b,c}^{\bk{1}} = 0$ and $\kappa_{a,b,c}^{\bk{1}}\bk{q} = 0$.
\end{Lem}

\begin{Lem}
\begin{itemize}
\item
If $E=F\times K$ then
\[
\intl_{\bk{\unif^{-1}\mO^\times}^2} \psi\bk{D_{\bk{a,b,c}}r_2+\frac{r_3^2}{r_2}} dr_2\ dr_3=1-q.
\]
\item
If $E$ is a field then
\[
\intl_{\bk{\unif^{-1}\mO^\times}^2} \psi\bk{N_{\bk{a,b,c}}\frac{r_2^2}{r_3}+D_{\bk{a,b,c}} r_2+\frac{r_3^2}{r_2}} dr_2\ dr_3 =1-q .
\]

\item
$\intl_{\FNorm{x}, \FNorm{y}, \FNorm{xy}\leq q} \psi\bk{x+y} dx\ dy =-1$.
\end{itemize}
\end{Lem}

\begin{Remark}
The difference between the different types of \'etale cubic algebras boils down to the last two lemmas as the results for $E$ split would be different.
\end{Remark}


\begin{Lem}
For $g\in G\bk{F}$ assume that $E_k\bk{g}=0$ for any $k\neq n,n+1$. Then
\[
D_s^{\Psi_E}\bk{g} = 
\frac{q^{-\bk{s+\frac{7}{2}}n}}{\zfun_F\bk{s+\frac{7}{2}}} \bk{E_n\bk{g}+E_{n+1}\bk{g}q^{-s-\frac{7}{2}}} .
\]
\end{Lem}

Let $g=h_\alpha(t_1)h_\beta(t_2) x_{\alpha}\bk{d} = x_{\alpha}\bk{p}  h_\alpha(t_1)h_\beta(t_2)$, where $p=\frac{dt_1^2}{t_2}$.
Denote $\FNorm{t_1}=q^{-n}$, $\FNorm{t_2}=q^{-m}$ and $\FNorm{p}=q^{l}$, we have $u\bk{r_1,r_2,r_3,r_4,r_5} \in U_k\bk{g}$ if and only if
\begin{align*}
&1\leq q^{k+n}\\
&1,\FNorm{p}\leq q^{k+m-n}\\
&1,\FNorm{r_1},\FNorm{r_2}\leq q^{k+2n-m}\\
&1,\FNorm{p},\FNorm{r_3-pr_2},\FNorm{r_2-pr_1}\leq q^k\\
&1,\FNorm{p},\FNorm{p^2},\FNorm{2pr_3-p^2r_2-r_4},\FNorm{2pr_2-r_3-p^2r_1}\leq q^{k+m-2n}\\
&1,\FNorm{r_1},\FNorm{r_2},\FNorm{r_3},\FNorm{r_2r_3+r_5},\FNorm{r_2^2-r_1r_3}\leq q^{k+n-m}\\
&1,\FNorm{p},\FNorm{pr_1-r_2},\FNorm{pr_2-r_3},\FNorm{pr_3-r_4},\FNorm{r_2r_4-r_3^2-pr_2r_3-pr_5}\leq q^{k-n}\\
&\FNorm{r_1r_4-2r_2r_3+pr_2^2-pr_1r_3-r_5}\leq q^{k-n} .
\end{align*}

\subsection{Toral elements}
In this section we consider the case of $g\in UTK$.
Since $E_k^{\Psi_E}\in\M_{\Psi_E}$ it is sufficient to consider $g=h_\alpha(t_1)h_\beta(t_2)$.
Let $\FNorm{t_1}=q^{-n}$ and $\FNorm{t_2}=q^{-m}$.
\begin{Remark}
\label{Rem: Conditions on m - toral}
From \cref{Lem: Conditions on m} we have $D_s^{\Psi_E}\bk{g}=0$ unless
\begin{equation}
\piece{t_1,t_2,\frac{t_2}{t_1},\frac{t_1^3}{t_2}\in\mO,& N_{\bk{a,b,c}}=0 \\ 
t_1,t_2,\frac{t_2}{t_1},\frac{t_1^3}{t_2},\frac{t_2^2}{t_1^3}\in\mO,& N_{\bk{a,b,c}}\in\mO^\times} .
\end{equation}
Thus we may assume that $\FNorm{t_2} \leq \FNorm{t_1} \leq 1$.

Furthermore, $U_k\bk{g}=\emptyset$ unless $k\geq n,m-n$ and we have $u\bk{r_1,r_2,r_3,r_4,r_5}\in U_k\bk{g}$ if and only if
\begin{align*}
&\FNorm{r_1}, \FNorm{r_2}, \FNorm{r_3}, \FNorm{r_1r_3-r_2^2}, \FNorm{r_2r_3+r_5} \leq q^{k+n-m}\\
&\FNorm{r_2}, \FNorm{r_3}, \FNorm{r_4}, \FNorm{r_2r_4-r_3^2}, \FNorm{r_1r_4-2r_2r_3-r_5} \leq q^{k-n} .
\end{align*}
\end{Remark}

\begin{Remark}
As an analogue of \cite[Lemma B.2]{MR3284482}, note that
\[
\intl_{U_k\bk{g}}\Psi_E\bk{u} du = \intl_{\widehat{U_k\bk{g}}}\Psi_E\bk{u} du,
\]
where
\[
\widehat{U_k\bk{g}} = \set{u\bk{r_1,r_2,r_3,r_4,r_5}\in U_k\bk{g} \mvert \FNorm{r_4+D_{\bk{a,b,c}}r_2-N_{\bk{a,b,c}}r_1}\leq q} .
\]
\end{Remark}

We now split the computation to two cases, $N_{\bk{a,b,c}}=0$ or $N_{\bk{a,b,c}}\in\mO^\times$.

\subsubsection{$N_{\bk{a,b,c}}=0$}


\paragraph{\underline{Case of $m<2n$}:}
We write
\[
z:=r_1r_4-2r_2r_3-r_5 .
\]
In this case, $u\bk{r_1,r_2,r_3,r_4,r_5}\in U_k\bk{g}$ if and only if
\begin{align*}
&\FNorm{r_1}, \FNorm{r_1r_3-r_2^2}, \FNorm{r_1r_4-r_2r_3} \leq q^{k+n-m}\\
&\FNorm{r_2}, \FNorm{r_3}, \FNorm{r_4}, \FNorm{r_2r_4-r_3^2}, \FNorm{z} \leq q^{k-n} .
\end{align*}

\begin{itemize}
\item
\underline{$k=n$:}
In this case , $u\bk{r_1,r_2,r_3,r_4,r_5}\in U_k\bk{g}$ if and only if
\begin{align*}
&\FNorm{r_1}\leq q^{2n-m}\\
&\FNorm{r_2}, \FNorm{r_3}, \FNorm{r_4}, \FNorm{z} \leq 1 .
\end{align*}
Hence
\[
E_n^{\Psi_E}\bk{g} = \intl_{\unif^{m-2n}\mO} dr_1 \intl_{\mO^4} \psi\bk{r_4+D_{\bk{a,b,c}r_2}} dr_2\ dr_3\ dr_4\ dz = q^{2n-m} .
\]

\item
\underline{$k=n+1$:}
In this case , $u\bk{r_1,r_2,r_3,r_4,r_5}\in U_k\bk{g}$ if and only if
\begin{align*}
&\FNorm{r_1}, \FNorm{r_1r_3}, \FNorm{r_1r_4} \leq q^{2n-m+1}\\
&\FNorm{r_2}, \FNorm{r_3}, \FNorm{r_4}, \FNorm{r_2r_4-r_3^2}, \FNorm{z} \leq q .
\end{align*}
Note that since $\FNorm{r_2}, \FNorm{r_3}, \FNorm{r_4}, \FNorm{r_2r_4-r_3^2}, \FNorm{z} \leq q$, if $\FNorm{r_3}=q$ then also $\FNorm{r_2}=\FNorm{r_4}=q$.

\begin{itemize}
\item
\underline{$\FNorm{r_3},\FNorm{r_4}\leq 1$:}
In this case, $u\bk{r_1,r_2,r_3,r_4,r_5}\in U_k\bk{g}$ if and only if
\begin{align*}
&\FNorm{r_1} \leq q^{2n-m+1}\\
&\FNorm{r_2}, \FNorm{z} \leq q .
\end{align*}

\item
\underline{$\FNorm{r_3}\leq 1,\FNorm{r_4}=q$:}
In this case, $u\bk{r_1,r_2,r_3,r_4,r_5}\in U_k\bk{g}$ if and only if
\begin{align*}
&\FNorm{r_1} \leq q^{2n-m}\\
&\FNorm{r_2} \leq 1 \\
&\FNorm{z} \leq q .
\end{align*}

\item
\underline{$\FNorm{r_3}=q$:}
In this case, $u\bk{r_1,r_2,r_3,r_4,r_5}\in U_k\bk{g}$ if and only if
\begin{align*}
&\FNorm{r_1} \leq q^{2n-m}\\
&\FNorm{r_2}, \FNorm{r_2r_4-r_3^2}, \FNorm{z} \leq q .
\end{align*}
We make a change of variables
\[
x=r_2r_4-r_3^2
\]
and then $u\bk{r_1,r_2,r_3,r_4,r_5}\in U_k\bk{g}$ if and only if
\begin{align*}
&\FNorm{r_1} \leq q^{2n-m}\\
&\FNorm{x}, \FNorm{z} \leq q .
\end{align*}
\end{itemize}

In conclusion,
\begin{align*}
E_{n+1}^{\Psi_E}\bk{g} =& q^{2n-m+2} \intl_{\mO}\psi\bk{r_4} dr_4 \intl_{\unif^{-1}\mO}\psi\bk{D_{\bk{a,b,c}}r_2} dr_2 + \\
&q^{2n-m+1}\intl_{\mO}\psi\bk{D_{\bk{a,b,c}}r_2} dr_2 \intl_{\unif^{-1}\mO^\times}\overline{\psi\bk{r_4}} dr_4 + \\
&q^{2n-m+1} \intl_{\bk{\unif^{-1}\mO^\times}^2} \psi\bk{D_{\bk{a,b,c}}r_2+\frac{r_3^2}{r_2}} dr_2\ dr_3 = -q^{2n-m+2} .
\end{align*}

\item
\underline{$k>n+1$:}
In this case , $u\bk{r_1,r_2,r_3,r_4,r_5}\in U_k\bk{g}$ if and only if
\begin{align*}
&\FNorm{r_1}, \FNorm{r_1r_3-r_2^2}, \FNorm{r_1r_4-r_2r_3} \leq q^{k+n-m}\\
&\FNorm{r_2}, \FNorm{r_3}, \FNorm{r_4}, \FNorm{r_2r_4-r_3^2}, \FNorm{z} \leq q^{k-n} .
\end{align*}
We make a change of variables
\[
x=r_4+D_{\bk{a,b,c}}r_2
\]
So that $u\bk{r_1,r_2,r_3,r_4,r_5}\in \widehat{U_k\bk{g}}$ if and only if
\begin{align*}
&\FNorm{x}\leq q \\
&\FNorm{r_1}, \FNorm{r_1r_3-r_2^2}, \FNorm{r_1\bk{x-D_{\bk{a,b,c}}r_2}-r_2r_3} \leq q^{k+n-m}\\
&\FNorm{r_2}, \FNorm{r_3}, \FNorm{r_2\bk{x-D_{\bk{a,b,c}}r_2}-r_3^2}, \FNorm{z} \leq q^{k-n} .
\end{align*}

\begin{itemize}
\item
\underline{$\FNorm{x},\FNorm{r_2}\leq 1$:}
In this case, $u\bk{r_1,r_2,r_3,r_4,r_5}\in \widehat{U_k\bk{g}}$ if and only if
\begin{align*}
&\FNorm{r_1}, \FNorm{r_1r_3} \leq q^{k+n-m}\\
&\FNorm{r_3}\leq q^{\frac{k-n}{2}}\\
&\FNorm{z} \leq q^{k-n} .
\end{align*}

\item
\underline{$\FNorm{x}\leq 1,1<\FNorm{r_2}\leq q^{\frac{k-n}{2}}$:}
In this case, $u\bk{r_1,r_2,r_3,r_4,r_5}\in \widehat{U_k\bk{g}}$ if and only if
\begin{align*}
&\FNorm{r_1r_2}, \FNorm{r_1r_3} \leq q^{k+n-m} \\
&\FNorm{r_3}\leq q^{\frac{k-n}{2}} \\
&\FNorm{z} \leq q^{k-n} .
\end{align*}

\item
\underline{$\FNorm{x}\leq 1,q^{\frac{k-n}{2}}<\FNorm{r_2}\leq q^{k-n}$:}
In this case, $u\bk{r_1,r_2,r_3,r_4,r_5}\in \widehat{U_k\bk{g}}$ if and only if
\begin{align*}
&\FNorm{r_1}, \FNorm{r_1r_3-r_2^2}, \FNorm{D_{\bk{a,b,c}}r_1r_2+r_2r_3} \leq q^{k+n-m}\\
&\FNorm{r_3}, \FNorm{D_{\bk{a,b,c}}r_2^2+r_3^2}, \FNorm{z} \leq q^{k-n} .
\end{align*}
We make a change of variables
\[
y=D_{\bk{a,b,c}}r_1r_2+r_2r_3
\]
and then $u\bk{r_1,r_2,r_3,r_4,r_5}\in U_k\bk{g}$ if and only if
\begin{align*}
&\FNorm{y} \leq q^{k+n-m}\\
&\FNorm{r_3}=\FNorm{r_2} \\
&\FNorm{D_{\bk{a,b,c}}r_2^2+r_3^2}, \FNorm{z} \leq q^{k-n} .
\end{align*}

\item
\underline{$\FNorm{x}=q,\FNorm{x-D_{\bk{a,b,c}}r_2}\leq 1$:}
In this case, $u\bk{r_1,r_2,r_3,r_4,r_5}\in \widehat{U_k\bk{g}}$ if and only if
\begin{align*}
&\FNorm{r_1}, \FNorm{r_1r_3} \leq q^{k+n-m}\\
&\FNorm{r_3} \leq q^{\frac{k-n}{2}} \\
&\FNorm{z} \leq q^{k-n} .
\end{align*}

\item
\underline{$\FNorm{x}=q,\FNorm{x-D_{\bk{a,b,c}}r_2}=q$:}
In this case, $u\bk{r_1,r_2,r_3,r_4,r_5}\in \widehat{U_k\bk{g}}$ if and only if
\begin{align*}
&\FNorm{r_1} \leq q^{k+n-m-1}\\
&\FNorm{r_1r_3} \leq q^{k+n-m}\\
&\FNorm{r_3} \leq q^{\frac{k-n}{2}} \\
&\FNorm{z} \leq q^{k-n} .
\end{align*}

\item
\underline{$\FNorm{x}=q,q<\FNorm{r_2}\leq q^{\frac{k-n}{2}}$:}
In this case, $u\bk{r_1,r_2,r_3,r_4,r_5}\in \widehat{U_k\bk{g}}$ if and only if
\begin{align*}
&\FNorm{r_1 r_2}, \FNorm{r_1 r_3} \leq q^{k+n-m}\\
&\FNorm{r_3} \leq q^{\frac{k-n}{2}} \\
&\FNorm{z} \leq q^{k-n} .
\end{align*}

\item
\underline{$\FNorm{x}=q,q^{\frac{k-n}{2}}<\FNorm{r_2}\leq q^{k-n-1}$:}
In this case, $u\bk{r_1,r_2,r_3,r_4,r_5}\in \widehat{U_k\bk{g}}$ if and only if
\begin{align*}
&\FNorm{r_1}, \FNorm{r_1r_3-r_2^2}, \FNorm{r_1\bk{x-D_{\bk{a,b,c}}r_2}-r_2r_3} \leq q^{k+n-m}\\
&\FNorm{r_3}=\FNorm{r_2} \\
&\FNorm{D_{\bk{a,b,c}}r_2^2+r_3^2}, \FNorm{z} \leq q^{k-n} .
\end{align*}
We make a change of variables
\[
y=r_1\bk{x-D_{\bk{a,b,c}}r_2}-r_2r_3
\]
and then
$u\bk{r_1,r_2,r_3,r_4,r_5}\in U_k\bk{g}$ if and only if
\begin{align*}
&\FNorm{y} \leq q^{k+n-m}\\
&\FNorm{r_3}=\FNorm{r_2} \\
&\FNorm{D_{\bk{a,b,c}}r_2^2+r_3^2}, \FNorm{z} \leq q^{k-n} .
\end{align*}

\item
\underline{$\FNorm{x}=q,\FNorm{r_2}=q^{k-n}$:}
In this case, $u\bk{r_1,r_2,r_3,r_4,r_5}\in \widehat{U_k\bk{g}}$ if and only if
\begin{align*}
&\FNorm{r_1}, \FNorm{r_1r_3-r_2^2}, \FNorm{r_1\bk{x-D_{\bk{a,b,c}}r_2}-r_2r_3} \leq q^{k+n-m}\\
&\FNorm{r_2}, \FNorm{r_3}, \FNorm{r_2\bk{x-D_{\bk{a,b,c}}r_2}-r_3^2}, \FNorm{z} \leq q^{k-n} .
\end{align*}
We make a change of variables
\begin{align*}
&y=r_1\bk{x-D_{\bk{a,b,c}} r_2}-r_2r_3 \\
&r_2'=r_2-\frac{x}{2 D_{\bk{a,b,c}}}.
\end{align*}
Note that $\FNorm{r_3}=\FNorm{r_2'}$.
It holds that $u\bk{r_1,r_2,r_3,r_4,r_5}\in \widehat{U_k\bk{g}}$ if and only if
\begin{align*}
&\FNorm{y}\leq q^{k+n-m}\\
&\FNorm{r_3}=\FNorm{r_2'}\\
&\FNorm{D_{\bk{a,b,c}}\bk{r_2'}^2+r_3^2}, \FNorm{z} \leq q^{k-n} .
\end{align*}
\end{itemize}

Denote $l=\lfloor \frac{k-n}{2} \rfloor$, it holds that
\begin{align*}
E_{k}^{\Psi_E}\bk{g} &= q^{k-n} \left( \kappa\bk{k+n-m,l,k+n-m} + \suml_{j=1}^l q^j\bk{1-q^{-1}} \kappa\bk{k+n-m-j,l,k+n-m} \right. \\
&\left.+ q^{k+n-m} \suml_{j=l+1}^{k-n} \kappa_{\bk{a,b,c}}^{\bk{k-n}}\bk{q^j}\right) - \left( \kappa\bk{k+n-m,l,k+n-m} + \right. \\
&\left. \suml_{j=1}^l q^j\bk{1-q^{-1}} \kappa\bk{k+n-m-j,l,k+n-m} + q^{k+n-m} \suml_{j=l+1}^{k-n} \kappa_{\bk{a,b,c}}^{\bk{k+n-m}}\bk{q^j}  \right) = 0 .
\end{align*}

\end{itemize}

\paragraph{\underline{Case of $m=2n$}:}
We write
\[
z:=r_2r_3+r_5
\]
In this case, $u\bk{r_1,r_2,r_3,r_4,r_5}\in U_k\bk{g}$ if and only if
\begin{align*}
&\FNorm{r_1}, \FNorm{r_2}, \FNorm{r_3}, \FNorm{r_4}, \FNorm{z}, \FNorm{r_1r_3-r_2^2}, \FNorm{r_2r_4-r_3^2}, \FNorm{r_1r_4-r_2r_3} \leq q^{k-n} .
\end{align*}
We compute $E_k^{\Psi_E}\bk{g}$ for different values of $k$.

\begin{itemize}
\item
\underline{$k=n$:}
In this case , $u\bk{r_1,r_2,r_3,r_4,r_5}\in U_k\bk{g}$ if and only if
\begin{align*}
&\FNorm{r_1}, \FNorm{r_2}, \FNorm{r_4}, \FNorm{x}, \FNorm{z} \leq 1 .
\end{align*}
and hence
\[
E_n^{\Psi_E}\bk{g} = \intl_{\mO^5} \psi\bk{r_4+D_{\bk{a,b,c}r_2}} dr_1\ dr_2\ dr_3\ dr_4\ dz = 1 .
\]

\item
\underline{$k=n+1$:}
In this case, $u\bk{r_1,r_2,r_3,r_4,r_5}\in U_k\bk{g}$ if and only if
\begin{align*}
&\FNorm{r_1}, \FNorm{r_2}, \FNorm{r_3}, \FNorm{r_4}, \FNorm{z}, \FNorm{r_1r_3-r_2^2}, \FNorm{r_2r_4-r_3^2}, \FNorm{r_1r_4-r_2r_3} \leq q .
\end{align*}

Note that from $\FNorm{r_1r_3-r_2^2}, \FNorm{r_2r_4-r_3^2}\leq q$ it follows that $\FNorm{r_2}=q$ if and only if $\FNorm{r_3}=q$, in which case $\FNorm{r_1}=\FNorm{r_4}=q$ also.

\begin{itemize}
\item
\underline{$\FNorm{r_2},\FNorm{r_3}\leq 1$:}
In this case, $u\bk{r_1,r_2,r_3,r_4,r_5}\in U_k\bk{g}$ if and only if
\begin{align*}
&\FNorm{r_1}, \FNorm{r_4}, \FNorm{z}, \FNorm{r_1r_4} \leq q .
\end{align*}

\underline{$\FNorm{r_2},\FNorm{r_3}=q$:}
In this case, $u\bk{r_1,r_2,r_3,r_4,r_5}\in U_k\bk{g}$ if and only if
\begin{align*}
&\FNorm{r_1}, \FNorm{r_4}, \FNorm{z}, \FNorm{r_1r_3-r_2^2}, \FNorm{r_2r_4-r_3^2}, \FNorm{r_1r_4-r_2r_3} \leq q .
\end{align*}
We make change of variables
\begin{align*}
x&=r_1r_3-r_2^2 \\
y&=r_4r_2-r_3^2 .
\end{align*}
and hence $u\bk{r_1,r_2,r_3,r_4,r_5}\in U_k\bk{g}$ if and only if
\begin{align*}
&\FNorm{x}, \FNorm{y}, \FNorm{z}\leq q
\end{align*}
The other inequalities are satisfied immediately.
\end{itemize}

In conclusion, we have
\begin{align*}
E_{n+1}^{\Psi_E}\bk{g} &= \intl_{\unif^{-1}\mO} dz \intl_{\mO^2} dr_2\ dr_3 \intl_{\FNorm{r_1}, \FNorm{r_4}, \FNorm{r_1r_4} \leq q} \psi\bk{r_4} dr_1\ dr_4 + \\
&+ \intl_{\unif^{-1}\mO} dz \intl_{\unif^{-1}\mO^\times} dr_3 \intl_{\unif^{-1}\mO^\times} \psi\bk{D_{\bk{a,b,c}}r_2} dr_2 \intl_{\unif^{-1}\mO} \frac{dx}{\FNorm{r_3}} \intl_{\unif^{-1}\mO} \psi\bk{\frac{y+r_3^2}{r_2}} \frac{dy}{\FNorm{r_2}} = \\
&=q\bk{q-1} + \intl_{\bk{\unif^{-1}\mO^\times}^2} \psi\bk{D_{\bk{a,b,c}}r_2+\frac{r_3^2}{r_2}} dr_2\ dr_3 \intl_{\unif^{-1}\mO} \psi\bk{\frac{y}{r_2}} dy = 0 .
\end{align*}

\item
\underline{$k>n+1$:}
In this case, $u\bk{r_1,r_2,r_3,r_4,r_5}\in U_k\bk{g}$ if and only if
\begin{align*}
&\FNorm{r_1}, \FNorm{r_2}, \FNorm{r_3}, \FNorm{r_4}, \FNorm{z}, \FNorm{r_1r_3-r_2^2}, \FNorm{r_2r_4-r_3^2}, \FNorm{r_1r_4-r_2r_3} \leq q^{k-n} .
\end{align*}
We make a change of variables
\[
x=r_4+D_{\bk{a,b,c}} r_2 .
\]
We then have $u\bk{r_1,r_2,r_3,r_4,r_5}\in \widehat{U_k\bk{g}}$ if and only if
\begin{align*}
&\FNorm{x} \leq q \\
&\FNorm{r_1}, \FNorm{r_2}, \FNorm{r_3}, \FNorm{z}, \FNorm{r_1r_3-r_2^2}, \FNorm{r_2\bk{x-D_{\bk{a,b,c}} r_2}-r_3^2}, \FNorm{r_1\bk{x-D_{\bk{a,b,c}} r_2}-r_2r_3} \leq q^{k-n} .
\end{align*}

\begin{itemize}
\item
\underline{$\FNorm{x},\FNorm{r_2}\leq 1$:}
In this case, $u\bk{r_1,r_2,r_3,r_4,r_5}\in \widehat{U_k\bk{g}}$ if and only if
\begin{align*}
&\FNorm{r_3}\leq q^{\frac{k-n}{2}} \\
&\FNorm{r_1}, \FNorm{r_1r_3}, \FNorm{z} \leq q^{k-n} .
\end{align*}

\item
\underline{$\FNorm{x}\leq 1,1<\FNorm{r_2}\leq q^{\frac{k-n}{2}}$:}
In this case, $u\bk{r_1,r_2,r_3,r_4,r_5}\in \widehat{U_k\bk{g}}$ if and only if
\begin{align*}
&\FNorm{r_3}\leq q^{\frac{k-n}{2}} \\
&\FNorm{z}, \FNorm{r_1r_2}, \FNorm{r_1r_3} \leq q^{k-n} .
\end{align*}

\item
\underline{$\FNorm{x}\leq 1,q^{\frac{k-n}{2}}<\FNorm{r_2}\leq q^{k-n}$:}
In this case, $u\bk{r_1,r_2,r_3,r_4,r_5}\in \widehat{U_k\bk{g}}$ if and only if
\begin{align*}
&\FNorm{r_1}, \FNorm{r_3}, \FNorm{z}, \FNorm{r_1r_3-r_2^2}, \FNorm{D_{\bk{a,b,c}} r_2^2+r_3^2}, \FNorm{\bk{D_{\bk{a,b,c}} r_1 + r_3}r_2} \leq q^{k-n} .
\end{align*}
Note that from $\FNorm{D_{\bk{a,b,c}} r_2^2+r_3^2}\leq q^{k-n}$ it follows that $\FNorm{r_2}=\FNorm{r_3}$.
We make a change of variables
\[
y=\bk{D_{\bk{a,b,c}} r_1 + r_3}r_2
\]
to have $u\bk{r_1,r_2,r_3,r_4,r_5}\in U_k\bk{g}$ if and only if
\begin{align*}
&\FNorm{r_2} = \FNorm{r_3} \\
&\FNorm{y}, \FNorm{z}, \FNorm{D_{\bk{a,b,c}} r_2^2+r_3^2} \leq q^{k-n} .
\end{align*}

\item
\underline{$\FNorm{x}=q,\FNorm{x-D_{\bk{a,b,c}}r_2}\leq 1$:}
In this case $\FNorm{r_2}\leq q\leq q^{k-n}$.
We have $u\bk{r_1,r_2,r_3,r_4,r_5}\in \widehat{U_k\bk{g}}$ if and only if
\begin{align*}
&\FNorm{r_3} \leq q^{\frac{k-n}{2}} \\
&\FNorm{r_1}, \FNorm{z}, \FNorm{r_1r_3}  \leq q^{k-n} .
\end{align*}

\item
\underline{$\FNorm{x}=q,\FNorm{x-D_{\bk{a,b,c}}r_2}=q$:}
In this case $\FNorm{r_2}\leq q\leq q^{k-n}$.
We have $u\bk{r_1,r_2,r_3,r_4,r_5}\in \widehat{U_k\bk{g}}$ if and only if
\begin{align*}
&\FNorm{r_1} \leq q^{k-n-1} \\
&\FNorm{r_3} \leq q^{\frac{k-n}{2}} \\
&\FNorm{z}, \FNorm{r_1r_3} \leq q^{k-n} .
\end{align*}

\item
\underline{$\FNorm{x}=q,q<\FNorm{r_2}\leq q^{\frac{k-n}{2}}\leq q^{k-n-1}$:}
Note that in this case (if it happens) $\FNorm{x-D_{\bk{a,b,c}} r_2}=\FNorm{r_2}$
and hence $u\bk{r_1,r_2,r_3,r_4,r_5}\in \widehat{U_k\bk{g}}$ if and only if
\begin{align*}
&\FNorm{r_3} \leq q^{\frac{k-n}{2}} \\
&\FNorm{z}, \FNorm{r_1r_3}, \FNorm{r_1r_2} \leq q^{k-n} .
\end{align*}

\item
\underline{$\FNorm{x}=q, q^{\frac{k-n}{2}}<\FNorm{r_2}\leq q^{k-n-1}$:}
In this case, $u\bk{r_1,r_2,r_3,r_4,r_5}\in \widehat{U_k\bk{g}}$ if and only if
\begin{align*}
&\FNorm{r_1}, \FNorm{r_3}, \FNorm{z}, \FNorm{r_1r_3-r_2^2}, \FNorm{D_{\bk{a,b,c}} r_2^2+r_3^2}, \FNorm{r_1\bk{x-D_{\bk{a,b,c}} r_2}-r_2r_3} \leq q^{k-n} .
\end{align*}
Note that here $\FNorm{r_1}=\FNorm{r_2}=\FNorm{r_3}$ since $\FNorm{r_1r_3-r_2^2}, \FNorm{D_{\bk{a,b,c}} r_2^2+r_3^2}\leq q^{k-n}$ and $q^{\frac{k-n}{2}}<\FNorm{r_2}$.
We make a change of variables
\[
y=r_1\bk{x-D_{\bk{a,b,c}} r_2}-r_2r_3
\]
and then$u\bk{r_1,r_2,r_3,r_4,r_5}\in U_k\bk{g}$ if and only if
\begin{align*}
&\FNorm{r_2}=\FNorm{r_3} \\
&\FNorm{y}, \FNorm{z}, \FNorm{D_{\bk{a,b,c}} r_2^2+r_3^2} \leq q^{k-n} .
\end{align*}

\item
\underline{$\FNorm{x}=q ,\FNorm{r_2}=q^{k-n}$:}
In this case, $u\bk{r_1,r_2,r_3,r_4,r_5}\in \widehat{U_k\bk{g}}$ if and only if
\begin{align*}
&\FNorm{r_1}, \FNorm{r_3}, \FNorm{z}, \FNorm{r_1r_3-r_2^2}, \FNorm{r_2\bk{x-D_{\bk{a,b,c}} r_2}-r_3^2}, \FNorm{r_1\bk{x-D_{\bk{a,b,c}} r_2}-r_2r_3} \leq q^{k-n} .
\end{align*}
Note that here $\FNorm{r_1}=\FNorm{r_2}=\FNorm{r_3}=q^{k-n}$ since $\FNorm{r_1r_3-r_2^2}, \FNorm{D_{\bk{a,b,c}} r_2^2+r_3^2}\leq q^{k-n}$ and $q^{\frac{k-n}{2}}<\FNorm{r_2}$.
We make a change of variables
\begin{align*}
&y=r_1\bk{x-D_{\bk{a,b,c}} r_2}-r_2r_3 \\
&r_2'=r_2-\frac{x}{2 D_{\bk{a,b,c}}}.
\end{align*}
Note that $\FNorm{r_2} = \FNorm{r_2'}$.
We then have $u\bk{r_1,r_2,r_3,r_4,r_5}\in \widehat{U_k\bk{g}}$ if and only if
\begin{align*}
&\FNorm{y}, \FNorm{r_3}, \FNorm{z}, \FNorm{D_{\bk{a,b,c}}\bk{r_2'}^2+r_3^2} \leq q^{k-n} .
\end{align*}
\end{itemize}

Denote $l=\lfloor \frac{k-n}{2} \rfloor$, it holds that
\begin{align*}
E_{k}^{\Psi_E}\bk{g} &= q^{k-n} \left( \kappa\bk{k-n,l,k-n} + \suml_{j=1}^l q^j\bk{1-q^{-1}} \kappa\bk{k-n-j,l,k-n} \right. \\
&\left.+ q^{k-n} \suml_{j=l+1}^{k-n} \kappa_{\bk{a,b,c}}^{\bk{k-n}}\bk{q^j}\right) - \left( \kappa\bk{k-n,l,k-n} + q\bk{1-q^{-1}} \kappa\bk{k-n-1,l,k-n} + \right. \\
&\left. \suml_{j=2}^l q^j\bk{1-q^{-1}} \kappa\bk{k-n-j,l,k-n} + q^{k-n} \suml_{j=l+1}^{k-n} \kappa_{\bk{a,b,c}}^{\bk{k-n}}\bk{q^j}  \right) = 0 .
\end{align*}
\end{itemize}

\paragraph{\underline{Case of $m>2n$}:}
Since the proof of this case is identical to the case of $N_{\bk{a,b,c}}\in\mO^\times$, we write here
\[
\Psi_E\bk{u\bk{r_1,r_2,r_3,r_4,r_5}} = \psi\bk{r_4+D_{\bk{a,b,c}}r_2-N_{\bk{a,b,c}}r_1}
\]
so that the proof will fit both cases.
We write
\[
z:=r_2r_3+r_5
\]
In this case, $u\bk{r_1,r_2,r_3,r_4,r_5}\in U_k\bk{g}$ if and only if
\begin{align*}
&\FNorm{r_1}, \FNorm{r_2}, \FNorm{r_3}, \FNorm{r_1r_3-r_2^2}, \FNorm{z} \leq q^{k+n-m}\\
&\FNorm{r_2}, \FNorm{r_3}, \FNorm{r_4}, \FNorm{r_2r_4-r_3^2}, \FNorm{r_1r_4-r_2r_3} \leq q^{k-n} .
\end{align*}

\begin{itemize}
\item
\underline{$k=m-n$:}
In this case , $u\bk{r_1,r_2,r_3,r_4,r_5}\in U_k\bk{g}$ if and only if
\begin{align*}
&\FNorm{r_1}, \FNorm{r_2}, \FNorm{r_3}, \FNorm{z} \leq 1\\
&\FNorm{r_4} \leq q^{m-2n} .
\end{align*}
and since $m-2n\geq 1$ it holds that
\[
E_n^{\Psi_E}\bk{g} = \intl_{\FNorm{r_4} \leq q^{m-2n}} \psi\bk{r_4} dr_4 = 0 .
\]

\item
\underline{$k=m-n+1$:}
In this case , $u\bk{r_1,r_2,r_3,r_4,r_5}\in U_k\bk{g}$ if and only if
\begin{align*}
&\FNorm{r_1}, \FNorm{r_2}, \FNorm{r_3}, \FNorm{r_1r_3-r_2^2}, \FNorm{z} \leq q\\
&\FNorm{r_4}, \FNorm{r_2r_4}, \FNorm{r_1r_4} \leq q^{m-2n+1} .
\end{align*}

\begin{itemize}
\item
\underline{$\FNorm{r_1},\FNorm{r_2}\leq 1$:}
In this case, $u\bk{r_1,r_2,r_3,r_4,r_5}\in U_k\bk{g}$ if and only if
\begin{align*}
&\FNorm{r_3}, \FNorm{z} \leq q\\
&\FNorm{r_4} \leq q^{m-2n+1} .
\end{align*}

\item
\underline{$\FNorm{r_2}\leq 1,\FNorm{r_1}=q$:}
In this case, $u\bk{r_1,r_2,r_3,r_4,r_5}\in U_k\bk{g}$ if and only if
\begin{align*}
&\FNorm{r_3} \leq 1 \\
&\FNorm{z} \leq q\\
&\FNorm{r_4} \leq q^{m-2n} .
\end{align*}

\item
\underline{$\FNorm{r_2}=q$:}
In this case, $u\bk{r_1,r_2,r_3,r_4,r_5}\in U_k\bk{g}$ if and only if
\begin{align*}
&\FNorm{r_1}=\FNorm{r_3}=q \\
&\FNorm{z} \leq q\\
&\FNorm{r_4} \leq q^{m-2n} .
\end{align*}
\end{itemize}

In conclusion, we have
\begin{align*}
E_{n+1}^{\Psi_E}\bk{g} &= q^2\intl_{\FNorm{r_4} \leq q^{m-2n+1}} \psi\bk{r_4} dr_4 + q\bk{q-1} \intl_{\FNorm{r_4} \leq q^{m-2n}} \psi\bk{r_4} dr_4 + \\
&q \bk{q-1} \intl_{\FNorm{r_1}=\FNorm{r_2}=q} \psi\bk{D_{\bk{a,b,c}}r_2-N_{\bk{a,b,c}}r_1} dr_1\ dr_2\ dr_4 \intl_{\FNorm{r_4} \leq q^{m-2n}} \psi\bk{r_4} dr_4 = 0 .
\end{align*}

\item
\underline{$k>m-n+1$:}
In this case , $u\bk{r_1,r_2,r_3,r_4,r_5}\in U_k\bk{g}$ if and only if
\begin{align*}
&\FNorm{r_1}, \FNorm{r_2}, \FNorm{r_3}, \FNorm{r_1r_3-r_2^2}, \FNorm{z} \leq q^{k+n-m}\\
&\FNorm{r_2}, \FNorm{r_3}, \FNorm{r_4}, \FNorm{r_2r_4-r_3^2}, \FNorm{r_1r_4-r_2r_3} \leq q^{k-n} .
\end{align*}
We make a change of variables
\[
x=r_4+D_{\bk{a,b,c}}r_2-N_{\bk{a,b,c}}r_1
\]
and then
$u\bk{r_1,r_2,r_3,r_4,r_5}\in \widehat{U_k\bk{g}}$ if and only if
\begin{align*}
&\FNorm{x}\leq q \\
&\FNorm{r_1}, \FNorm{r_2}, \FNorm{r_3}, \FNorm{r_1r_3-r_2^2}, \FNorm{z} \leq q^{k+n-m}\\
&\FNorm{r_2\bk{N_{\bk{a,b,c}}r_1-D_{\bk{a,b,c}}r_2}-r_3^2}, \FNorm{r_1\bk{N_{\bk{a,b,c}}r_1-D_{\bk{a,b,c}}r_2}-r_2r_3} \leq q^{k-n} .
\end{align*}
Denote by $C\subset F^3$ the set of elements $\bk{r_1,r_2,r_3}\in C$ such that
\begin{align*}
&\FNorm{r_1}, \FNorm{r_2}, \FNorm{r_3}, \FNorm{r_1r_3-r_2^2} \leq q^{k+n-m}\\
&\FNorm{r_2\bk{N_{\bk{a,b,c}}r_1-D_{\bk{a,b,c}}r_2}-r_3^2}, \FNorm{r_1\bk{N_{\bk{a,b,c}}r_1-D_{\bk{a,b,c}}r_2}-r_2r_3} \leq q^{k-n} .
\end{align*}
It then holds that
\begin{align*}
E_k^{\Psi_E}\bk{g} &= q^{k+n-m} \meas\bk{C} \intl_{\FNorm{r_4} \leq q} \psi\bk{r_4} dr_4 = 0 .
\end{align*}
\end{itemize}

\subsubsection{$N_{\bk{a,b,c}}\in\mO^\times$}


\paragraph{\underline{Case of $m<2n$}:}
We write
\[
z:=r_1r_4-2r_2r_3-r_5 .
\]
In this case, $u\bk{r_1,r_2,r_3,r_4,r_5}\in U_k\bk{g}$ if and only if
\begin{align*}
&\FNorm{r_1}, \FNorm{r_1r_3-r_2^2}, \FNorm{r_1r_4-r_2r_3} \leq q^{k+n-m}\\
&\FNorm{r_2}, \FNorm{r_3}, \FNorm{r_4}, \FNorm{r_2r_4-r_3^2}, \FNorm{z} \leq q^{k-n} .
\end{align*}

\begin{itemize}
\item
\underline{$k=n$:}
In this case , $u\bk{r_1,r_2,r_3,r_4,r_5}\in U_k\bk{g}$ if and only if
\begin{align*}
&\FNorm{r_1} \leq q^{2n-m}\\
&\FNorm{r_2}, \FNorm{r_3}, \FNorm{r_4}, \FNorm{z} \leq 1
\end{align*}
and hence
\begin{align*}
E_n^{\Psi_E}\bk{g} &= \intl_{\FNorm{r_1} \leq q^{2n-m}} \psi\bk{N_{\bk{a,b,c}}r_1} dr_1 = 0 ,
\end{align*}
since $2n-m\geq 1$.

\item
\underline{$k=n+1$:}
In this case , $u\bk{r_1,r_2,r_3,r_4,r_5}\in U_k\bk{g}$ if and only if
\begin{align*}
&\FNorm{r_1}, \FNorm{r_1r_3}, \FNorm{r_1r_4} \leq q^{2n-m+1}\\
&\FNorm{r_2}, \FNorm{r_3}, \FNorm{r_4}, \FNorm{r_2r_4-r_3^2}, \FNorm{z} \leq q .
\end{align*}

\begin{itemize}
\item
\underline{$\FNorm{r_4} \leq 1$:}
In this case, $u\bk{r_1,r_2,r_3,r_4,r_5}\in U_k\bk{g}$ if and only if
\begin{align*}
&\FNorm{r_1} \leq q^{2n-m+1}\\
&\FNorm{r_2}, \FNorm{z} \leq q \\
&\FNorm{r_3} \leq 1 .
\end{align*}

\item
\underline{$\FNorm{r_4}=q$:}
In this case, $u\bk{r_1,r_2,r_3,r_4,r_5}\in U_k\bk{g}$ if and only if
\begin{align*}
&\FNorm{r_1} \leq q^{2n-m}\\
&\FNorm{r_2}, \FNorm{r_3}, \FNorm{r_2r_4-r_3^2}, \FNorm{z} \leq q .
\end{align*}
\end{itemize}
In conclusion
\begin{align*}
E_{n+1}^{\Psi_E}\bk{g} &= q^2 \intl_{\FNorm{r_1} \leq q^{2n-m+1}} \psi\bk{N_{\bk{a,b,c}}r_1} dr_1 + \\
& q \intl_{\FNorm{r_1} \leq q^{2n-m}} \psi\bk{N_{\bk{a,b,c}}r_1} dr_1 \intl_{\stackrel{\FNorm{r_4}=q}{\FNorm{r_2}, \FNorm{r_3}, \FNorm{r_2r_4-r_3^2}, \FNorm{z} \leq q}} \psi\bk{r_4+D_{\bk{a,b,c}}r_2} dr_2\ dr_3\ dr_4 = 0 .
\end{align*}

\item
\underline{$k>n+1$:}
In this case , $u\bk{r_1,r_2,r_3,r_4,r_5}\in U_k\bk{g}$ if and only if
\begin{align*}
&\FNorm{r_1}, \FNorm{r_1r_3-r_2^2}, \FNorm{r_1r_4-r_2r_3} \leq q^{k+n-m}\\
&\FNorm{r_2}, \FNorm{r_3}, \FNorm{r_4}, \FNorm{r_2r_4-r_3^2}, \FNorm{z} \leq q^{k-n} .
\end{align*}
We make a change of variables
\[
x=r_4+D_{\bk{a,b,c}}r_2-N_{\bk{a,b,c}}r_1
\]
and then $u\bk{r_1,r_2,r_3,r_4,r_5}\in \widehat{U_k\bk{g}}$ if and only if
\begin{align*}
&\FNorm{x} \leq q \\
&\FNorm{\bk{r_4+D_{\bk{a,b,c}}r_2}r_3-N_{\bk{a,b,c}}r_2^2}, \FNorm{\bk{r_4+D_{\bk{a,b,c}}r_2}r_4-N_{\bk{a,b,c}r_2r_3}}l \leq q^{k+n-m}\\
&\FNorm{r_2}, \FNorm{r_3}, \FNorm{r_4}, \FNorm{r_2r_4-r_3^2}, \FNorm{z} \leq q^{k-n} .
\end{align*}
We denote by $C\subset F^3$ the set of all elements $\bk{r_2,r_3,r_4}$ such that
\begin{align*}
&\FNorm{\bk{r_4+D_{\bk{a,b,c}}r_2}r_3-N_{\bk{a,b,c}}r_2^2}, \FNorm{\bk{r_4+D_{\bk{a,b,c}}r_2}r_4-N_{\bk{a,b,c}r_2r_3}} \leq q^{k+n-m}\\
&\FNorm{r_2}, \FNorm{r_3}, \FNorm{r_4}, \FNorm{r_2r_4-r_3^2} \leq q^{k-n} .
\end{align*}
It holds that
\begin{align*}
E_k^{\Psi_E}\bk{g} &= q^{k-n} \meas\bk{C} \intl_{\FNorm{x} \leq q} \psi\bk{x} dx = 0 .
\end{align*}
\end{itemize}

\paragraph{\underline{Case of $m=2n$}:}
We write
\[
z:=r_2r_3+r_5 .
\]
In this case, $u\bk{r_1,r_2,r_3,r_4,r_5}\in U_k\bk{g}$ if and only if
\begin{align*}
&\FNorm{r_1}, \FNorm{r_2}, \FNorm{r_3}, \FNorm{r_4}, \FNorm{z}, \FNorm{r_1r_3-r_2^2}, \FNorm{r_2r_4-r_3^2}, \FNorm{r_1r_4-r_2r_3} \leq q^{k-n} .
\end{align*}

\begin{itemize}
\item
\underline{$k=n$:}
In this case , $u\bk{r_1,r_2,r_3,r_4,r_5}\in U_k\bk{g}$ if and only if
\begin{align*}
&\FNorm{r_1}, \FNorm{r_2}, \FNorm{r_3}, \FNorm{r_4}, \FNorm{z} \leq 1
\end{align*}
and hence
\begin{align*}
E_n^{\Psi_E}\bk{g} &= 1.
\end{align*}

\item
\underline{$k=n+1$:}
In this case , $u\bk{r_1,r_2,r_3,r_4,r_5}\in U_k\bk{g}$ if and only if
\begin{align*}
&\FNorm{r_1}, \FNorm{r_2}, \FNorm{r_3}, \FNorm{r_4}, \FNorm{z}, \FNorm{r_1r_3-r_2^2}, \FNorm{r_2r_4-r_3^2}, \FNorm{r_1r_4-r_2r_3} \leq q .
\end{align*}
Note that $\FNorm{r_2}=q$ if and only if $\FNorm{r_3}=q$, in which case $\FNorm{r_1}=\FNorm{r_4}=q$.

\begin{itemize}
\item
\underline{$\FNorm{r_2}, \FNorm{r_3}\leq 1$:}
In this case , $u\bk{r_1,r_2,r_3,r_4,r_5}\in U_k\bk{g}$ if and only if
\begin{align*}
&\FNorm{r_1}, \FNorm{r_4}, \FNorm{z}, \FNorm{r_1r_4} \leq q .
\end{align*}

\underline{$\FNorm{r_2}=\FNorm{r_3}=q$:}
In this case , $u\bk{r_1,r_2,r_3,r_4,r_5}\in U_k\bk{g}$ if and only if
\begin{align*}
&\FNorm{r_1}, \FNorm{r_4}, \FNorm{z}, \FNorm{r_1r_3-r_2^2}, \FNorm{r_2r_4-r_3^2}, \FNorm{r_1r_4-r_2r_3} \leq q .
\end{align*}
We make a change of variables
\begin{align*}
x&=r_1r_3-r_2^2 \\
y&=r_2r_4-r_3^2 .
\end{align*}
It then holds that
$u\bk{r_1,r_2,r_3,r_4,r_5}\in U_k\bk{g}$ if and only if
\begin{align*}
&\FNorm{x}, \FNorm{y}, \FNorm{z} \leq q .
\end{align*}
\end{itemize}
In conclusion
\begin{align*}
E_{n+1}^{\Psi_E}\bk{g} &= q \intl_{\FNorm{r_1},\FNorm{r_4},\FNorm{r_1r_4}\leq q} \psi\bk{r_4+D_{\bk{a,b,c}}r_2} dr_2\ dr_4 + \\
& q\intl_{\stackrel{\FNorm{r_2}=\FNorm{r_3}=q}{\FNorm{x},\FNorm{y}\leq q}} \psi\bk{\frac{y+r_3^2}{r_2} + D_{\bk{a,b,c}}r_2 + N_{\bk{a,b,c}\frac{x+r_2^2}{r_3} }} dr_2\ dr_3\ \frac{dx}{q}\ \frac{dy}{q}=-q+q\bk{1-q} = -q^2 .
\end{align*}

\item
\underline{$k>n+1$:}
We make a change of variables
\[
x=r_4+D_{\bk{a,b,c}}r_2-N_{\bk{a,b,c}}r_1 .
\]
In this case , $u\bk{r_1,r_2,r_3,r_4,r_5}\in \widehat{U_k\bk{g}}$ if and only if
\begin{align*}
&\FNorm{x}\leq q \\
&\FNorm{r_1}, \FNorm{r_2}, \FNorm{r_3}, \FNorm{z}, \FNorm{r_1\bk{x-D_{\bk{a,b,c}}r_2+N_{\bk{a,b,c}}r_1}-r_2r_3} \leq q^{k-n} \\
&\FNorm{r_1r_3-r_2^2}, \FNorm{r_2\bk{x-D_{\bk{a,b,c}}r_2+N_{\bk{a,b,c}}r_1}-r_3^2} \leq q^{k-n} .
\end{align*}

\begin{itemize}




\item
\underline{$\FNorm{r_2}\leq q^{\frac{k-n}{2}}$:}
In this case , $u\bk{r_1,r_2,r_3,r_4,r_5}\in \widehat{U_k\bk{g}}$ if and only if
\begin{align*}
&\FNorm{x}\leq q \\
&\FNorm{r_1}, \FNorm{r_3} \leq q^{\frac{k-n}{2}} \\
&\FNorm{z} \leq q^{k-n} .
\end{align*}

\item
\underline{$q^{\frac{k-n}{2}}<\FNorm{r_2} < q^{k-n}$:}
In this case , $u\bk{r_1,r_2,r_3,r_4,r_5}\in \widehat{U_k\bk{g}}$ if and only if
\begin{align*}
&\FNorm{x} \leq q \\
&\FNorm{r_1} \leq q^{k-n-1} \\
&\FNorm{r_3}, \FNorm{z}, \FNorm{r_1\bk{N_{\bk{a,b,c}}r_1-D_{\bk{a,b,c}}r_2}-r_2r_3} \leq q^{k-n} \\
&\FNorm{r_1r_3-r_2^2}, \FNorm{r_2\bk{N_{\bk{a,b,c}}r_1-D_{\bk{a,b,c}}r_2}-r_3^2} \leq q^{k-n} .
\end{align*}


\item
\underline{$\FNorm{r_2}=q^{k-n}$:}
In this case , $u\bk{r_1,r_2,r_3,r_4,r_5}\in \widehat{U_k\bk{g}}$ if and only if
\begin{align*}
&\FNorm{x} \leq q \\
&\FNorm{r_1}, \FNorm{r_2}, \FNorm{r_3}, \FNorm{z}, \FNorm{r_1\bk{x-D_{\bk{a,b,c}}r_2+N_{\bk{a,b,c}}r_1}-r_2r_3} \leq q^{k-n} \\
&\FNorm{r_1r_3-r_2^2}, \FNorm{r_2\bk{x-D_{\bk{a,b,c}}r_2+N_{\bk{a,b,c}}r_1}-r_3^2} \leq q^{k-n} .
\end{align*}
First note that $\FNorm{r_1}=\FNorm{r_2}=\FNorm{r_3}=q^{k-n}$, hence we may write $\epsilon=\frac{r_2}{r_1}$ with $\FNorm{\epsilon}=1$.
Consider $\FNorm{r_1\bk{x-D_{\bk{a,b,c}}r_2+N_{\bk{a,b,c}}r_1}-r_2r_3} \leq q^{k-n}$, multiplying by $\epsilon^3$ and dividing by $r_2^2$ we have 
\[
\FNorm{\bk{\epsilon^3+D_{\bk{a,b,c}}\epsilon-N_{\bk{a,b,c}}}-\frac{\epsilon x}{r_2}} \leq \frac{1}{q^{k-n}}<1 .
\]
Since $E\rmod F$ is unramified and $\epsilon\in\mO$, it holds that $\FNorm{\epsilon^3+D_{\bk{a,b,c}}\epsilon-N_{\bk{a,b,c}}}=1$ and hence also $\FNorm{\frac{\epsilon x}{r_2}}=1$ contradicting the fact that $\FNorm{\frac{\epsilon x}{r_2}}\leq \frac{q}{q^{k-n}}<1$.

\end{itemize}

Let $l=\lfloor \frac{k-n}{2} \rfloor$.
Also, let $C\subset F^3$ be the set such that $\bk{r_1,r_2,r_3}\in C$ if and only if
\begin{align*}
&q^{\frac{k-n}{2}} < \FNorm{r_2} \leq q^{k-n-1} \\
&\FNorm{r_1} \leq q^{k-n-1} \\
&\FNorm{r_3}, \FNorm{r_1\bk{N_{\bk{a,b,c}}r_1-D_{\bk{a,b,c}}r_2}-r_2r_3} \leq q^{k-n} \\
&\FNorm{r_1r_3-r_2^2}, \FNorm{r_2\bk{N_{\bk{a,b,c}}r_1-D_{\bk{a,b,c}}r_2}-r_3^2} \leq q^{k-n} .
\end{align*}
It then holds that
\begin{align*}
E_k^{\Psi_E}\bk{g} = q^{k-n}\bk{q^{l+k-n}+\meas\bk{C}} \intl_{\FNorm{x}\leq q} \psi\bk{x}dx = 0 .
\end{align*}
\end{itemize}

\paragraph{\underline{Case of $m>2n$}:}
This was already treated in the discussion of the case $N_{\bk{a,b,c}}=0$.

\subsection{Non-Toral Elements}
In this section we consider the case of $g\notin UTK$.
Since $E_k^{\Psi_E}\in\M_{\Psi_E}$ it is sufficient to consider $g=h_\alpha(t_1)h_\beta(t_2)x_\alpha\bk{d}=x_\alpha\bk{p} h_\alpha(t_1)h_\beta(t_2)$, where $p=\frac{t_1^2 d}{t_2}$.
Let $\FNorm{t_1}=q^{-n}$, $\FNorm{t_2}=q^{-m}$ and $\FNorm{p}=q^l$.

\begin{Remark}
\begin{itemize}
\item
Let $E=F\times K$ and $\eta = w_\alpha h_\beta \bk{bc} \in G\bk{\mO}$.
Recall that according to {\bf (CT)} we assume that $a=0$ and $b,c\neq 0$.
Since $E_k\in \Hecke$, it holds that
\begin{align*}
E_k^{\Psi_{\bk{0,b,c}}}\bk{g} &= \intl_{U\bk{F}} E_k\bk{ug} \bk{\Psi_{\bk{0,b,c}}\bk{u}} du = \intl_{U\bk{F}} E_k\bk{\eta^{-1} ug \eta} \bk{\Psi_{\bk{0,b,c}}\bk{u}} du = \\
&=\intl_{U\bk{F}} E_k\bk{u'\eta^{-1} g\eta} \bk{\Psi_{\bk{0,b,c}}\bk{\eta u' \eta^{-1}}} du' =
\intl_{U\bk{F}} E_k\bk{u'g'} \bk{\Psi_{\bk{0,\frac{1}{b},\frac{1}{c}}}}\bk{u'} du' = E_k^{\Psi_{\bk{0,\frac{1}{b},\frac{1}{c}}}}\bk{g'},
\end{align*}
where $g'= x_{\alpha}\bk{p'} h_\alpha(\frac{bc t_2}{d t_1})h_\beta(t_2)$ and $p'=bc\frac{t_2}{dt_1^2}$.

\item
Let $E$ be a cubic field extension and $\eta = w_\alpha h_\beta \bk{abc} \in G\bk{\mO}$.
Recall that according to {\bf (CT)} we assume that $a,b,c\neq 0$.
As above, it holds that
\begin{align*}
E_k^{\Psi_E}\bk{g} &= E_k^{\Psi_{\bk{\frac{1}{a},\frac{1}{b},\frac{1}{c}}}}\bk{g'},
\end{align*}
where $g'= x_{\alpha}\bk{p'} h_\alpha(\frac{abc t_2}{d t_1})h_\beta(t_2)$ and $p'=abc \frac{t_2}{dt_1^2}$.
\end{itemize}

In both cases $\FNorm{p'}=\frac{1}{\FNorm{p}}$ and hence we may assume from now on that $\FNorm{p}\leq 1$ since $\bk{a,b,c}$ is arbitrary.
Therefore we assume for the rest of this section that $\FNorm{p}\leq 1$.
Also note that under this assumption, it follows that $\FNorm{\frac{t_1^2}{t_2}}<1$.

\end{Remark}



\begin{Remark}
From \cref{Lem: Conditions on m}, and under the assumption that $\FNorm{p}\leq 1$, we have $D_s^{\Psi_E}\bk{g}=0$ unless
\[
\frac{t_1^3}{t_2}, t_1, \frac{t_2}{t_1^3} \in \mO .
\]

Furthermore, $U_k\bk{g}=\emptyset$ unless $k\geq n$. We have $u\bk{r_1,r_2,r_3,r_4,r_5}\in U_k\bk{g}$ if and only if
\begin{align*}
&k\geq n,m-n \\
&\FNorm{r_1},\FNorm{r_2},\FNorm{r_3},\FNorm{r_2r_3+r_5},\FNorm{r_2^2-r_1r_3}\leq q^{k+n-m}\\
&\FNorm{pr_1-r_2},\FNorm{pr_2-r_3},\FNorm{pr_3-r_4},\FNorm{r_2r_4-r_3^2-pr_2r_3-pr_5}\leq q^{k-n}\\
&\FNorm{r_1r_4-2r_2r_3+pr_2^2-pr_1r_3-r_5}\leq q^{k-n} .
\end{align*}
\end{Remark}

\begin{Remark}
As an analogue of \cite[Lemma B.2]{MR3284482}, note that
\[
\intl_{U_k\bk{g}} \Psi_E\bk{u} du = \intl_{\widehat{U_k\bk{g}}} \Psi_E\bk{u} du,
\]
where
\[
\widehat{U_k\bk{g}} = 
\set{u\bk{r_1,r_2,r_3,r_4,r_5}\in U_k\bk{g} \mvert \FNorm{r_4+D_{\bk{a,b,c}}r_2-N_{\bk{a,b,c}}r_1}\leq q} .
\]
\end{Remark}

We now split the computation to two cases, $N_{\bk{a,b,c}}=0$ or $N_{\bk{a,b,c}}\in\mO^\times$.

\subsubsection{$N_{\bk{a,b,c}}=0$}
We write
\begin{align*}
x&=r_4+D_{\bk{a,b,c}}r_2 \\
z&=r_2r_3+r_5 .
\end{align*}
In this case, $u\bk{r_1,r_2,r_3,r_4,r_5}\in \widehat{U_k\bk{g}}$ if and only if
\begin{align*}
&\FNorm{x}\leq q \\
&\FNorm{r_1},\FNorm{r_2},\FNorm{r_3},\FNorm{z},\FNorm{r_2^2-r_1r_3}\leq q^{k+n-m}\\
&\FNorm{pr_1-r_2},\FNorm{pr_2-r_3},\FNorm{pr_3-x+D_{\bk{a,b,c}}r_2},\FNorm{r_2x-D_{\bk{a,b,c}}r_2^2-r_3^2-pz}\leq q^{k-n}\\
&\FNorm{r_1x-D_{\bk{a,b,c}}r_1r_2-r_2r_3+pr_2^2-pr_1r_3-z}\leq q^{k-n} .
\end{align*}

Let $\kappa=p^2+D_{\bk{a,b,c}}$.

\paragraph{\underline{Case of $l<0$}:}

\begin{itemize}
\item
\underline{$k=n$:}
In this case , $u\bk{r_1,r_2,r_3,r_4,r_5}\in \widehat{U_k\bk{g}}$ if and only if
\begin{align*}
&\FNorm{x}\leq q \\
&\FNorm{r_1},\FNorm{r_2},\FNorm{r_3},\FNorm{z},\FNorm{r_2^2-r_1r_3}\leq q^{2n-m}\\
&\FNorm{pr_1-r_2},\FNorm{pr_2-r_3},\FNorm{pr_3-x+D_{\bk{a,b,c}}r_2},\FNorm{r_2x-D_{\bk{a,b,c}}r_2^2-r_3^2-pz}\leq 1\\
&\FNorm{r_1x-D_{\bk{a,b,c}}r_1r_2-r_2r_3+pr_2^2-pr_1r_3-z}\leq 1 .
\end{align*}

\begin{itemize}
\item
\underline{$\FNorm{x}\leq 1$:}
We make a change of variables
\[
y=r_1x-D_{\bk{a,b,c}}r_1r_2-z
\]
and then $u\bk{r_1,r_2,r_3,r_4,r_5}\in \widehat{U_k\bk{g}}$ if and only if
\begin{align*}
&\FNorm{r_1}\leq q^{-l}\\
&\FNorm{r_2},\FNorm{r_3},\FNorm{y}\leq 1 .
\end{align*}

\item
\underline{$\FNorm{x}=q$:}
We make a change of variables
\begin{align*}
& y=r_1x-D_{\bk{a,b,c}}r_1r_2-r_2r_3+pr_2^2-pr_1r_3-z \\
& l_1=pr_1-r_2 \\
& l_2=D_{\bk{a,b,c}}r_2-x
\end{align*}
and then $u\bk{r_1,r_2,r_3,r_4,r_5}\in \widehat{U_k\bk{g}}$ if and only if
\begin{align*}
&\FNorm{l_1},\FNorm{l_2},\FNorm{r_3},\FNorm{y}\leq 1 .
\end{align*}
\end{itemize}

In conclusion
\[
E_n^{\Psi_E}\bk{g} = q^{-l} \intl_{\FNorm{x}\leq 1} \psi\bk{x} dx +q^{-l} \intl_{\FNorm{x}=q} \psi\bk{x} dx = 0.
\]

\item
\underline{$k>n$:}
In this case , $u\bk{r_1,r_2,r_3,r_4,r_5}\in \widehat{U_k\bk{g}}$ if and only if
\begin{align*}
&\FNorm{x}\leq q \\
&\FNorm{z},\FNorm{r_2^2-r_1r_3}\leq q^{k+n-m}\\
&\FNorm{r_1},\FNorm{r_2},\FNorm{r_3},\FNorm{r_2x-D_{\bk{a,b,c}}r_2^2-r_3^2-pz}\leq q^{k-n}\\
&\FNorm{r_1x-D_{\bk{a,b,c}}r_1r_2-r_2r_3+pr_2^2-pr_1r_3-z}\leq q^{k-n} .
\end{align*}

We make a change of variables
\begin{align*}
& y=r_1x-D_{\bk{a,b,c}}r_1r_2-r_2r_3+pr_2^2-pr_1r_3-z \\
& l_1=r_2-pr_1 \\
& l_2=pr_2-r_3 \\
& l_3=D_{\bk{a,b,c}}r_2+pr_3-x
\end{align*}
and then $u\bk{r_1,r_2,r_3,r_4,r_5}\in U_k\bk{g}$ if and only if
\begin{align*}
&\FNorm{x}\leq q \\
&\FNorm{p\bk{pl_2+l_3}\cdot l_2 + \bk{\kappa l_1+pl_2+l_3} l_3}\leq q^{k+n-m+l}\\
&\FNorm{p\bk{pl_2+l_3}^2-\bk{\kappa l_1+pl_2+l_3}\bk{pl_3-D_{\bk{a,b,c}}l_2}}\leq q^{k+n-m+l}\\
&\FNorm{l_1}, \FNorm{p l_2+l_3}, \FNorm{p l_3-D_{\bk{a,b,c}}l_2}, \FNorm{y}, \FNorm{l_1 l_3+l_2^2}\leq q^{k-n} .
\end{align*}

Let $C\subseteq F^3$ be the subset such that $\bk{l_1,l_2,l_3}\in C$ if and only if
\begin{align*}
&\FNorm{p\bk{p l_2+l_3}\cdot l_2 + \bk{\kappa l_1+p l_2+l_3} l_3}\leq q^{k+n-m+l}\\
&\FNorm{p\bk{p l_2+l_3}^2-\bk{\kappa l_1+p l_2+l_3}\bk{p l_3-D_{\bk{a,b,c}}l_2}}\leq q^{k+n-m+l}\\
&\FNorm{l_1}, \FNorm{p l_2+l_3}, \FNorm{p l_3-D_{\bk{a,b,c}} l_2}, \FNorm{l_1l_3+l_2^2}\leq q^{k-n} .
\end{align*}
Hence
\[
E_k^{\Psi_E}\bk{g} = q^{k-n} \meas\bk{C} \intl_{\FNorm{x}\leq q} \psi\bk{x} dx = 0 .
\]
\end{itemize}

\paragraph{\underline{Case of $l=0$}:}

\begin{itemize}
\item
\underline{$k=n$:}
We make a change of variables
\begin{align*}
y&=r_1x-D_{\bk{a,b,c}}r_1r_2-r_2r_3+pr_2^2-pr_1r_3-z \\
l_1&=p r_1-r_2 \\
l_2&=\kappa r_2-x \\
l_3&=r_3-p r_2 .
\end{align*}
In this case , $u\bk{r_1,r_2,r_3,r_4,r_5}\in \widehat{U_k\bk{g}}$ if and only if
\begin{align*}
&\FNorm{x}\leq q \\
&\FNorm{y}, \FNorm{l_1},\FNorm{l_2},\FNorm{l_3}\leq 1
\end{align*}
and hence
\[
E_n\bk{g}= \intl_{\FNorm{x}\leq q} \psi\bk{x} dx = 0 .
\]

\item
\underline{$k>n$:}
We make a change of variables
\begin{align*}
& y=r_1x-D_{\bk{a,b,c}}r_1r_2-r_2r_3+pr_2^2-pr_1r_3-z \\
& l_1=r_2-pr_1 \\
& l_2=pr_2-r_3 \\
& l_3=D_{\bk{a,b,c}}r_2+p r_3-x .
\end{align*}
In this case , $u\bk{r_1,r_2,r_3,r_4,r_5}\in \widehat{U_k\bk{g}}$ if and only if
\begin{align*}
&\FNorm{x}\leq q \\
&\FNorm{\frac{pl_2+l_3}{\kappa}\cdot l_2 + \bk{\frac{pl_2+l_3}{p\kappa}-\frac{l_1}{p}}\cdot l_3}\leq q^{k+n-m}\\
&\FNorm{\bk{\frac{p l_2+l_3}{\kappa}}^2-\bk{\frac{p l_2+l_3}{p\kappa}-\frac{l_1}{p}}\cdot \frac{p l_3-D_{\bk{a,b,c}}l_2}{\kappa}}\leq q^{k+n-m} \\
&\FNorm{\frac{p l_2+l_3}{p\kappa}-\frac{l_1}{p}},
\FNorm{\frac{p l_2+l_3}{\kappa}},
\FNorm{\frac{p l_3-D_{\bk{a,b,c}}l_2}{\kappa}}, \FNorm{l_1l_3+l_2^2},
\FNorm{y}\leq q^{k-n} .
\end{align*}

Let $C\subseteq F^3$ be the subset such that $\bk{l_1,l_2,l_3}\in C$ if and only if
\begin{align*}
&\FNorm{\frac{p l_2+l_3}{\kappa}\cdot l_2 + \bk{\frac{p l_2+l_3}{p \kappa}-\frac{l_1}{p }}\cdot l_3}\leq q^{k+n-m}\\
&\FNorm{\bk{\frac{p l_2+l_3}{\kappa}}^2-\bk{\frac{p l_2+l_3}{p\kappa}-\frac{l_1}{p}}\cdot \frac{p l_3-D_{\bk{a,b,c}}l_2}{\kappa}}\leq q^{k+n-m} \\
&\FNorm{\frac{pl_2+l_3}{p\kappa}-\frac{l_1}{p}},
\FNorm{\frac{p l_2+l_3}{\kappa}},
\FNorm{\frac{p l_3-D_{\bk{a,b,c}}l_2}{\kappa}}, \FNorm{l_1l_3+l_2^2}\leq q^{k-n} .
\end{align*}

Hence
\[
E_k^{\Psi_E}\bk{g} = q^{k-n} \meas\bk{C} \intl_{\FNorm{x}\leq q} \psi\bk{x} dx = 0 .
\]

\end{itemize}

\subsubsection{$N_{\bk{a,b,c}}\in\mO^\times$}
We write
\begin{align*}
x&=r_4+D_{\bk{a,b,c}}r_2-N_{\bk{a,b,c}}r_1 \\
z&=r_1r_4-2r_2r_3-r_5 .
\end{align*}


For any $k\geq n$, $u\bk{r_1,r_2,r_3,r_4,r_5}\in \widehat{U_k\bk{g}}$ if and only if
\begin{align*}
&\FNorm{x}\leq q \\
&\FNorm{r_1},\FNorm{r_2},\FNorm{r_3},\FNorm{z},\FNorm{r_2^2-r_1r_3}\leq q^{k+n-m}\\
&\FNorm{p r_1-r_2}, \FNorm{p r_2-r_3}, \FNorm{p r_3-\bk{x+N_{\bk{a,b,c}}r_1-D_{\bk{a,b,c}}r_2}}, \FNorm{r_2\bk{x+N_{\bk{a,b,c}}r_1-D_{\bk{a,b,c}}r_2}-r_3^2-pz}\leq q^{k-n}\\
&\FNorm{r_1\bk{x+N_{\bk{a,b,c}}r_1-D_{\bk{a,b,c}}r_2}-r_2r_3+p r_2^2-p r_1r_3-z}\leq q^{k-n} .
\end{align*}
Let $\kappa=p^3+D_{\bk{a,b,c}}p-N_{\bk{a,b,c}}\in \mO^\times$.

\begin{itemize}
\item
\underline{$k=n$:}
We note that
\[
\kappa r_1-x = p^2\bk{pr_1-r_2}+p\bk{pr_2-r_3}+\bk{pr_3-\bk{x+N_{\bk{a,b,c}}r_1-D_{\bk{a,b,c}}r_2}}
\]
Which means that $\FNorm{\kappa r_1-x}\leq 1$ and hence $\FNorm{r_1},\FNorm{r_2},\FNorm{r_3}\leq q$.
We make a change of variables
\begin{align*}
& y=r_1\bk{x+N_{\bk{a,b,c}}r_1-D_{\bk{a,b,c}}r_2-pr_3}-r_2\bk{pr_2-r_3}-z \\
& l_1=\kappa r_1-x \\
& l_2=r_2-pr_1 \\
& l_3=r_3-pr_2 .
\end{align*}
One checks that $u\bk{r_1,r_2,r_3,r_4,r_5}\in \widehat{U_k\bk{g}}$ if and only if
\begin{align*}
&\FNorm{x}\leq q \\
&\FNorm{\kappa^2l_2^2+p\kappa l_1l_2-\kappa l_1l_3}\leq q^{2n-m} \\
&\FNorm{l_1},\FNorm{l_2},\FNorm{l_3},\FNorm{y} \leq 1 .
\end{align*}

Let $C\subseteq F^3$ be the subset such that $\bk{l_1,l_2,l_3}\in C$ if and only if
\begin{align*}
&\FNorm{\kappa^2l_2^2+p\kappa l_1l_2-\kappa l_1l_3}\leq q^{2n-m} \\
&\FNorm{l_1},\FNorm{l_2},\FNorm{l_3} \leq 1 .
\end{align*}

Hence
\[
E_n^{\Psi_E}\bk{g} = \meas\bk{C} \intl_{\FNorm{x}\leq q} \psi\bk{x} dx = 0 .
\]

\item
\underline{$k>n$:}
We note that
\[
\kappa r_1-x = p^2\bk{pr_1-r_2}+p\bk{pr_2-r_3}+\bk{pr_3-\bk{x+N_E r_1-D_E r_2}}
\]
Which means that $\FNorm{\kappa r_1-x}\leq q^{k-n}$ and hence $\FNorm{r_1},\FNorm{r_2},\FNorm{r_3}\leq q^{k-n}$.
Moreover $\FNorm{br_3-\bk{x+N_{\bk{a,b,c}}r_1-D_{\bk{a,b,c}}r_2}}\leq q^{k-n}$ if and only if $\FNorm{\kappa r_1-x}\leq q^{k-n}$.


We make a change of variables
\begin{align*}
& y=r_1\bk{x-N_{\bk{a,b,c}}r_1-D_{\bk{a,b,c}}r_2-pr_3}-r_2\bk{p r_2-r_3}-z \\
& l_1=\kappa r_1-x \\
& l_2=r_2-p r_1 \\
& l_3=r_3-p r_2 .
\end{align*}
One checks that $u\bk{r_1,r_2,r_3,r_4,r_5}\in \widehat{U_k\bk{g}}$ if and only if
\begin{align*}
&\FNorm{x}\leq q \\
&\FNorm{\bk{\kappa l_2+pl_1}^2-l_1\bk{\kappa l_3+p\kappa l_2+p^2l_1}}, \FNorm{l_1\bk{l_1+\bk{p^2+D_{\bk{a,b,c}}}l_2+p l_3}+\bk{\kappa l_2+p l_1}l_3} \leq q^{k+n-m} \\
&\FNorm{l_1},\FNorm{l_2},\FNorm{l_3},\FNorm{y}, \FNorm{l_1\bk{l_1+\bk{p^2+D_{\bk{a,b,c}}}l_2+p l_3}-l_3^2}\leq q^{k-n} .
\end{align*}
Let $C\subseteq F^3$ be the subset such that $\bk{l_1,l_2,l_3}\in C$ if and only if
\begin{align*}
&\FNorm{l_1\bk{l_1+\bk{p^2+D_{\bk{a,b,c}}}l_2+p l_3}+\bk{\kappa l_2+p l_1}l_3} \leq q^{k+n-m} \\
&\FNorm{l_1},\FNorm{l_2},\FNorm{l_3}, \FNorm{l_1\bk{l_1+\bk{p^2+D_{\bk{a,b,c}}}l_2+p l_3}-l_3^2}\leq q^{k-n} .
\end{align*}

Hence
\[
E_k^{\Psi_E}\bk{g} = q^{k-n} \meas\bk{C} \intl_{\FNorm{x}\leq q} \psi\bk{x} dx = 0 .
\]
\end{itemize}

%
%
%
%
%
%
%

%
%
%
%


\bibliographystyle{abbrv}
\bibliography{bib}
\end{document}